\documentclass[11pt]{amsart}
\usepackage[colorlinks,citecolor=red,pagebackref,hypertexnames=false]{hyperref}
\usepackage{soul}

\makeatletter
\def\@tocline#1#2#3#4#5#6#7{\relax
  \ifnum #1>\c@tocdepth 
  \else
    \par \addpenalty\@secpenalty\addvspace{#2}%
    \begingroup \hyphenpenalty\@M
    \@ifempty{#4}{%
      \@tempdima\csname r@tocindent\number#1\endcsname\relax
    }{%
      \@tempdima#4\relax
    }%
    \parindent\z@ \leftskip#3\relax \advance\leftskip\@tempdima\relax
    \rightskip\@pnumwidth plus4em \parfillskip-\@pnumwidth
    #5\leavevmode\hskip-\@tempdima
      \ifcase #1
       \or\or \hskip 1em \or \hskip 2em \else \hskip 3em \fi%
      #6\nobreak\relax
    \dotfill\hbox to\@pnumwidth{\@tocpagenum{#7}}\par
    \nobreak
    \endgroup
  \fi}
\makeatother

\usepackage{tikz,amsthm,amsmath,amstext,amssymb,amscd,epsfig,euscript, mathrsfs, dsfont,pspicture,multicol, mathtools,graphpap,graphics,graphicx,times,enumerate,subfig,sidecap,wrapfig,color,bm,comment,todonotes}
\mathtoolsset{showonlyrefs}
\usepackage{ hyperref}
\usepackage{enumerate}
\usepackage{esint}

 \numberwithin{equation}{section}

\def\N{{\mathbb{N}}}

\def\cH{{\mathcal{H}}}
\def\cF{{\mathcal{F}}}

\def\cB{{\mathscr{B}}}
\def\cC{{\mathscr{C}}}
\def\cD{{\mathscr{D}}}

\def\cF{{\mathscr{F}}}
\def\cG{{\mathscr{G}}}

\def\cL{{\mathscr{L}}}
\def\cM{{\mathscr{M}}}
\def\cN{{\mathscr{N}}}

\def\cP{{\mathscr{P}}}
\def\cQ{{\mathscr{Q}}}

\def\cT{{\mathscr{T}}}

\def\ve{\varepsilon}

\def\Rn1{\mathbb{R}^{n+1}}

\def\rn{\mathbb{R}^{n}}
 
\DeclareMathOperator{\diam}{diam}

\newcounter{eps}

\def\Lip{\mathop\mathrm{Lip}} 						

\def\dist{\textup{dist}} 						
\def\supp{\mathop\mathrm{supp}}					

\def\Xint#1{\mathchoice
{\XXint\displaystyle\textstyle{#1}}%
{\XXint\textstyle\scriptstyle{#1}}%
{\XXint\scriptstyle\scriptscriptstyle{#1}}%
{\XXint\scriptscriptstyle\scriptscriptstyle{#1}}%
\!\int}
\def\XXint#1#2#3{{\setbox0=\hbox{$#1{#2#3}{\int}$ }
\vcenter{\hbox{$#2#3$ }}\kern-.58\wd0}}

\def\avint{\Xint-}
\theoremstyle{plain}
\newtheorem{theorem}{Theorem}

\newtheorem{corollary}[theorem]{Corollary}

\newtheorem{lemma}[theorem]{Lemma}

\newtheorem{proposition}[theorem]{Proposition}

\theoremstyle{definition}

\newtheorem{definition}[theorem]{Definition}
\newtheorem{remark}[theorem]{Remark}

\numberwithin{equation}{section}
\numberwithin{theorem}{section}

\newcommand\mg{\sigma}
\newcommand\mgs{\sigma_s}
\newcommand\mgf{\sigma_f}

  \DeclareFontFamily{U}{mathb}{\hyphenchar\font45} 
\DeclareFontShape{U}{mathb}{m}{n}{
      <5> <6> <7> <8> <9> <10> gen * mathb
      <10.95> mathb10 <12> <14.4> <17.28> <20.74> <24.88> mathb12
      }{}
\DeclareSymbolFont{mathb}{U}{mathb}{m}{n}

\DeclareMathSymbol{\toitself}      {3}{mathb}{"FD}  

\def\R{\mathbb{R}}

\newcommand{\BWGL}{\text{BWGL}}

\newcommand{\Top}{\mathrm{Top}}

\newcommand{\Child}{\mathrm{Child}}

\newcommand{\Parent}{\mathrm{Parent}}
\newcommand{\Sibling}{\mathrm{Sibling}}
\newcommand{\bilatbeta}{b\beta}
\newcommand{\unilatbeta}{\beta}
\newcommand{\baupbeta}{Ub\beta}
\newcommand{\cubes}{\cD}
\newcommand{\pdist}{\dist_p}
\newcommand{\pd}{d_p}
\newcommand{\pdiam}{\diam_p}
\newcommand{\verticalplanes}{\cP}
\newcommand{\Hmeasure}{\cH}
\newcommand{\pHmeasure}{\Hmeasure_p}
\newcommand{\surfacemeasure}{\sigma}
\newcommand{\slicemeasure}{\mu}

\newcommand{\adim}{d}
\newcommand{\idim}{n}

\newcommand{\CAR}{C_0}
\newcommand{\Cinner}{c_1}
\newcommand{\Couter}{C_1}
\newcommand{\CCCBP}{C_2}
\newcommand{\Cproj}{c_2}
\newcommand{\CSchulprime}{c_2}
\newcommand{\Cnext}{C_3}
\newcommand{\Cnextlower}{c_3}
\newcommand{\Cnextprime}{C_4}
\newcommand{\Cnextlowerprime}{c_4}
\newcommand{\Cball}{C_5}
\newcommand{\Cind}{c_6}
\newcommand{\Cmintwo}{C_6}
\newcommand{\Cproof}{c_7}

\newcommand{\Nflat}{N_0}
\newcommand{\Nnonflat}{N_1}
\newcommand{\veflat}{\ve_0}
\newcommand{\venonflat}{\ve_1}

\begin{document}

\title[]{{Parabolic Reifenberg Parametrizations and Corona Decompositions, with Applications to Singular Integrals}}
\author{John Hoffman}
\address{Department of Mathematics, University of the Basque Country (EHU), 48940 Leioa, Spain}
\email{johnloche.hoffman@ehu.eus}
\thanks{J.H. was supported in part by grant EUR2025-165042, funded by MICIU/AEI/10.13039/501100011033, and by grant PID2024-157724NB-I00, funded by MICIU/AEI/10.13039/501100011033 and by ERDF/EU}

\author{Matthew Hyde}
\address{Matthew Hyde\\
	Department of Mathematics and Statistics \\
	University of Jyv\"askyl\"a \\
	P.O. Box 35 (MaD) \\
	FI-40014 University of Jyv\"askyl\"a \\
	Finland}
\email{matthew.j.hyde@jyu.fi}
\thanks{M.H. was supported by the Research Council of Finland via the project Quantitative differentiability and rectifiability in metric spaces, grant no.~363800.}

\author[Mihalis Mourgoglou]{Mihalis Mourgoglou}
\address{Departamento de Matem\'aticas, Universidad del Pa\' is Vasco, Barrio Sarriena s/n 48940 Leioa, Spain and\\
Ikerbasque, Basque Foundation for Science, Bilbao, Spain.}
\email{michail.mourgoglou@ehu.eus}
\thanks{M.M. was supported by IKERBASQUE and partially supported by project PID2024-157724NB-I00, funded by MICIU/AEI/10.13039/501100011033/FEDER, EU, and by grants IT-1615-22 and IT1875-26 of the Basque Government.}


\newcommand{\mih}[1]{\marginpar{\color{red} \scriptsize \textbf{Mih:} #1}}
\newcommand{\matt}[1]{\marginpar{\color{blue} \scriptsize \textbf{Matt:} #1}}
\newcommand{\john}[1]{\marginpar{\color{red} \scriptsize \textbf{John:} #1}}

\begin{abstract}

For parabolic Ahlfors regular sets, we show that the bilateral weak geometric lemma (BWGL), bilateral approximation by unions of planes (BAUP), and corona decompositions by $\operatorname{Lip}(1,1/2)$ graphs (whose defining functions need not have half-order time derivatives in BMO) are equivalent. These geometric conditions do not imply the strong geometric lemma. The proof that BAUP implies the BWGL is direct and applies in Euclidean space. A key tool is a parabolic version of Reifenberg's theorem for sets with holes, which we prove here. As an application of these geometric results, we prove that a parabolic Ahlfors $(n+2)$-regular measure $\mu$ on $\mathbb{R}^{d+1}$, $1\leq n\leq d-1$, is parabolic uniformly rectifiable if its support satisfies any of these equivalent conditions and the parabolic Riesz transform is bounded on $L^2(\mu)$.
\end{abstract}

\maketitle

\tableofcontents
\section{Introduction}

	With a focus on the purely geometric aspects of subsets of parabolic space, Mattila \cite{mattila2022parabolic} introduced the notion of parabolic LG-rectifiability and gave a number of equivalent conditions based on the classical theory of Euclidean rectifiability, in particular in terms of approximate tangents and tangent measures. Roughly speaking, a set is \textit{parabolic LG-rectifiable} (of dimension $n$) if it can be covered $\pHmeasure^n$-a.e. by countably many $n$-dimensional Lipschitz graphs with arbitrarily small Lipschitz constants. As part of a broader work to extend Preiss' rectifiability criterion to the parabolic setting, a parabolic version of the \textit{Marstrand-Mattila rectifiability criterion} \cite{marstrand1961hausdorff,mattila1975hausdorff} was established in \cite{merlo2022on}. Under the assumption of positive lower density, this condition asks that for almost every point and all small balls centred at such points, there is a vertical plane which bilaterally well-approximates the set. We highlight that the planes may vary at different scales. \\

	A key point of this paper is to study parabolic \textit{uniform} rectifiability from the purely geometric point of view. We propose that the natural quantitative analogues of LG-rectifiability and the Marstrand-Mattila rectifiability criterion in this setting are \textit{corona decompositions} and the \textit{bilateral weak geometric lemma} (BWGL), respectively, both of which were introduced by David and Semmes in the Euclidean setting \cite{david1993analysis}. The precise definitions are given in Section \ref{s:prelims}. For now, let us give a more informal overview.

Roughly speaking, an Ahlfors regular set admits a corona decomposition if, for every $\eta>0$ and every prescribed approximation error, its dyadic cubes can be partitioned into bad cubes and coherent stopping-time regions. At the locations and scales of each region, the set is well approximated by a parabolic Lipschitz graph with Lipschitz constant at most $\eta$. The bad cubes and the top cubes of the regions satisfy Carleson packing conditions. These conditions imply that the set is covered almost everywhere by the corresponding Lipschitz graphs. Since the decomposition is required for every $\eta>0$, the set is LG-rectifiable in the sense of Mattila. A bilateral corona decomposition also requires the Lipschitz graphs to be well approximated by the set at the same locations and scales.
	
	The BWGL asks, for any given $\ve > 0$, that in \textit{most} balls the set is $\ve$-close (bilaterally) to a \textit{vertical plane} of spatial dimension $n$. As above, $\textit{most}$ is quantified by a Carleson-type condition. Thus, the BWGL requires coarse restrictions on the geometry of the set, not just infinitesimal. Closely related to the BWGL is the \textit{bilateral approximation by unions of planes} (BAUP). It is a  weakening of the BWGL where one asks that, for most balls, the set is bilaterally $\ve$-close to a union of planes (rather than a single plane). \\
	
	Our first main result establishes the equivalence between the above conditions.

	\begin{theorem}\label{t:characterisation}
		Let $n,d$ be integers with $1\leq n\leq d-1$, and let $E \subseteq \R^{d+1}$ be parabolic Ahlfors $(n+2)$-regular. The following are equivalent. 
		\begin{enumerate}
			\item \label{i:CD} $E$ admits a corona decomposition by $\Lip(1,1/2)$-graphs.
			\item \label{i:BCD} $E$ admits a bi-lateral corona decomposition by $\Lip(1,1/2)$-graphs.   
			\item \label{i:BWGL} $E$ satisfies the \emph{BWGL}.
			\item \label{i:BAUP} $E$ satisfies the \emph{BAUP}. 
		\end{enumerate}
	\end{theorem}

In Euclidean space, the conditions in Theorem \ref{t:characterisation} are equivalent to uniform rectifiability.  In the parabolic setting, the story is more nuanced. Hofmann, Lewis and Nystr\"om introduced parabolic uniform rectifiability in their study of caloric measure and boundary value problems for the heat equation in time-varying domains \cite{hofmann2003existence,hofmann2004caloric}. Their definition is given by the strong geometric lemma.\footnote{{That is, a Carleson measure estimate for the squared parabolic $\beta_2$-numbers, where the approximating planes are parallel to the time axis.}} For parabolic Lipschitz graphs, this condition is equivalent to requiring that the half-order time derivative of the defining function belong to parabolic BMO \cite[Section 2]{hofmann2004caloric}. Such graphs are called regular parabolic Lipschitz graphs. The additional condition on the time derivative is necessary because a parabolic Lipschitz graph need not be parabolic uniformly rectifiable \cite{bortz2022coronizations,lewis1988parabolic,kaufman1980singularity}. In particular, the equivalent conditions in Theorem \ref{t:characterisation} do not imply the strong geometric lemma \cite{bortz2022coronizations}, nor do they imply $L^2$-boundedness of the corresponding parabolic singular integral operators \cite{hoffmanjaye2026}.

Hofmann proved that a parabolic Lipschitz function is regular if and only if the corresponding parabolic commutator is bounded on $L^2$ \cite{hofmann1995characterization}. Bortz, Hoffman, Hofmann, Luna Garc\'ia and Nystr\"om subsequently characterized parabolic uniform rectifiability in codimension one by bilateral corona decompositions using regular parabolic Lipschitz graphs \cite{bortz2022coronizations,bortz2023corona}, and proved boundedness of parabolic singular integrals on such sets \cite[Corollary 4.9]{bortz2022coronizations}. The first author and Jaye proved a converse in every integer codimension under the higher spatial dimension condition (HSDC) \cite{hoffmanjaye2026}. This condition rules out closeness to a canonical family of counterexamples, except on a set of scales and locations satisfying a ``Carleson set'' condition. These counterexamples are measures for which all admissible parabolic singular integrals are bounded, but which are not parabolic uniformly rectifiable \cite[Section 7]{hoffmanjaye2026}. More recently, the first author proved that if an Ahlfors regular measure $\mu$ has support equal to a parabolic Lipschitz graph with sufficiently small Lipschitz constant, then $L^2(\mu)$ boundedness of the parabolic Riesz transform implies parabolic uniform rectifiability \cite{hoffman2026riesz}. We use this result in the proof of Theorem \ref{t:riesztheorem}, after transferring the Riesz transform bounds to graphs with small Lipschitz constant.

There are also results which deduce parabolic uniform rectifiability directly from properties of caloric measure. Bortz, Hofmann, Martell and Nystr\"om proved that a codimension-one parabolic Lipschitz graph is parabolic uniformly rectifiable if caloric measure for the region above the graph satisfies the $A_\infty$ condition with respect to parabolic surface measure \cite[Theorem 1.1]{bortz2025solvabilitygraphs}. They subsequently obtained the same conclusion for boundaries of domains satisfying an interior corkscrew condition, assuming that the boundary is time-symmetric ADR and that caloric measure satisfies weak-$A_\infty$ \cite[Theorem 1.1]{bortz2025solvability}.

	Before moving on to Theorem \ref{t:riesztheorem}, let us make a few remarks regarding Theorem \ref{t:characterisation}. The implications we prove (with the exception of BWGL implies BAUP and bilateral corona decomposition implies corona decomposition, which are obvious) are the following: 
	\begin{align}\label{e:chain}
		\mbox{BAUP} \iff \mbox{BWGL} \implies  \mbox{bi-lateral corona} \implies \mbox{corona} \implies \mbox{BWGL}. 
	\end{align}
	Although the BAUP and the BWGL are known to be equivalent in the Euclidean case \cite{david1993analysis}, the proof of this is rather indirect and passes via several difficult properties related to the theory of uniform rectifiability (WHIP, generalised corona decompositions etc). Since the BAUP is morally very closely related to the BWGL, this is rather unsatisfactory. Here, we provide a more direct proof of this fact which works in both
	the elliptic and parabolic settings. The Euclidean proof that the BWGL implies a corona decomposition passes via a property known as the \textit{strong geometric lemma}. In the parabolic setting, the BWGL does not imply the strong geometric lemma (which is the usual definition of parabolic uniform rectifiability), so we need a different approach. Instead, we adapt and improve upon (in the parabolic setting) the method given in \cite{bate2023uniformly} for constructing corona decompositions from the BWGL. A key tool for this direction is a parabolic version of the Reifenberg parametrisation theorem of David and Toro \cite{david2012reifenberg}. This is a little technical for an introduction, so we leave the statement of this until Section \ref{s:Reif}.  {A parabolic Reifenberg type theorem has also recently been announced in \cite{koirala2026parabolic}. However, as stated, it is not suitable for our purposes.}  Finally, it was shown in \cite{bortz2022coronizations} that a corona decomposition with respect to \textit{regular} parabolic Lipschitz graphs implies the BWGL. The final implication shows that the regularity property of the graphs can be dropped from the hypotheses (as long as you have the decomposition for all Lipschitz constants). 
   \\

As an application of Theorem \ref{t:characterisation}, we make significant progress toward a resolution of the parabolic variant of the David-Semmes problem.  Before describing our results in this direction, we provide a summary of previous results concerning singular integrals and uniform rectifiability in the Euclidean and parabolic setting.  We start with Euclidean results.  In \cite{david1991singular}, the authors proved that, in the presence of the Ahlfors regularity condition, boundedness of a large class of singular integrals implies uniform rectifiability.   In \cite{david1993analysis}, the same authors asked whether $L^2$ boundedness of the Riesz transforms (i.e. for an Ahlfors $n$-regular measure $\mu$ on $\mathbb R^d$, $1 \leq n \leq d-1$,  the singular integrals associated to the kernel $X/|X|^{n+1}$) is sufficient to conclude uniform rectifiability of the underlying measure.  For Ahlfors $1$-regular measures, this was answered affirmatively by Mattila, Melnikov, and Verdera in \cite{mattila1996}.  The case of $(d-1)$-ADR measures on $\mathbb R^{d}$ was answered affirmatively by Nazarov, Tolsa, and Volberg in \cite{nazarov2014uniform}. OpenAI claims a solution to the remaining cases $2\leq n\leq d-2$ of the David--Semmes problem \cite{openai2026riesz}.

We let $\mathcal R = \mathcal R_n$ denote the parabolic Riesz transform associated to the kernel
$$K_n(X,t) = \frac{X}{|t|^{\frac{n+1}{2}+1}} \exp\bigg(-\frac{|X|^2}{4t}\bigg)1_{t>0}.$$
Notice that when $n=d-1$, this kernel is {a nonzero dimensional constant multiple of} the spatial gradient of the fundamental solution of the heat equation on $\mathbb R^{d+1}$. In the context of a $(n+2)$-ADR {measure} $\mu$, {we consider} $\mathcal R_n$ and $K_n$, and often we will suppress the {dimensional subscript} from the notation, only referring to $\mathcal R$ and $K$.

\begin{theorem}[\cite{hoffmanjaye2026}]
    
Suppose $\mu$ is a parabolic $(n+2)$-ADR measure on $\mathbb R^{d+1}$, $0\leq n\leq d-1$, satisfying the $\epsilon$-\textbf{HSDC} for some $\epsilon>0$. If every admissible $(n+2)$-dimensional parabolic Calder\'on--Zygmund singular integral operator, in the sense of \cite[Section~2.5]{hoffmanjaye2026}, is bounded on $L^2(\mu)$, then $\mu$ is parabolic uniformly rectifiable.
\end{theorem}
In Theorem \ref{t:riesztheorem}, we show that the geometric conditions proved to be equivalent in Theorem \ref{t:characterisation}, in conjunction with $L^2$-boundedness of the parabolic Riesz transform, imply parabolic uniform rectifiability. The conditions of Theorem \ref{t:characterisation} are strictly stronger than the HSDC introduced by the first author and Jaye \cite{hoffmanjaye2026}.

\begin{theorem}\label{t:riesztheorem}
    
Suppose $\mu$ is a parabolic $(n+2)$-ADR measure on $\mathbb R^{d+1}$, $1\leq n\leq d-1$, whose support satisfies BAUP. If the parabolic Riesz transform $\mathcal R$ is bounded on $L^2(\mu)$, then $\mu$ is parabolic uniformly rectifiable.
\end{theorem}

The proof uses the recent theorem of the first author \cite[Theorem 1.1]{hoffman2026riesz}, which shows that $L^2$-boundedness of the parabolic Riesz transform implies parabolic uniform rectifiability for parabolic $(n+2)$-ADR measures whose support is a parabolic Lipschitz graph with sufficiently small Lipschitz constant. We apply this theorem after transferring the Riesz transform bounds to the modified graphs associated with the corona decomposition produced by Theorem \ref{t:characterisation}.

To understand the significance of Theorem \ref{t:riesztheorem}, it is important to recall that, in the parabolic setting, BWGL does not imply uniform rectifiability, as is shown in \cite{bortz2022coronizations}.  Because BWGL is stronger than BAUP, BAUP cannot imply parabolic uniform rectifiability either.  This is in contrast to the Euclidean case, where BAUP and BWGL both imply uniform rectifiability. 

\subsection{Structure of the paper} In Section \ref{s:prelims} we gather all  preliminary definitions and results. In Section \ref{s:Reif} we state and prove the parabolic Reifenberg theorem \textit{\`a la} David-Toro. We prove Theorem \ref{t:characterisation} in Section \ref{s:BWGL-corona} and Section \ref{s:BAUP-BWGL}. Indeed, the main bulk of Section \ref{s:BWGL-corona} is dedicated to the proof that \eqref{i:BWGL} implies \eqref{i:BCD}. At the end of Section \ref{s:BWGL-corona}, we give a short proof to show that \eqref{i:CD} implies \eqref{i:BWGL}. In Section \ref{s:BAUP-BWGL}, we show that \eqref{i:BAUP} implies \eqref{i:BWGL}. These, together with the obvious implications (that \eqref{i:BCD} implies \eqref{i:CD}, and \eqref{i:BWGL} implies \eqref{i:BAUP}) give \eqref{e:chain}. Hence, this completes the proof of Theorem \ref{t:characterisation}. We prove Theorem \ref{t:riesztheorem} in Section \ref{s:Riesz}. \\

 \noindent\textbf{Statement on AI usage:} AI tools (ChatGPT) were used in the editing of the manuscript. The proofs and the arguments in the manuscript were provided by the authors.

\section{Preliminaries}\label{s:prelims}

\subsection{Notation and basic definitions}\label{s:notation}

Fix integers $1 \leq \idim \leq \adim-1$ for the remainder remainder of the paper. Denote points in $\R^{\adim+1}$ by $(X,t) \in \R^\adim \times \R$. Let $|\cdot|$ denote the standard Euclidean norm on $\R^\adim$ and define the parabolic norm on $\R^{\adim+1}$ by
\[\|(X,t)\| \coloneqq |X| + |t|^\frac{1}{2} .\]
For $(X,t),(Y,s) \in \R^{\adim+1},$ define
\[ \pd(X,t,Y,s) = \pd((X,t),(Y,s)) = \|(X,t) - (Y,s)\|.\]
For $(X,t) \in \R^{\adim+1}$ and $r > 0$, let 
\[ B_r(X,t) = \{ (Y,s) \in \R^{\adim+1} \colon |X-Y| \leq r \mbox{ and } |t-s| \leq r^2 \}. \] 
If $E \subseteq \R^{\adim+1},$ define the \textit{parabolic distance to $E$} by $\pdist(X,t,E) = \inf \{\pd(X,t,Y,s) : (Y,s) \in E\}$ and its \textit{parabolic diameter} by $\pdiam(E) = \sup\{\pd(X,t,Y,s) : (X,t),(Y,s) \\ \in E\}.$ 
These sets will be referred to as \textit{balls}, the parameter $r$ is its \textit{radius}.  We define dilates of balls by
$$AB_r(X,t) = B_{Ar}(X,t), \,\,\, A>0.$$
Given $t \in \R,$ let 
 \begin{align}\label{e:E^t}
 	E^t = E \cap (\R^\adim \times \{t\}).
\end{align}

\begin{definition}
Given a pair of non-empty closed sets $E,F \subseteq \R^{\adim+1}$, define the \textit{excess of $E$ over $F$} by 
\begin{align}\label{e:ex}
	 \text{ex}(E,F) = \sup_{y \in E} \pdist(y,F).
\end{align}
By convention, we set $\text{ex}(\emptyset,F) = 0.$ Then, given a ball $B_r(X,t)$, define
\begin{align}\label{e:local-hausdorff-distance}
	d_{B_r(X,t)}(E,F) = \frac{2}{r}\max \left\{\text{ex}(E \cap B_r(X,t),F), \text{ex}(F \cap B_r(X,t),E) \right\}.
\end{align}
Occasionally we may write $d_{X,t,r}(E,F)$ to denote the above quantity, which is sometimes referred to as the \textit{relative Walkup-Wets distance}.
\end{definition}

\begin{definition}
	For $\eta,\delta > 0$ and $E \subseteq \R^{\adim+1}$, define 
	\[ \Hmeasure^\eta_{p,\delta}(E) = \inf \left\{ \sum \pdiam(E_j)^\eta \colon E \subseteq \bigcup_j E_j \mbox{ and } \pdiam(E_j) \leq \delta\right\}.\]
	Then, define the \textit{parabolic Hausdorff measure of dimension} $\eta$ by 
	\[ \pHmeasure^\eta(E) = \lim_{\delta \rightarrow 0} \Hmeasure^\eta_{p,\delta}(E). \]
	Define the \textit{(Euclidean) Hausdorff measure of dimension $\eta$}, $\Hmeasure^\eta$, similarly, by replacing $\pdiam$ in the above definition by the usual Euclidean diameter. 
\end{definition}

\begin{definition}
Let $1 \leq \idim \leq \adim-1$. A set $E \subseteq \R^{\adim+1}$ is called \textit{parabolic Ahlfors $(\idim +2)$-regular} if there exists a constant $A \geq 1$ such that whenever $(X,t) \in E$ and $0 < r < \pdiam(E),$ we have
	\[ A^{-1} r^{\idim+2} \leq \pHmeasure^{\idim+2}(E \cap B_r(X,t)) \leq A r^{\idim+2}.\]
Given a parabolic $(\idim+2)$-regular set $E \subseteq \R^{\adim+1}$, we will often write 
\[ \surfacemeasure \coloneqq \pHmeasure^{\idim+2}|_E.\]
\end{definition}

\subsection{Vertical planes} In this section we collect some useful results for \textit{$t$-independent planes} and their associated projections. Let us start by setting some notation.

\begin{definition}
	An affine plane $V \subseteq \R^{\adim+1}$ is \textit{$t$-independent} if $V$ contains a line in the $t$-direction. Such a plane is also called \textit{vertical}. We let $V^0$ denote the affine plane in $\R^\adim$ such that $V = V^0 \times \R$. We let $\cP_m$ be the collection of vertical planes for which $V^0$ has Euclidean dimension $m$ and $\verticalplanes \coloneqq \verticalplanes_\idim$. Every $V\in\verticalplanes$ has Euclidean dimension $\idim+1$ and parabolic Hausdorff dimension $\idim+2$. Given $V \in \verticalplanes$, we let
	\begin{itemize}
		\item [(i)] $V'$ denote the linear subspace which is parallel to $V$;
		\item [(ii)] $V^\perp$ denote the $(\adim - \idim)$-dimensional linear subspace which is orthogonal to $V'$ (with respect to the Euclidean inner product on $\R^{\adim+1}$);
		\item [(iii)] $\pi_V, \ \pi_V'$ and $\pi_V^\perp$ denote the orthogonal projections onto $V$, $V'$ and $V^\perp$ respectively.
	\end{itemize} 
	Given a further plane $W \in \verticalplanes$, we let 
	\begin{itemize}
		\item [(iv)] $\angle(V,W) \coloneqq d_{B_1(0,0)}(V',W'). $
	\end{itemize}
\end{definition}

\begin{lemma}\label{l:angAB}
	Let $V,W \in \cP$ and $(X,t) \in V$. Let $\ve ,r > 0$ and suppose $d_{B_r(X,t)}(V,W) \leq \ve.$ Then, if $W'' \in \verticalplanes$ is the plane parallel to $V$ such that $(X,t) \in W''$, we have 
	\begin{align}\label{e:d_B2}
		\angle(V,W) = d_{AB_{r}(X,t)}(V,W'') \lesssim \ve
	\end{align}
	for all $A > 0$, with implicit constant which does not depend on $A$. 
\end{lemma}

\begin{proof}
	We can suppose without loss of generality that $(X,t) = (0,0)$ and $r = 1$. In this case it follows that $W'' = W'$ (the linear subspace parallel to $W$). The first equality in \eqref{e:d_B2} follows from the fact that 
	\begin{align}\label{e:scaling}
		d_{B_1(0,0)}(P,T) = d_{B_A(0,0)}(P,T)  \quad  \mbox{ for all } A > 0
	\end{align}
	whenever $P,T \in \verticalplanes$ are linear subspaces. Thus, our goal is to prove the inequality. Using \eqref{e:scaling}, it suffices to prove the inequality in the case that $A = 1/2$. Now, let $B = B_{1/2}(0,0)$ and fix some $(Y,s) \in W' \cap B$.  Because $d_{2B}(V,W) \leq \epsilon$, it must be that $\dist(Y,s,Y',s) \leq 2\epsilon$ for every $(Y',s) \in W$.  Fix $(Y',s) \in W$.  Then, there is some $(Y'',s) \in V$ with $d_p(Y',s,Y'',s) \leq \epsilon$.  So, 
    $\dist_p(Y,s,V) \lesssim \epsilon \approx \epsilon r_B.$  Similarly, we find that $\dist_p(Z,u,W') \lesssim \epsilon r_B$.  Combining these two estimates gives the desired result because $(Y,s)$ and $(Z,\tau)$ were arbitrary.
\end{proof}

\begin{lemma}\label{l:proj}
	Suppose $V,W \in \verticalplanes$ have non-empty intersection and $(X,t) \in V \cap W.$ Let $\ve,r>0$ and suppose that $d_{B_r(X,t)}(V,W) \leq \ve.$ Then, for $(Y,s) \in W$, we have 
	\begin{align}\label{e:proj}
		\|\pi^\perp_W ( \pi_V(Y,s))\| \lesssim \ve |X-Y|.
	\end{align}
\end{lemma}

\begin{proof}
	We can suppose without loss of generality that $(X,t) = (0,0)$ and $r = 1$. Furthermore, since $V,W \in \verticalplanes$, we can also suppose that $(Y,s) = (Y,0)$. Let $(Z,0) = \pi_V(Y,0)$. Since $d_{B_1(0,0)}(V,W) \leq \ve$ and $(0,0) \in V \cap W$, it follows that $d_{B_{|Z|}(0,0)}(V,W) \leq \ve$. In particular, since $(Z,0) \in V \cap B_{|Z|}(0,0)$, we get 
	\begin{align}
		\| \pi^\perp_W(\pi_V(Y,0)) \| = \| \pi_W^\perp(Z,0) \|  = \| (Z,0) - \pi_W(Z,0) \| \leq \ve |Z| \leq \ve |Y|,
	\end{align}
	which finishes the proof. 
\end{proof}

\begin{lemma}\label{l:angAB2}
	Let $V,W \in \verticalplanes$, $(X,t) \in V$, $\ve, r > 0$ and suppose that $d_{B_r(X,t)}(V,W) \leq \ve$. Then, for each $(Y,s),(Z,u) \in W$, we have 
	\begin{align}\label{e:proj1}
		\|\pi_V(Y,s) - (Y,s)\| \lesssim \ve |X-Y| + \| \pi^\perp_W(X-Y,0) \| \lesssim \ve(|X-Y| + r)
	\end{align}
	and 
	\begin{align}\label{e:proj2}
		\|\pi_V(Y,s) - \pi_V(Z,u) - (Y,s) + (Z,u)\| \lesssim  \ve |Y-Z| . 
	\end{align}
\end{lemma}

\begin{proof}
	Assume, without loss of generality, that $(X,t) = (0,0)$ and $r = 1$. Set $\ve = d_{B_1(0,0)}(V,W)$. Fix to begin with some $(Y,s) \in W$, let $W'$ be the linear subspace parallel to $W$ and let $(Y',s) = \pi_{W'}(Y,s)$. By definition, we have 
	\begin{align}
		\| (Y,s) - (Y',s) \|  = \| \pi_W^{\perp}(-Y,0) \|. 
	\end{align}
	By \eqref{e:d_B2}, we have $d_{B_1(0,0)}(V,W') \lesssim \ve$. Thus, by \eqref{e:proj},  
	\begin{align}
		\| \pi_V(Y,s) - (Y,s) \| &\leq \| \pi_V(Y,s) - \pi_V(Y',s) \| +  \| \pi_V(Y',s) - (Y',s) \| \\
		&\hspace{2em} + \| (Y',s) - (Y,s) \| \\
		&\leq \| \pi_V^\perp(Y',s) \| + 2\| \pi_W^{\perp}(-Y,0) \| \lesssim \ve |Y| + \| \pi_W^{\perp}(-Y,0) \| . 
	\end{align}
	which completes the proof of the left-most inequality in \eqref{e:proj1}. For the second inequality, we note instead that Lemma \ref{l:proj} gives $\| \pi_W^\perp(-Y,0) \| \lesssim \ve |Y|$, which yields the desired estimate after normalization.
	
	For \eqref{e:proj2}, let us fix a further point $(Z,u) \in W$. We note that $(Y-Z,s-u) \in W'$, so that $\| \pi_W^\perp(Z-Y,0) \| = 0$. Thus, by linearity of the projection and \eqref{e:proj1}, the left-hand side of \eqref{e:proj2} is 
	\begin{align}
		\| \pi_V(Y-Z,s-u) - (Y-Z,s-u) \| \lesssim \ve |Y-Z|, 
	\end{align}
	as required. 
\end{proof}

The following topological result will be useful throughout. 

\begin{lemma}\label{l:topology}
	Let $P \in \verticalplanes$ and $B = B_r(Y,s)$ be a ball centered on $P$. Let $0 < \ve <1/4$ and suppose $g \colon P \cap B \to P$ is a continuous map which fixes the $t$-component and satisfies $\pd(g(Y,s),(Y,s)) \leq \ve r$. Then, $g(P \cap B) \supseteq P \cap (1-2\ve)B$. 
\end{lemma}

\begin{proof}
	Without loss of generality, we can suppose 
	\begin{align}
		P = \R^{n-1} \times \{0\} \times \dots \times \{0\} \times \R \quad \mbox{ and } \quad B = B_1(0,0). 
	\end{align}
	Define $\lambda \colon \R^{\adim+1} \to \R^{\adim+1}$ by $\lambda(Y,s) = (Y/(1+\ve),s)$ and set $\tilde{g} \coloneqq \lambda \circ g$. Then, $\tilde{g}$ is continuous, fixes the $t$-component, maps $P \cap B$ to $P \cap B$, and satisfies $\pd(\tilde{g}(Y,s) , (Y,s)) \leq  2\ve$ for all $(Y,s) \in P \cap B$. For $t \in [-1,1]$, let $S^t \coloneq (P \cap B)^t$ and let $\tilde{g}_t \colon S^t \to S^t$ be the restriction of $\tilde{g}$ to $S^t$. Then $\tilde{g}_t$ is a continuous map from $S^t$ (a Euclidean ball of radius one) to itself, which distorts distances by at most an additive factor $2\ve$. It now follows (see for example \cite[Lemma 7.3]{bate2020purely}) that
	\begin{align}
		g_t(S_t) \supseteq (1-2\ve) S_t.
	\end{align}
	Since $t \in [-1,1]$ was arbitrary, this completes the proof. 
\end{proof}

\subsection{Christ-David cubes and coronizations} For any given parabolic Ahlfors regular set, we will need a collection of subsets of $E$ which behave like Euclidean dyadic cubes. These are the so-called \textit{Christ-David cubes}. Fix some constants 
\begin{align}
	\Cinner \coloneqq 1/48 \quad \mbox{ and } \Couter \coloneqq 4. 
\end{align}
The following formulation is a consequence of \cite{hytonen2012non}, which holds for more general spaces. 

\begin{lemma}\label{cubes}
	There exists $\rho \in (0,1/1000)$ such that the following holds. Let $E \subseteq \R^{\adim+1}$ be parabolic Ahlfors $(\idim+2)$-regular and $\{E_k\}_{k \in \mathbb{Z}}$ be a sequence of subsets of $E$ such that 
	\begin{align}\label{e:sep-net}
		\begin{cases}
			\pd(X,t,Y,s) \geq \rho^k/4 &\mbox{ for all distinct } (X,t),(Y,s) \in E_k; \\
			\inf_{y \in E_k} \pd(X,t,Z,u) \leq 2\rho^k  &\mbox{ for all }  (Z,u) \in E. 
		\end{cases}
	\end{align}
	Then, for each $k \in \mathbb{Z}$, there is a collection $\cubes_k$ of subsets of $E$ such that:
	\begin{enumerate}
		\item \label{cubes1} If $Q_1,Q_2 \in \cubes = \bigcup_{k}\cubes_k$ and $Q_1 \cap Q_2 \not= \emptyset,$ then $Q_1 \subseteq Q_2$ or $Q_2 \subseteq Q_1.$ 
		\item \label{cubes2} For $Q \in \cubes_k$, let $\ell(Q) \coloneqq \Couter\rho^k.$ Then, there is $(X_Q,t_Q) \in E_k$ such that
		\begin{align*}
			E \cap B_{\Cinner\ell(Q)}(X_Q,t_Q) \subseteq Q \subseteq B_{\ell(Q)}(X_Q,t_Q) \eqqcolon B_Q. 
		\end{align*}
		\item \label{cubes3} For each $k \in \mathbb{Z},$ $E = \bigcup_{Q \in \cubes_k} Q.$
		\item \label{cubes4} For every $\ve > 0$ there is $\tau > 0$ such that for every $Q \in \cubes$, 
		\begin{align}
			\pHmeasure^{\idim+2}(\{(X,t) \in Q \colon \pdist(X,t,E \setminus Q) < \tau \ell(Q)\}) < \ve \pHmeasure^{\idim+2}(Q).
		\end{align}
	\end{enumerate}
\end{lemma}

\begin{remark}
	The exact constant $\Cinner$ and $\Couter$ are not so important. In fact, we will encounter in our proof collections of subsets which satisfy Lemma \ref{cubes} \eqref{cubes1} -- \eqref{cubes3} with different constants. Thus, given a set $E$, we will call any collection of subsets $\cubes$ of $E$ a \textit{system of Christ-David cubes with constant} $(C,c)$ if it satisfies Lemma \ref{cubes} \eqref{cubes1} -- \eqref{cubes3}  with constants $(C,c)$. 
\end{remark}

\begin{definition}
Given a cube $Q \in \cubes$, let $\Parent(Q)$ denote its \textit{parent}, i.e. the unique cube $R \in \cubes$ satisfying $Q \subseteq R$ and $\ell(Q) = \rho \ell(R),$ and $\text{Child}(Q)$ denotes its \textit{children}, i.e. the cubes $R \in \cubes$ such that $\Parent(R) = Q.$ We denote its \textit{siblings} by $\Sibling(Q) \coloneqq \text{Child}(\Parent(Q)).$  
\end{definition}

\begin{definition}\label{StoppingTime}
	A subset $S \subseteq \cubes$ is called a \textit{stopping-time region} if the following hold.
	\begin{enumerate}
		\item There is a \textit{top cube} $Q(S) \in S$ such that $Q(S)$ contains all cubes in $S$. 
		\item If $Q \in S$ and $Q \subseteq R \subseteq Q(S),$ then $R \in S$.
		\item If $Q \in S$, then all siblings of $Q$ are also in $S$. 
	\end{enumerate}
	The \textit{minimal cubes} of $S$, denoted by $\min(S)$, are those cubes in $S$ whose children are not in $S$. Let $z(S)$ denote the points contained in $Q(S)$ for which we never stopped, i.e. 
	\begin{align}\label{d:z}
		z(S) = Q(S) \setminus \bigcup \{Q : Q \in \min(S) \}.
	\end{align}
\end{definition}

\begin{definition}\label{d:packing-condition}
	A subset $\cB \subseteq \cubes$ is said to satisfy a \textit{Carleson packing condition} if there exists $C  > 1$ such that 
	\begin{align}
		\sum_{\substack{Q \in \cB \\ Q \subseteq Q_0}} \ell(Q)^{\idim+2} \leq C \ell(Q_0)^{\idim+2} , \quad Q_0 \in \cubes.
	\end{align}
\end{definition}

\begin{definition}\label{d:coronization}
	Given a region $S \subseteq \cubes$ satisfing Definition \ref{StoppingTime}, a triple $(\cB,\cG,\cF)$ is said to be a \textit{coronization} of $S$ if $\cB$ and $\cG$ are disjoint collections of cubes which partition $S$, $\cF$ is a collection of disjoint stopping-time regions in $S$ such that 
	\begin{align}
		\cG = \bigcup_{T \in \cF} T ,
	\end{align}
	and $\cB$ and $\{Q(T)\}_{T \in \cF}$ satisfy Carleson packing conditions. 
\end{definition}
\begin{remark}
    We can also define coronizations of a given grid $\mathcal S$ of Christ-David cubes in the same way.
\end{remark}

We will now gather some important properties of Christ-David cubes, stopping-time regions and coronizations. For this, let us fix for the remainder of the subsection a parabolic Ahlfors $(\idim+2)$-regular set and a system of Christ-David cubes $\cubes$ on $E$ with constants $(C_E,c_E)$.

\begin{proposition}\label{p:corresponding-stopping-time}
	Let $F \subseteq \R^{\adim+1}$ be parabolic Ahlfors $(\idim+2)$-regular, $\cubes_F$ a system of Christ-David cubes on $F$ with constants $(C_F,c_F)$, and $(\cB_F,\cG_F,\cF_F)$ a coronization of $\cubes_F$. Suppose $S$ is a stopping-time region in $\cubes$ such that 
	\begin{align}\label{e:close-to-F}
		\pdist(X_Q,t_Q,F) \lesssim \ell(Q) \quad \mbox{ for each } Q \in S.
	\end{align}
	Then, there is a coronization $(\cB_S,\cG_S,\cF_S)$ of $S$ such that the following holds. For any $T \in \cF_S$ there exists a corresponding $T' \in \cF_F$ and for any $Q \in T$ a corresponding $Q' \in T'$ such that 
	\begin{align}\label{e:properties-Q}
		\ell(Q) \sim \ell(Q') \quad \mbox{and} \quad \pdist(Q,Q') \lesssim \ell(Q). 
	\end{align} 
	Moreover, the packing constants for $(\cB_S,\cG_S,\cF_S)$ can be chosen to depend only on the packing constants for $(\cB_F,\cG_F,\cF_F)$.
\end{proposition}

\begin{proof}
	The proof is essentially contained in \cite[Section I.3.2]{david1993analysis}. We outline some of the details here, but direct the reader to aforementioned reference for the full details. Fix some stopping-time region $S$ satisfying \eqref{e:close-to-F}. Let $A$ be a large constant (to be chosen later). We say two set $Q,R$ in $\R^{\adim+1}$ are \textit{$A$-close} if 
	\begin{align}
		A^{-1} \diam(Q) \leq \diam(R) \leq A \diam(Q)
	\end{align}
	and 
	\begin{align}
		\pdist(Q,R) \leq A(\diam(Q) + \diam(R)). 
	\end{align}
	Let $\cB_F^1$ be the collection of cubes $Q$ in $\cubes_F$ which satisfy one of the conditions: 
	\begin{itemize}
		\item [(a)] $Q$ is $A$-close to some cube in $\cB_F;$
		\item [(b)] $Q$ is $A$-close to $Q(S')$ for some $S' \in \cF_F;$
		\item [(c)] $Q \in \cG_F$ and $Q$ is $A$-close to another element of $\cG_F$ which lies in a different stopping-time region. 
	\end{itemize}
	Let $\cB'$ be the cubes in $\cubes$ which are $A$-close to an element of $\cB_F^2$. It is shown in \cite[Lemma 3.26]{david1993analysis} that $\cB_F^1$ satisfies a Carleson packing condition. Thus, the same is true of $\cB'$ by \cite[Lemma 3.27]{david1993analysis} (which says that $A$-closeness does not destroy packing conditions). It is now standard to construct a coronization $(\cB_S,\cG_S,\cF_S)$ of $S$ such that
	\begin{align}\label{e:nested-bad}
		\cB' \cup \min(S) \subseteq \cB_S
	\end{align}
	(see, for example,  \cite[Lemma 3.22]{david1993analysis}). 
	
	Fix some $T \in \cF_S$ and $Q \in T.$ By construction $T \subseteq S$, so there exists a point $(X_Q',t_Q') \in F$ such that 
	\begin{align}
		\pdist(X_Q,t_Q,X_Q',t_Q') \lesssim \ell(Q). 
	\end{align}
	Let $Q'$ be the cube in $\cubes_{F,k}$ such that $(X_Q',t_Q') \in Q',$ where $k \in \mathbb{Z}$ is such that $Q \in \cubes_{E,k}.$ The inequalities in \eqref{e:properties-Q} are then clear, with constant depending on the cube family parameters. Thus, for $A$ large enough, $Q'$ is $A$-close to $Q$. This implies (recalling the definition of $\cB'$) that $Q' \not\in \cB_F^1$, so that $Q'$ and all its neighbours must lie in the same stopping-time region $T' \in \cF_F.$ Moreover, if $\tilde{Q}$ is a parent or child of $Q$, then $\tilde{Q}$ must be associated to the same $T' \in \cF_F.$ This fact implies that \textit{all} cubes in $S$ must be associated to the same $T' \in \cF_F$, as required. 
\end{proof}

\begin{corollary}\label{c:corresponding-stopping-time}
	Suppose $(\cB,\cG,\cF)$ is a coronization of $\cubes$ such that, for each $S \in \cF$, there is a set $F_S$ such that 
	\begin{align}\label{e:corresponding-stopping-time}
		\pdist(X_Q,t_Q,F_S) \lesssim \ell(Q) \quad \mbox{ for each } Q \in S.
	\end{align}
	Furthermore, suppose that each $F_S$ is parabolic Ahlfors $(\idim+2)$-regular with constant $C_0$, $\cubes_S$ is a system of Christ-David cubes on $F_S$ with constants $(\Cinner',\Couter')$, and $(\cB_{F_S},\cG_{F_S},\cF_{F_S})$ is a coronization of $\cubes_S$ with uniform packing constants. Then, there exists a coronization $(\cB',\cG',\cF')$ of $\cubes$ such that,
	\begin{enumerate}
		\item for each $T \in \cF'$ there exists $S \in \cF$ such that $T \subseteq S;$ 
		\item for each $T \in \cF'$ there exists a corresponding $T' \in \cF_{F_S}$ and for each $Q \in T$ a corresponding $Q' \in T'$ such that  
		\begin{align}
			\ell(Q) \sim \ell(Q') \quad \mbox{ and } \quad \pdist(Q,Q') \lesssim \ell(Q).
		\end{align}
	\end{enumerate}
	The packing constants for $(\cB',\cG',\cF')$ can be chosen to depend only on the packing constants for $(\cB,\cG,\cF)$ and the uniform packing constants for $\{(\cB_{F_S},\cG_{F_S},\cF_{F_S})\}_{S \in \cF}$. 
\end{corollary}

\begin{proof}
	For each $S \in \cF$, let $(\cB_S,\cG_S,\cF_S)$ be the coronization of $S$ obtained by applying Proposition \ref{p:corresponding-stopping-time}. Then, set 
	\begin{align}
		\cB' = \cB \cup \bigcup_{S \in \cF} \cB_S, \quad  \cG' = \bigcup_{S \in \cF} \cG_S \quad \mbox{and}  \quad \cF' = \bigcup_{S \in \cF} \cF_S. 
	\end{align}
	Properties (1) and (2) are then immediate from Proposition \ref{p:corresponding-stopping-time}. For the packing conditions, fix some $Q_0 \in \cubes.$ If $Q_0 \in \cB$ then, using the fact that $\cB,$ $\cB_S$ and $\{Q(S)\}_{S \in \cF}$ satisfy Carleson packing conditions, we have 
	\begin{align}
		\sum_{\substack{Q \in \cB' \\ Q \subseteq Q_0}} \ell(Q)^{\idim+2} &\leq \sum_{\substack{Q \in \cB \\ Q \subseteq Q_0}} \ell(Q)^{\idim+2} + \sum_{\substack{S \in \cF \\ Q(S) \subseteq Q_0}}\sum_{Q \in \cB_S} \ell(Q)^{\idim+2} \\
		&\lesssim \ell(Q_0)^{\idim+2} + \sum_{\substack{S \in \cF \\ Q(S) \subseteq Q_0}} \ell(Q(S))^{\idim+2} \lesssim \ell(Q_0)^{\idim+2}. 
	\end{align}
	If instead $Q_0 \in S_0 \in \cF,$ then using the above estimate and the fact that $\cB_{S_0}$ satisfies a Carleson packing condition, we have 
	\begin{align}
		\sum_{\substack{Q \in \cB' \\ Q \subseteq Q_0}} \ell(Q)^{\idim+2} &\leq \sum_{\substack{Q \in \cB \\ Q \subseteq Q_0}} \ell(Q)^{\idim+2} + \sum_{\substack{S \in \cF \\ Q(S) \subseteq Q_0}}\sum_{Q \in \cB_S} \ell(Q)^{\idim+2} + \sum_{\substack{Q \in \cB_{S_0} \\ Q \subseteq Q_0}} \ell(Q)^{\idim+2} \\
		&\lesssim \ell(Q_0)^{\idim+2}. 
	\end{align}
	The proof that $\{Q(S')\}_{S' \in \cF'}$ satisfies a Carleson packing condition is similar and we omit the details. 
\end{proof}

The following is essential well-known. We will give a sketch of the proof and refer the reader elsewhere for the specific details within that sketch. 

\begin{lemma}\label{l:smooth}
	Let $n \in \N$ and suppose $\{S_i'\}_{i \in I}$ is a collection of stopping-time regions in $\cubes$ such that $\{Q(S_i')\}_{i \in I} \subseteq \cubes_n$ is a disjoint collection of cubes. For each $i \in I$ there is an extended stopping-time region $S_i \supseteq S_i'$ such that the following conditions hold. Let 
	\begin{align}\label{e:define-S}
		S \coloneqq \bigcup_{i \in I} S_i, \quad \min(S) \coloneqq \bigcup_{i \in I} \min(S_i) \quad \mbox{ and } \quad z(S) \coloneqq \bigcup_{i \in I} z(S_i). 
	\end{align}
	and define $S'$, $\min(S')$, $z(S')$ in the same way, but with respect to $\{S'_i\}_{i \in I}$.  Then, 
	\begin{enumerate}
		\item If $Q,R \in \min(S)$ are such that $10B_Q \cap 10B_R \neq\emptyset$ then $\ell(Q) \sim \ell(R)$. 
		\item If $Q \in \min(S)$ then $z(S) \cap 10B_Q = \emptyset$. 
		\item For each $x \notin z(S)$ there are only finitely many $Q \in S$ such that $x \in B_Q$. 
		\item For every $Q \in S$ there exists $i \in I$ and $Q' \in S_i$
		\begin{align}\label{e:smooth-stat}
			B_Q \subseteq CB_{Q'} \quad \mbox{ and } \quad \ell(Q') \lesssim_\tau  \ell(Q),  
		\end{align}
		where $C = C(\tau)$. 
		\item For every $Q \in \min(S)$, if $R \in \min(S')$ is such that $Q \subseteq R$ then $\ell(Q) < \tau \ell(R)$. Furthermore, for any $\delta > 0$, if $Q \cap \partial_{\delta \ell(R)} R = \emptyset$, then $\ell(Q)  \gtrsim \delta  \tau \ell(R)$. 
	\end{enumerate}
\end{lemma}

\begin{proof}[Sketch of the proof.]
	Let $S' \coloneqq \bigcup_{i \in I} S_i'$. Let $d_{S'} \colon \cubes \to \R$ be a map defined by
	\begin{align}
		d_{S'}(Q) = \inf_{R \in S'} \{ \ell(R) + \dist(Q,R)\}.
	\end{align}
	By \cite{azzam2018analyst}, for any $Q,Q' \in \cubes$ one has  the following Lipschitz-type condition
	\begin{align}
		d_{S'}(Q) \leq 2\ell(Q) + \dist(Q,Q') + 2\ell(Q') + d_{S'}(Q'). 
	\end{align}
	Let $\cM$ be the maximal collection of cubes $Q \in \cubes$ such that $Q \subseteq Q(S_i')$ for some $i \in I$, and 
	\begin{align}\label{e:d_S}
		\ell(Q') < \tau d_{S'}(Q') \quad \mbox{ for all } Q' \in \Sibling(Q). 
	\end{align}
	Let 
	\begin{align}
		S_i \coloneqq \{Q \subseteq Q(S_i') \colon \mbox{there exists } R \in \cM \mbox{ such that } R \subseteq Q\} \supseteq S_i'.  
	\end{align}
	In this way, $S_i$ defines a stopping-time region (this is why \eqref{e:d_S} is stated in terms of siblings). Moreover, we have $\cM = \min(S)$, where $S$ and $\min(S)$ are as in \eqref{e:define-S}. 
	By maximality (and assuming $\tau$ is sufficiently small), one can show (see \cite{azzam2018analyst} for the details), that 
	\begin{align}\label{e:smooth-comp}
		\tau d_{S'}(Q) \sim \ell(Q) \quad \mbox{ for all } Q \in \min(S). 
	\end{align}
	
	To see that Lemma \ref{l:smooth} (1) holds, fix some $Q,R \in \min(S)$ such that $10B_Q \cap 10B_R \neq\emptyset$. By \eqref{e:smooth-comp}, we have 
	\begin{align}
		\tau d_{S'}(Q) &\leq \tau [2\ell(Q) + \dist(Q,Q') + 2\ell(Q') + d_{S'}(Q')] \\
		&\lesssim \tau \ell(Q) + \ell(Q') \lesssim \tau^2 d_{S'}(Q) + \ell(Q'). 
	\end{align}
	Thus, after taking $\tau$ sufficiently small, rearranging the above inequality gives $\tau d_{S'}(Q) \lesssim \ell(Q').$ By \eqref{e:smooth-comp}, this implies $\ell(Q) \lesssim \ell(Q')$. Reversing the roles of $Q$ and $Q'$ gives the reverse inequality. 
	
	For Lemma \ref{l:smooth} (2), it follows from the definition of $z(S)$ that any $x \in z(S)$ is contained in arbitrary small cubes from $S$. If $Q \in \min(S)$ is such that $x \in 10B_Q$, this implies that $d_{S'}(Q) \leq 10\ell(Q)$. Thus, for $\tau$ small enough, \eqref{e:smooth-comp} implies $Q \not\in \min(S)$. 
	
	For Lemma \ref{l:smooth} (3), fix some $x \notin z(S)$. Since $E$ is closed, any point not in $E$ must be contained in finitely many cubes from $S$ since the (ordered) diameters of any infinite sequence of uniformly bounded cubes containing a point must converge to zero. So, let us assume that $x \in E \setminus z(S)$. In this case, it follows from the definition of $z(S)$ that there exists $Q \in \min(S)$ such that $x \in Q$. However, this contradicts Lemma \ref{l:smooth} (2). 
	
	For Lemma \ref{l:smooth} (4), fix some $Q \in S$. By \eqref{e:smooth-comp}, there exists $i \in I$ and $R' \in S_i'$ such that 
	\begin{align}
		\tau 	\ell(R') + \tau \dist(Q,R') \lesssim \ell(Q). 
	\end{align}
	Let $R \in \cubes$ be the maximal ancestor of $R'$ such that $\tau \ell(R) \lesssim \ell(Q)$. By maximality, we have $\tau \ell(R) \sim \ell(Q)$, and since taking ancestors does not increase distances, we still have $\tau \dist(Q,R) \lesssim \ell(Q) \lesssim \tau^{-1} \ell(R)$. These two facts imply \eqref{e:smooth-stat} with constant depending on $\tau$. We finish by showing that $R \in S_i'$. Indeed, let $j \in I$ such that $R \in S_j$. Since $R \subseteq Q(S_j')$, we have by the definition of $d_{S'}$ that $\ell(R) \leq \tau \ell(Q(S_j'))$. Recalling that $\{Q(S_i')\}_{i \in I} \subseteq \cubes_n$, this implies also that $\ell(R) \leq \tau \ell(Q(S_i'))$, which is enough to imply that $R \in S_i'$. 
	
	For Lemma \ref{l:smooth} (5), fix some $Q \in \min(S)$ and $R \in \min(S')$ such that $Q \subseteq R$. For the first part, since $Q \subseteq R$, it follows from \eqref{e:d_S} and the definition of $d_{S'}$, that $\ell(Q) < \tau \ell(R)$. For the second part, fix some $\delta > 0$ and assume that $Q \cap \partial_{\delta \ell(R)}R = \emptyset$. With this assumption, we have $d_{S'}(Q) \geq \delta \ell(R)$. Indeed, if $R$ is the infimiser of $d_{S'}$ then $d_{S'}(Q) = \ell(Q)$, if $R' \in S'$ is the infimiser of $d_{S'}(Q)$, then $d_{S'}(Q) \geq \dist(Q,R') \geq \delta \ell(R)$. We now have $\ell(Q) \gtrsim \delta \tau \ell(R)$ by \eqref{e:smooth-comp}.
\end{proof}

\subsection{\texorpdfstring{Uniform rectifiability coefficients}{beta-numbers and uniform rectifiability}} Here we shall gather some preliminaries on the various uniform rectifiability coefficients  that will be used through the paper. The coefficients $\unilatbeta_E$ and $\bilatbeta_E$ will measure (in distinct ways) how well a set is approximated by a single plane, the coefficient $\baupbeta_E$ will measure how well a set is approximated by a union of planes. Both are defined below.

\begin{definition}
	For a set $E \subseteq \R^{\adim+1}$ and a ball $B = B_r(X,t)$ centered on $E$, let 
	\begin{align}
		\unilatbeta_E(B) &\coloneqq \inf_{P \in \verticalplanes} \unilatbeta_E(B,P) \coloneqq  \inf_{P \in \verticalplanes} \text{ex}(E \cap B,P)/r; \\
		\bilatbeta_E(B) &\coloneqq \inf_{P \in \verticalplanes} \bilatbeta_E(B,P) \coloneqq  \inf_{P \in \verticalplanes} d_B(E,P); \\
			\baupbeta_E(B) &\coloneqq \inf_{U \in \mathscr{U}} \baupbeta_E(B,U) \coloneqq \inf_{U \in \mathscr{U}} d_B(E,U),
	\end{align}
	with $\text{ex}$ as in \eqref{e:ex}, $d_B$ as in \eqref{e:local-hausdorff-distance}, and 
	\begin{align}
		\mathscr{U} \coloneqq \left\{ \bigcup_{i =1}^N P_i \colon N \in \N \mbox{ and } P_i \in \mathscr{P} \right\}
	\end{align}
	i.e., the set of all finite unions of planes. 
\end{definition}

The following monotonicity type property and angle estimate are well-known, see for example \cite{david1993analysis}. 

\begin{lemma}\label{l:monotonicity}
	Let $E \subseteq \R^{\adim+1}$ and suppose $B \subseteq B'$ are balls centered on $E$. Then, 
	\begin{align}
		\bilatbeta_E(B) \leq \left(\frac{r_{B'}}{r_B} \right)\bilatbeta_E(B') \quad \mbox{ and } \quad \baupbeta_E(B) \leq \left(\frac{r_{B'}}{r_B} \right)\baupbeta_E(B').
	\end{align}
\end{lemma}

\begin{lemma}\label{l:beta-angle}
	Let $E \subseteq \R^{\adim+1}$. Suppose $B = B_r(X,t)$ and $B' = B_R(Y,s)$ are ball centered on $E$ such that $B \subseteq B'$, and $P,P' \in \verticalplanes$ are such that $(X,t) \in P$ and $(Y,s) \in P'.$ Then, 
	\begin{align}
		d_B(P,P') \lesssim \bilatbeta_E(B,P) + \left( \frac{R}{r} \right) \bilatbeta_E(2B',P'). 
	\end{align}
\end{lemma}

\subsection{Corona decompositions, BWGL and BAUP} Fix some parabolic Ahlfors $(\idim+2)$-regular set $E \subseteq \R^{\adim+1}$. 

\begin{definition}
Suppose $P \in \cP$.  Suppose $\Gamma = \{(X,A(X,t),t): (X,t) \in P\}$ for some $A:P \rightarrow 
P^\perp$ which is Lip$(1,1/2)$ with constant $\eta$, i.e.,
$$\|A(X,t)-A(Y,s)\| \leq \eta\|(X,t)-(Y,s)\|,\,\,\, \forall \  (X,t),(Y,s) \in P.$$
Then we say that $\Gamma \in \Lip (1,1/2,\eta)$, and refer to $\Gamma$ as an ``$n$-dimensional Lip$(1,1/2)$ graph with constant $\eta$."
\end{definition}

\begin{definition}\label{d:corona-decomposition}
	We say $E$ admits a \textit{corona decomposition} if, for each $A \geq 1$, $\eta , \theta > 0,$ and each system of Christ-David cubes $\cubes$ on $E$, there exists a coronization $(\cB,\cG,\cF)$ of $\cubes$ such that the following holds. For each $S \in \cF$ there is an $\idim$-dimensional $\Lip(1,1/2,\eta)$ graph such that 
	\begin{align}\label{e:definition-corona}
		\pdist(X,t,\Gamma_S) \leq \theta \ell(Q)
	\end{align}
	for all $Q \in S$ and $(X,t) \in E \cap AB_Q.$ 
\end{definition}

\begin{definition}\label{d:BWGL}
	We say $E$ satisfies the \textit{bilateral weak geometric lemma} (BWGL) if, for each $A \geq 1$ and $\ve > 0$, the set $\{Q \in \cubes \colon \bilatbeta_E(AB_Q) > \ve\}$
	satisfies a Carleson packing condition, whenever $\cubes$ is a system of Christ-David cubes on $E$. 
\end{definition}

\begin{definition}\label{d:BAUP}
	We say $E$ satisfies the \textit{bilateral approximation by unions of planes} (BAUP) if, for each $A \geq 1$ and $\ve > 0$, the set $\{Q \in \cubes \colon \baupbeta_E(AB_Q) > \ve\}$
	satisfies a Carleson packing condition, whenever $\cubes$ is a system of Christ-David cubes on $E$. 
\end{definition}

\section{Reifenberg parametrizations}\label{s:Reif}

\subsection{Definitions and the main statement}
In this section we discuss a parabolic version of David and Toro's Reifenberg parametrization theorem for sets with holes \cite{david2012reifenberg}. 

\begin{definition}\label{d:rf}
    Let $\ve>0$.  We say that a set $\Sigma \subset \mathbb R^{d+1}$ is $(n,\epsilon)$-Reifenberg flat if, for each $(X,t) \in \Sigma$ and $r \in (0,\diam \Sigma)$, there is a plane $P \in \cP$ such that $d_{X,t,r}(\Sigma,P) \leq \epsilon$.
\end{definition}

Following \cite{david2012reifenberg}, we state the main result of this section in terms of a coherent collection of balls and planes, defined below.

\begin{definition}\label{d:CCBP}
	Let $\ve > 0$ and $P_0 \in \cP.$  Fix some $k_0 \in \mathbb Z$.  For $k \geq k_0$, let $r_k = 10^{-k}$ and let $\{ (X_{j,k},t_{j,k}) \}_{j \in J_k}$ be a collection of points in $\R^{\adim+1}$ such that 
	\begin{align}\label{e:net}
		\pd( X_{i,k},t_{i,k} , X_{j,k},t_{j,k} ) \geq r_k \mbox{ for all } i,j \in J_k, \ i\not=j.
	\end{align}
	Additionally, let $B_{j,k} = B_{r_k}(X_{j,k},t_{j,k})$ and suppose $P_{j,k} \in \cP$ is such that $(X_{j,k},t_{j,k}) \in P_{j,k}.$ The triple $( P_0 , \{B_{j,k}\},\{P_{j,k}\})_{k \geq k_0, \ j \in J_k}$ is said to be an \textit{$\ve$-coherent collection of balls and planes} ($\ve$-CCBP) if it satisfies the following conditions:
	\begin{enumerate}
		\item \label{CCBP1}  If $k \geq k_0$ and $j \in J_k$, then
		\begin{align}\label{e:cinV^2}
			(X_{j,k},t_{j,k}) \in V^2_{k-1}, 
		\end{align}
		where 
		\begin{align}\label{e:V_k}
			V_{k}^\lambda = \bigcup_{j \in J_k} \lambda B_{j,k} .
		\end{align}
		\item \label{CCBP2} If $j \in J_{k_0}$,  then
		\begin{align}\label{e:cdist0}
		 d_{100B_{j,0}}(P_{j,0},P_0) \leq \ve.
		\end{align}
		\item \label{CCBP3} If $k \geq k_0$ and $i,j \in J_k$ are such that $\pd(X_{i,k},t_{i,k},X_{j,k},t_{j,k}) \leq 100 r_k,$ then
		\begin{align}\label{e:c100rk}
			d_{100B_{i,k}}(P_{i,k},P_{j,k}) \leq \ve.
		\end{align}
		\item \label{CCBP4} If $k \geq k_0 $, $i \in J_k$ and $j \in J_{k+1}$ are such that $\pd(X_{i,k},t_{i,k},X_{j,k+1},t_{j,k+1}) \leq 4r_k,$ then
		\begin{align}\label{e:c2r_k}
			d_{20B_{i,k}}(P_{i,k},P_{j,k+1}) \leq \ve.
		\end{align}
	\end{enumerate}
\end{definition}

Before stating the parametrization theorem, we must introduce several mappings and prove some preliminary results concerning $\ve$-CCPBs. For Lemma \ref{l:bounded-overlap} -- \ref{l:uniform} below, let us fix some $\ve>0$, $k_0 \in \mathbb Z$, and some $\ve$-CCBP, $(P_0,\{B_{j,k}\},\{P_{j,k}\})_{k \geq k_0,\,j\in J_k}$. 

\begin{lemma}\label{l:bounded-overlap}
	Let $k \geq k_0.$ Then, the balls $\{10B_{j,k}\}_{j \in J_k}$ have bounded overlap in $\R^{\adim+1}$, meaning there exists a constant $M = M(\adim) \geq 1$ such that
	\[ \sum_{j \in J_k} \mathds{1}_{10B_{j,k}}(X,t) \leq M \mbox{ for all } (X,t) \in \R^{\adim+1}. \]
\end{lemma}

\begin{proof}
Let $(X,t) \in \R^{\adim+1}$ and let $J_k' \subseteq J_k$ be the indices $j \in J_k$ such that $(X,t) \in 10B_{j,k}.$ Then, the balls $\{\tfrac{1}{5}B_{j,k}\}_{j \in J_k'}$ are contained in $B_{20r_k}(X,t)$ and are disjoint by \eqref{e:net}. Since $\pHmeasure^{\idim+2}(B_r(Y,s)) \sim r^{\idim+2}$ for all $(Y,s) \in \R^{\adim+1}$ and $r > 0$, the result follows by a standard volume argument. 
\end{proof}

\begin{lemma}\label{l:theta}
	Let $k \geq k_0$. Then, there exists a collection of functions $\{\theta_{j,k}\}_{j \in J_k}$ from $\R^{\adim+1}$ to $\R$ with $\supp(\theta_{j,k}) \subseteq 10B_{j,k}$ such that $0 \leq \theta_{j,k} \leq 1,$  
	\begin{align}\label{e:=1}
		\sum_{j \in J_k} \theta_{j,k}(X,t) = 1 \mbox{ for all } (X,t) \in V_k^9,
	\end{align}
	\begin{align}\label{e:lip-theta}
		|\theta_{j,k}(X,t) - \theta_{j,k}(Y,s)| \lesssim r_k^{-1} \pd(X,t,Y,s)
	\end{align}
	and, 
	\begin{align}\label{e:deriv-theta}
		|D^m\theta_{j,k}(X,t)| \lesssim_m r_k^{-m} \mbox{ for each } (X,t) \in \R^{\adim+1} \mbox{ and } m \in \{1,2\},
	\end{align}
	where $D^m$ denotes the $m^{\text{th}}$ differential with respect to the spacial variables in the coordinates generated by $P_0$. 
\end{lemma}

\begin{proof}
	Let $\theta_1$ be a $C^\infty$ bump function on $\R^\adim$ with compact support in $B_\adim(0,10)$ such that $0 \leq \theta_1 \leq 1$ and $\theta_1 \equiv 1$ on $B_\adim(0,9).$ Here, $B_\adim(x,r)$ denotes a Euclidean ball in $\R^\adim$ with centre $x$ and radius $r.$ Then $\theta_1$ satisfies $|D^m \theta_1(X)| \lesssim_m 1$ for all $X \in \R^\adim$ and $m \in \{1,2\}.$ Similarly, let $\theta_2$ be a $C^\infty$ bump function on $\R$ with compact support in $B_1(0,100)$ such that $0 \leq \theta_2\leq 1$ and $\theta_2 \equiv 1$ on $B_1(0,81).$ Note that $\theta_2$ is such that $|\theta_2'(t)| \lesssim 1$ for all $t \in \R$. Then, for $k \geq 0$ and $j \in J_k,$ define 
	\[ \tilde{\theta}_{j,k}(X,t) = \theta_1 \left( \frac{ X - X_{j,k}}{r_k}  \right) \theta_2 \left( \frac{ t - t_{j,k}}{r_k^2} \right). \]
	In this way, 
	\[ \supp(\tilde{\theta}_{j,k}) \subseteq 10B_{j,k} \mbox{ and } \tilde{\theta}_{j,k} \equiv 1 \mbox{ on } 9B_{j,k}. \]  
	Moreover, we have 
	\begin{align}\label{e:Theta} 
		\Theta_k(X,t) \coloneqq \sum_{j \in J_k} \tilde{\theta}_{j,k}(X,t) \geq 1 \mbox{ for all } (X,t) \in V_k^9.
	\end{align} 
	Let $b : \R_+ \rightarrow \R_+$ be a $C^\infty$ function such that $b(s) = 0$ for $s \leq \tfrac{1}{2}$, $b(s) = 1$ for $s \geq 1$ and $|b'(s)|,|b''(s)| \lesssim 1$. Then, let $b_k \coloneqq b \circ \Theta_k$ and define 
	\begin{align} 
		\theta_{j,k} (X,t) = 
		\begin{cases}
			\frac{b_k(X,t)}{\Theta_k(X,t)} \tilde{\theta}_{j,k}(X,t) & \mbox{ if } \Theta_k(X,t) \geq 1/4 \\
			0 & \mbox{ otherwise.}
		\end{cases}		
	\end{align}
	
	It is clear that $\supp(\theta_{j,k}) \subseteq 10B_{j,k}$ and $0 \leq \theta_{j,k} \leq 1$. For \eqref{e:=1}, we note by \eqref{e:Theta} that $\Theta_k(X,t) \geq 1$ whenever $(X,t) \in V_k^9$. In particular, it follows that $b_k(X,t) = 1$, and \eqref{e:=1} follows. The inequality \eqref{e:deriv-theta} follow from direct computation using derivative estimates for $\theta_1, \ \theta_2$ and $b$, the chain and product rules, and Lemma \ref{l:bounded-overlap}. This are tedious so we will not include the details. For \eqref{e:lip-theta}, since $\supp(\theta) \subseteq 10B_{j,k}$, it suffices to check the estimate for some $(X,t) , (Y,s) \in 10B_{j,k}$. In this case, $r_k^{-1}|t-s|^{1/2} \lesssim 1$. Moreover, by direct computation, we have $|\partial_t \theta(Z,u)| \lesssim r_k^{-2}$ for all $(Z,u) \in \R^{\adim + 1}$. These previous two estimates and \eqref{e:deriv-theta} then imply 
	\begin{align}
		|\theta(X,t) - \theta(Y,s)| &\leq |\theta(X,t) - \theta(Y,t) | + | \theta(Y,t) - \theta(Y,s) | \\
		&\lesssim r_k^{-1}|X-Y| + r_k^{-2}|t-s|  \lesssim r_k^{-1}\left[|X-Y| + |t-s|^\frac{1}{2} \right].
	\end{align} 
	\end{proof}

\begin{definition}
	For $k \geq k_0$ and $j \in J_k$, let $\pi_{j,k}$ denote the orthogonal projection onto $P_{j,k}$ and define
	\begin{align}\label{e:sigma_k} 
		\sigma_k(X,t) =  (X,t) + \sum_{j \in J_k} \theta_{j,k}(X,t) [ \pi_{j,k}(X,t) - (X,t) ]. 
	\end{align}
	Then, inductively, define a sequence of functions $\{f_k\}_{k \geq k_0}$ such that $f_k : P_0 \rightarrow \R^{\adim+1}$ by setting
	\begin{align}\label{e:deff}
		f_{k_0}(X,t) = (X,t) \quad \text{and} \quad f_{k+1}= \sigma_k \circ f_k. 
	\end{align}
\end{definition}

\begin{lemma}\label{l:uniform}
	The sequence of functions $\{f_k\}_{k \geq k_0}$ converges uniformly to a continuous function $f$ on $P_0.$  
\end{lemma}

\begin{proof}
	Let $(X,t) \in \R^{\adim+1}.$ Since $\supp \theta_{j,k} \subseteq 10B_{j,k},$ if $j \in J_k$ is such that $\theta_{j,k}(X,t) \not=0,$ then $(X,t) \in 10B_{j,k}.$ Since $(X_{j,k},t_{j,k}) \in P_{j,k},$ this implies $\|\pi_{j,k}(X,t) - (X,t)\| \leq 20r_k$, which, along with Lemma \ref{l:bounded-overlap}, gives
	\[ \|\sigma_k(X,t) - (X,t)\| \leq \sum_{j \in J_k} |\theta_{j,k}(X,t)|  \|\pi_{j,k}(X,t) - (X,t)\| \lesssim 20r_k. \] 
	It follows that
	\begin{align*}\label{e:uniform}
		\pd( f_k(X,t) , f_{k+1}(X,t) ) \lesssim 20r_k
	\end{align*}
	for all $(X,t) \in \R^{\adim+1}.$ Now, since each of the functions $\sigma_k(X,t)$ is continuous (being a combination of sums and products of continuous functions), the result follows.
\end{proof}

\begin{remark}\label{r:timeslice}
	Since the plane $P_{j,k}$ is $t$-independent, both $(X,t) \in \R^{\adim+1}$ and $\pi_{j,k}(X,t)$ have the same time component $t$. It follows from \eqref{e:sigma_k} that $\sigma_k(X,t)$ has time component $t$ and from \eqref{e:deff} that $f(X,t)$ also has time component $t.$ 
\end{remark}

Now for the statement. 

\begin{theorem}\label{t:DT}
		There is $\ve_0 >0$ such that the following holds. Let $0 < \ve < \ve_0$, fix $k_0 \in \mathbb Z$ and suppose $(P_0,\{B_{j,k}\},\{P_{j,k}\})_{k \geq k_0, \ j \in J_k}$ is an $\ve$-CCBP. Then the map $f: P_0 \rightarrow \R^{\adim+1}$ from Lemma \ref{l:uniform} satisfies the following:  
		\begin{enumerate}
			\item \label{DT1} We have 
			\begin{equation}
				E_\infty \coloneqq \bigcap_{K=k_0+1}^\infty \overline{\bigcup_{k=K}^\infty \{(X_{j,k},t_{j,k})\}_{j \in J_k}} \subseteq \Sigma \coloneqq f(P_0).
			\end{equation}
		
			\item \label{DT2} There exists $0 < \tau = \tau(\ve) < 1$, depending only on $\ve,\adim,\idim$, such that, whenever $(X,t),(Y,t) \in P_0$ and $|X-Y| \leq r_{k_0}$,
			\begin{equation}
				\frac{1}{4}r_{k_0}^{-\tau}|X-Y|^{1+\tau} \leq \|f(X,t)-f(Y,t)\| \leq 10r_{k_0}^{\tau}|X-Y|^{1-\tau}.
			\end{equation}
			\item \label{DT3} We have $\|f(X,t)-(X,t)\| \leq C\ve r_{k_0}$ for all $(X,t) \in P_0$, where $C$ depends only on $\adim,\idim$.
			\item \label{DT4} For $k \geq k_0,$
			\begin{equation}\label{e:10}
				\sigma_k(X,t) = (X,t) \ \text{and} \ D\sigma_k(X,t)=I \ \text{for} \ (X,t) \in \R^{\adim+1} \setminus V^{10}_k.
			\end{equation}
			\item \label{DT5} Let $\Sigma_{k_0} = P_0$ and 
			\begin{equation}
				\Sigma_k = \sigma_k(\Sigma_{k-1}), \,\,\, k >k_0.
			\end{equation}
			For $k \geq k_0$, there is a $\emph{Lip}(1,\tfrac{1}{2},C\ve)$ function $A_{j,k}:P_{j,k} \cap 49B_{j,k} \rightarrow P_{j,k}^\perp$ which is $C^2$ in space such that $\|A_{j,k}(X_{j,k},t_{j,k})\| \ \lesssim \ve r_k$, and if $\Gamma_{j,k}$ is its graph over $P_{j,k}$, then 
			\begin{equation}
				\Sigma_k \cap D(X_{j,k},t_{j,k},P_{j,k},49r_k) = \Gamma_{j,k} \cap D(X_{j,k},t_{j,k},P_{j,k},49r_k)
			\end{equation}
			where
			\begin{equation}
				D(X,t,V,r) = \{z+w : z \in V \cap B_r(X,t), \ w \in V^\perp \cap B_r(0,0)\}.
			\end{equation}
			In particular, 
			\begin{equation}\label{Sigma_kPlane}
				d_{49B_{j,k}}(\Sigma_k , P_{j,k}) \lesssim \ve. 
			\end{equation}
			\item \label{DT6} For $k \geq k_0$ and $(X,t) \in \Sigma_k$, there is a plane $P \in \verticalplanes$ and a $\emph{Lip}(1,\tfrac{1}{2},C\ve)$ function $A:P \rightarrow P^\perp$ which is $C^2$ in space such that, if $\Gamma$ is the graph of $A$ over $P$, then 
			\begin{equation}
				\Sigma_k \cap B_{19r_k}(X,t) = \Gamma \cap B_{19r_k}(X,t). 
			\end{equation}
			\item \label{DT7} The surface $\Sigma = f(P_0)$ is $(C\ve,n)$-Reifenberg flat in the sense that for all $(X,t) \in \Sigma,$ and $0 < r < \infty$ there exists a vertical plane $P \in \verticalplanes$ such that $d_{X,t,r}(\Sigma,P) \lesssim \ve.$
			 
			\item \label{DT8} For all $k \geq k_0$ and $(X,t) \in \Sigma_k$,
			\begin{align}\label{Sigma_ksigma_k}
				\|\sigma_k(X,t) - (X,t)\| \lesssim \ve r_k
			\end{align}
			and moreover,
			\begin{align}\label{Sigma_kSigma}
				\pdist(X,t,\Sigma) \lesssim \ve r_k  \ \mbox{ for } \ (X,t) \in \Sigma_k.
			\end{align}
			\item \label{DT9} Let $k \geq k_0$, suppose $(X,t) \in \Sigma_k \cap V_k^8$ and choose $i \in J_k$ such that $(X,t) \in 10B_{i,k}.$ Then
			\begin{align}
				\|\sigma_k(X,t) - \pi_{i,k}(X,t)\| \lesssim \ve r_k. 
			\end{align}
			\item \label{DT11} Let $(X,t) \in \Sigma$ and $0 < r < \infty.$ Then,
			\begin{align}\label{e:lowerreg}
				\sigma(\Sigma \cap B_r(X,t)) \gtrsim r^{\idim+2}. 
			\end{align}			
		\end{enumerate}
\end{theorem}

{
\begin{remark}
    Rather than just the spacial bi-H\"older estimates (as in Theorem \ref{t:DT} \eqref{DT2}), it is possible to obtain global (i.e., considering both space and time coordinates) bi-H\"older estimates. Since this requires some extra work and they are not needed here, we have not included them. With an additional summability condition on the angles between consecutive planes, it is also possible to obtain global (i.e., considering both space and time coordinates) bi-Lipschitz estimates for the parameterizing map. If this control is obtained from some geometric lemma type condition with respect to some underlying set, the map can be made to be \textit{regular} parabolic bi-Lipschitz. These latter results, along with the global bi-H\"older estimates, are contained in an upcoming paper of Bortz, Gurung, Hidalgo-Palencia and the second author \cite{bortz2026parabolic}.  
\end{remark}}

For the remainder of this section we assume, without loss of generality, that $k_0=0$.

The rest of this section is devoted to the proof of Theorem \ref{t:DT}. Let us begin with some remarks. Theorem \ref{t:DT} \eqref{DT4} is clear from \eqref{e:sigma_k} since $\supp \theta_{j,k} \subseteq 10B_{j,k}$ by Lemma \ref{l:theta}. The proof of Theorem \ref{t:DT} \eqref{DT1} is easily translatable from the proof in \cite{david2012reifenberg}, see the beginning of Chapter 6. Theorem \ref{t:DT} \eqref{DT2} takes place in a single time slice, see Remark \ref{r:timeslice}, and the proof is almost identical to the proof of Theorem 8.1 in \cite{david2012reifenberg}. Because these proof is long but only requires minor notational changes, we decided not to include it. We shall include the proofs of the remaining results which generally still follow the ideas presented in \cite{david2012reifenberg} or \cite{edelen2018effective}. The majority of the section will be taken up by the proof of Proposition \ref{p:DT} below which contains the statement of Theorem \ref{t:DT} \eqref{DT5}. Along the way to showing Proposition \ref{p:DT}, we also prove Theorem \ref{t:DT} \eqref{DT8} and \eqref{DT9}. This is the parabolic analogue of \cite[Proposition 5.1]{david2012reifenberg}. After this, we will prove Theorem \ref{t:DT} \eqref{DT3}, \eqref{DT6}, \eqref{DT7} and \eqref{DT11}. \\

\subsection{The main proposition} 

\begin{proposition}\label{p:DT}
	There are constants $\{C_i\}_{i=1}^7$ such that the following holds for all $k \geq 0$ and $j \in J_k.$ First, there is a $\emph{Lip}(1,1/2)$ function $A_{j,k} : P_{j,k} \cap 49B_{j,k} \rightarrow P_{j,k}^\perp$ which is $C^2$ in space, such that 
	\begin{align}\label{e:Anorm}
		\|A_{j,k}(X_{j,k},t_{j,k})\| \leq C_1\ve r_k,
	\end{align}
	\begin{align}\label{e:Alip}
		\|A_{j,k}(X,t) - A_{j,k}(Y,s)\| \leq C_2 \ve \pd(X,t,Y,s)
	\end{align}
	for all $(X,t),(Y,s) \in P_{j,k} \cap 49B_{j,k}$ and, if $\Gamma_{A_{j,k}}$ is the graph of $A_{j,k}$ over $P_{j,k}$, then 
	\begin{align}\label{e:graphA}
		\Sigma_{k} \cap D(X_{j,k},t_{j,k},P_{j,k},49r_k) = \Gamma_{A_{j,k}} \cap D(X_{j,k},t_{j,k},P_{j,k},49r_k).
	\end{align}
	Next, there is a $\emph{Lip}(1,1/2)$ function $F_{j,k} : P_{j,k} \cap 40B_{j,k} \rightarrow P_{j,k}^\perp$ which is $C^2$ in space, such that
	\begin{align}\label{e:Fnorm}
		\|F_{j,k}(X_{j,k},t_{j,k})\| \leq C_3\ve r_k,
	\end{align}
	\begin{align}\label{e:Flip}
		\|F_{j,k}(X,t) - F_{j,k}(Y,s)\| \leq C_4\ve \pd(X,t,Y,s) 
	\end{align}
	for $(X,t),(Y,s) \in P_{j,k} \cap 40B_{j,k},$
	\begin{align}\label{e:Flip1}
		\|F_{j,k}(X,t) - F_{j,k}(Y,s)\| \leq C_5\ve \pd(X,t,Y,s)
	\end{align}
	for $(X,t),(Y,s) \in P_{j,k} \cap 7B_{j,k}$ and, if $\Gamma_{F_{j,k}}$ is the graph of $F_{j,k}$ over $P_{j,k}$, then 
	\begin{align}\label{e:graphF}
		\Sigma_{k+1} \cap D(X_{j,k},t_{j,k},P_{j,k},40r_k) = \Gamma_{F_{j,k}} \cap D(X_{j,k},t_{j,k},P_{j,k},40r_k).
	\end{align}
\end{proposition}

\begin{remark}
    After we prove Proposition \ref{p:DT}, the only properties we will need are \eqref{e:Anorm}--\eqref{e:graphA}. Properties \eqref{e:Fnorm}--\eqref{e:graphF} are there to assist the proof by induction. It may seem that \eqref{e:Flip1} is redundant in light of \eqref{e:Flip}. The point of \eqref{e:Flip1} is that the estimate on the smaller ball $7B_{j,k}$ can be obtained with a constant $C_5$ which is independent of all other constants. This is critical in order to close the induction.  
\end{remark}

We prove Proposition \ref{p:DT} by induction. Since $\Sigma_0 = P_0$ it follows that \eqref{e:Anorm}-\eqref{e:graphA} hold for $k=0$ by Definition \ref{d:CCBP} \eqref{CCBP2}. This establishes the base case. The induction will be carried out in two stages. \\

\noindent\textbf{Stage 1.} The aim of this stage is to show the following:
\begin{lemma}\label{l:Stage1}
	Let $k \geq 0$ and suppose \eqref{e:Fnorm}-\eqref{e:graphF} hold for $k$, then \eqref{e:Anorm}-\eqref{e:graphA} hold for $k+1.$
\end{lemma}
Let $j \in J_{k+1}.$ Before defining the map $A_{j,k+1}$, we need an auxiliary function $g$ which we now describe. Recalling Definition \ref{d:CCBP} \eqref{CCBP1}, there is $i \in J_k$ such that
\begin{align}\label{e:2B_ik}
	(X_{j,k+1},t_{j,k+1}) \in 2B_{i,k}.
\end{align}
By translating we may assume $(X_{i,k},t_{i,k}) = (0,0)$ so that $P_{i,k}$ is a linear subspace of $\R^{\adim+1}$. To simplify notation we write $V = P_{i,k}$ and $W=P_{j,k+1}$. Let $F = F_{i,k}$ be the map satisfying \eqref{e:Fnorm}-\eqref{e:graphF} for $i \in J_k$. Then, define a map $g \colon  W \cap 39B_{i,k} \rightarrow W$ by setting
\begin{align}\label{e:definitionf}
	 g(X,t) = \pi_W [ \pi_V(X,t) + F(\pi_V(X,t))]. 
\end{align}
Since $F$ is $C^2$ in space, the same is true of $g.$ We have the following.

\begin{lemma}
	For $\ve >0$ small enough, the map $g$ is well-defined and satisfies 
	\begin{align}\label{e:f-norm}
		\| g(X,t) - (X,t) \| \lesssim (1+C_3 + C_4)\ve r_k
	\end{align}
	and 
	\begin{align}\label{e:f-inv}
		 \| g(X,t) - g(Y,s) - (X,t) +(Y,s) \|  \leq (C+C_4)\ve \pd(X,t,Y,s)
	\end{align}
	for all $(X,t),(Y,s) \in W \cap 39B_{i,k}.$ In particular, for $\ve$ small enough, \eqref{e:f-inv} implies $g$ is invertible with $\emph{Lip}(1,1/2)$ inverse $g^{-1}: W \supseteq U \rightarrow W \cap 39B_{i,k}$, where 
	\begin{align}\label{e:def-U}
		U \coloneqq g(W \cap 39B_{i,k}) \supseteq W \cap 38B_{i,k} \supseteq W \cap 49B_{j,k+1}. 
	\end{align}
\end{lemma}
\begin{proof}
	Once \eqref{e:f-inv} has been established, the invertibility of $g$ is clear. Once \eqref{e:f-norm} has been established, the second to last inclusion in \eqref{e:def-U} follows from Lemma \ref{l:topology} and the final inclusion follows from \eqref{e:2B_ik}. Thus, it only remains to check that $g$ is well-defined and satisfies \eqref{e:f-norm} and \eqref{e:f-inv}. 
	
	Recalling Definition \ref{d:CCBP} \eqref{CCBP4} and the fact that $(X_{j,k+1} ,t_{j,k+1})\in 2B_{i,k}$, we have
	\begin{align}\label{e:dVW}
		d_{20B_{i,k}}(V,W) \leq \ve.
	\end{align}
	Then, if $(X,t) \in W \cap 39B_{i,k}$, \eqref{e:proj1} implies 
	\begin{align}\label{e:proj3}
		\|\pi_V(X,t)-(X,t)\| \lesssim \ve (|X| + 20r_k) \lesssim \ve r_k,
	\end{align}
	so that $\pi_V(X,t) \in 40B_{i,k}$ for $\ve$ small enough. Since $F$ is defined on $40B_{i,k},$ the map $g$ is well-defined. To prove \eqref{e:f-norm} and \eqref{e:f-inv}, we fix some $(X,t),(Y,s) \in W \cap 39B_{i,k}$. Recall the definition of $\pi_W'$ from Section \ref{s:notation}. By \eqref{e:Fnorm}, \eqref{e:Flip}, \eqref{e:definitionf} and \eqref{e:proj3}, we have
	\begin{align*}
			\|g(X,t) - (X,t)\| &= \|\pi_W [ \pi_V(X,t) + F(\pi_V(X,t))] - \pi_W(X,t)\| \\
			&= \|\pi_W' [ \pi_V(X,t) - (X,t) + F(\pi_V(X,t)) ]  | \\
			&\leq \| \pi_V(X,t) - (X,t)\| + \|F(X_{i,k},t_{i,k})\|  \\
			&\hspace{4em} +  \|F(\pi_V(X,t)) - F(X_{i,k},t_{i,k})\|   \\
			&\lesssim (1 + C_3 + C_4)\ve r_k,
	\end{align*}
	which proves \eqref{e:f-norm}. By \eqref{e:proj2} and \eqref{e:Flip}, the left-hand side of \eqref{e:f-inv} is at most
	\begin{align*}
			&\|\pi_{W}'[ \pi_V(X,t) - \pi_V(Y,s) - 
			(X,t) + (Y,s) +  F(\pi_V(X,t)) -  F(\pi_V(Y,s)) ]\| \\
			&\hspace{2em}\leq \| \pi_V(X,t) - \pi_V(Y,s) - 
			(X,t) + (Y,s)\| + \|F(\pi_V(X,t)) -  F(\pi_V(Y,s))\| \\
			&\hspace{2em}\leq (C+ C_4) \ve \pd(X,t,Y,s),
	\end{align*}
	which proves \eqref{e:f-inv}.
\end{proof}

Now that we have defined $g,$ we can define the function $A = A_{j,k+1} \colon U \rightarrow W^\perp$. Observe that since $(0,0) \in V,$ \eqref{e:dVW} implies there exists a point
\begin{align}\label{e:defX_0}
	(X_0,0) \in W \quad \mbox{such that} \quad |X_0| \leq 20\ve r_k.
\end{align}
Then, for $(X,t) \in U$, define 
\begin{align}\label{e:defA}
        A(X,t) = \pi^\perp_W[ \pi_V(g^{-1}(X,t)) + F( \pi_V(g^{-1}(X,t))) ] - \pi_W^\perp(X_0,0).
\end{align}
Since $F$ is $C^2$ is space, the same is true of $A$. It now only remains to check \eqref{e:Anorm}-\eqref{e:graphA}. Recall the definition of $D(X,t,V,r)$ from the statement of Theorem \ref{t:DT} \eqref{DT5}. We start with \eqref{e:graphA}, which is the following.

\begin{lemma}\label{l:graphA}
	Let $\Gamma_A$ denote the graph of $A$ over $U.$ We have
	\[ \Gamma_A \cap D(X_{j,k+1},t_{j,k+1},W,49r_{k+1}) = \Sigma_{k+1} \cap D(X_{j,k+1},t_{j,k+1},W,49r_{k+1}). \]
\end{lemma}

\begin{proof} We start with the forward inclusion. Let 
	\[(Z,t) \in \Gamma_A \cap D(X_{j,k+1},t_{j,k+1},W,49r_{k+1}).\]
	Since $A$ is defined on $U \supseteq W \cap 49B_{j,k+1}$ (recall this inclusion from \eqref{e:def-U}), there exists $(X,t) \in W \cap 49B_{j,k+1}$ such that $(Z,t) = (X,t) + A(X,t).$ Let $(Y,t) = g^{-1}(X,t)$. Then, by \eqref{e:definitionf}, we have 
	\[(X,t)= g(Y,t) = \pi_W [ \pi_V(Y,t) + F( \pi_V(Y,t))].\]
	Thus, by the definition of $A$, 
	\begin{align}
		\begin{split}\label{e:AF}
			(Z,t) &= (X,t) + A(X,t) \\
				&=	\pi_W [ \pi_V(Y,t) + F( \pi_V(Y,t))] + \pi^\perp_W[ \pi_V(Y,t) + F( \pi_V(Y,t)) ]  - \pi_W^\perp(X_0,0) \\
			&=\pi_W'[ \pi_V(Y,t) + F( \pi_V(Y,t)) ]+ \pi^\perp_W[ \pi_V(Y,t) + F( \pi_V(Y,t)) ]   \\
			&=\pi_V(Y,t) + F( \pi_V(Y,t)),
		\end{split}
	\end{align}
	which implies $(Z,t) \in \Gamma_F$. It follows from \eqref{e:2B_ik} and \eqref{e:dVW} (see \cite{hyde2022d} for a detailed proof in Hilbert space) that
		\begin{align}\label{e:D}
		D(X_{j,k+1},t_{j,k+1},W,49r_{k+1}) \subseteq D(0,0,V,35r_k).
	\end{align}
	Thus, $(Z,t) \in \Gamma_F \cap D(0,0,V,35r_k)$ and \eqref{e:graphF} gives $(Z,t) \in \Sigma_{k+1},$ as required. 
	
	Our goal now is to show the backward inclusion. Fix some
	\[(Z,t) \in \Sigma_{k+1} \cap D(X_{j,k+1},t_{j,k+1},W,49r_{k+1}).\]
	It suffices to show that $(Z,t) \in \Gamma_A.$ By \eqref{e:D} we have $(Z,t) \in D(0,0,V,35r_k),$ which, since $\Sigma_{k+1}$ is the graph of $F$ over $D(0,0,V,40r_k)$, means we can find a point $(Y,t) \in V \cap 35B_{i,k}$ such that 
	\[ (Z,t) = (Y,t) + F(Y,t). \] 
	For $\ve$ small enough, it follows from \eqref{e:proj3} and Lemma \ref{l:topology} that $V \cap 35B_{i,k} \subseteq \pi_V(W \cap 36B_{i,k})$, so we can find a point $(X,t) \in W \cap 36B_{i,k}$ such that $(Y,t) = \pi_V(X,t).$ Then,
	\[ (Z,t) = \pi_V(X,t) + F(\pi_V(X,t)) . \] 
	Let $(Q,t) = g(X,t) \in U.$ By \eqref{e:AF} (replacing $(X,t)$ and $(Y,t)$ with $(Q,t)$ and $(X,t),$ respectively), we have
	\[ (Z,t) = \pi_V(X,t) + F(\pi_V(X,t)) =  (Q,t) + A(Q,t) \in \Gamma_A, \] 
	as required. 
\end{proof}

The two following lemmas finish the proof of Lemma \ref{l:Stage1} and hence Stage 1.

\begin{lemma}
	We have $\|A(X_{j,k+1},t_{j,k+1})\| \leq C_1\ve r_k$.
\end{lemma}
	\begin{proof}
	To simplify notation we will write $(X,t) = (X_{j,k+1},t_{j,k+1})$ and we let $(Y,t) = g^{-1}(X,t).$ By \eqref{e:f-norm} and \eqref{e:proj3}, we have
	\[ \|\pi_V(Y,t)-(X,t)\| \leq \|\pi_V(Y,t)-(Y,t)\|+ \|(Y,t)-(X,t)\| \lesssim (1 + C_3 + C_4)\ve r_k.\]
	Since $(X,t) \in 2B_{i,k}$, this implies $\pi_V(Y,t) \in 7B_{i,k}$ so long as $\ve$ is taken small enough. In particular, we can make use of the estimate \eqref{e:Flip1}. Notice also, in the definition of $A$ (see \eqref{e:defA}) we can replace $(X_0,0)$ with any point in $W$ without changing the function. Choosing this point to be $(Y,t) \in W,$ and using \eqref{e:Fnorm}, \eqref{e:Flip1} and \eqref{e:proj3} with the fact that $\pi_V(Y,t) \in 7B_{i,k},$ we have
	\begin{align*}
			\|A(X,t)\| &= \|  \pi^\perp_W[ \pi_V(Y,t) + F( \pi_V(Y,t)) ] - \pi_W^\perp(Y,t) \| \\
			& =\|\pi_W^\perp[\pi_V(Y,t) - (Y,t) + F(\pi_V(Y,t))]\| \\
		&\leq \|\pi_V(Y,t) - (Y,t)\| + \|F( \pi_V(Y,t))\| \\
		&\leq \|\pi_V(Y,t) - (Y,t)\| + \|F(X_{i,k},t_{i,k})\| + \|F(\pi_V(Y,t)) - F(X_{i,k},t_{i,k})\| \\
		&\leq C(1 + C_3 + C_5)\ve r_k.
	\end{align*}
	The required estimate holds as long as we choose
	\[ C(1+C_3 + C_5) \leq C_1. \]
	\end{proof}
	
\begin{lemma}
	We have $\|A(X,t) - A(Y,s)\| \leq C_2 \ve \pd(X,t,Y,s)$ for all $(X,t),(Y,s) \in W \cap 49B_{j,k+1}.$
\end{lemma}

	\begin{proof}
	Let $(X,t),(Y,s) \in W \cap 49B_{j,k+1}$ and write $(\tilde{X},t) = g^{-1}(X,t)$ and $(\tilde{Y},s) = g^{-1}(Y,s).$ For $\ve$ small enough, \eqref{e:f-norm} implies $(\tilde{X},t),(\tilde{Y},s) \in W \cap 49.5B_{j,k+1}.$ Then, by \eqref{e:2B_ik} and \eqref{e:proj3}, for $\ve$ small enough, we have
	\[\pi_V(\tilde{X},t),\pi_V(\tilde{Y},s) \in V \cap 50B_{j,k+1} \subseteq V \cap 7B_{i,k}.\]
	Now, by Lemma \ref{l:angAB} and \eqref{e:dVW}, we know $d_{B_{i,k}}(V,W') \lesssim \ve,$ with $W'$ as in Section \ref{s:notation}. Then, recalling \eqref{e:defA}, using Lemma \ref{l:proj}, \eqref{e:Flip1}, and \eqref{e:f-inv}, we have
	\begin{align*}
		\|A(X,t) - A(Y,s)\| &\leq \|\pi_W^\perp(\pi_V( \tilde{X} - \tilde{Y},t-s))\| + \|F( \pi_V(\tilde{X},t)) - F( \pi_V( \tilde{Y},s))\| \\
		&\leq C\ve|\tilde{X}-\tilde{Y}| + C_5\ve \pd(\pi_V(\tilde{X},t),\pi_V(\tilde{Y},s))\\
		&\leq (C + C_5)\ve \pd(\tilde{X},t,\tilde{Y},s) \leq C(1+C_5)\ve \pd(X,t,Y,s),
	\end{align*}
	So, if
	\[ C(1+C_5) \leq C_2 \] 
	then \eqref{e:Alip} holds.
\end{proof}

\noindent\textbf{Stage 2.} Our goal here is to show the following.

\begin{lemma} 
	If \eqref{e:Anorm}-\eqref{e:graphA} hold for $k$ then \eqref{e:Fnorm}-\eqref{e:graphF} also hold for $k$.
\end{lemma}

For the rest of this stage we will assume \eqref{e:Anorm}-\eqref{e:graphA} hold for $k$. We start with the following which is Theorem \ref{t:DT} \eqref{DT8}. 

\begin{lemma} 
	Let $(X,t) \in \Sigma_k.$ Then
	\begin{align}\label{e:sig-y}
		\| \sigma_k(X,t) - (X,t) \| \leq C(1 + C_1  + C_2) \ve r_k. 
	\end{align}
\end{lemma}
\begin{proof}
	If $(X,t) \not\in V_k^{10}$ then by definition $\sigma_k(X,t) = (X,t)$ and \eqref{e:sig-y} is trivial. Assume instead that $(X,t) \in V_k^{10}.$ Let $j \in J_k$ such that $(X,t) \in 10B_{j,k}$ and let $A_{j,k}$ be the map which satisfies \eqref{e:Anorm}-\eqref{e:graphA}. If $i \in J_k$ also satisfies $(X,t) \in 10B_{i,k}$, it follows that $\pd(X_{j,k},t_{j,k},X_{i,k},t_{i,k}) \leq 40r_k$. Thus, by Definition \ref{d:CCBP} \eqref{CCBP3}, we have 
	\begin{align}\label{e:good-angle}
		d_{100B_{j,k}}(P_{j,k},P_{i,k}) \leq \ve.
	\end{align}
	By \eqref{e:graphA}, $(X,t)$ is contained in the graph of $A_{j,k}$ over $P_{j,k} \cap 49B_{j,k}$. This, \eqref{e:Anorm} and \eqref{e:good-angle} imply
	\begin{align}
		\begin{split}\label{e:before}
			\|\pi_{i,k}(X,t) - (X,t)\| & \leq  \pdist( \pi_{j,k}(X,t),P_{i,k}) +  \|  \pi_{j,k}(X,t) - (X,t) \|  \\
			&\leq 100 \ve r_k +  \|A_{j,k}(\pi_{j,k}(X,t))\|  \\
			&\leq  100 \ve r_k + \|A_{j,k}(X_{j,k},t_{j,k})\|  \\
			&\hspace{4em} + \|A_{j,k}(\pi_{j,k}(X,t)) - A_{j,k}(X_{j,k},t_{j,k})\| \\
			&\leq C(1 + C_1 + C_2) \ve r_k
		\end{split}
	\end{align}
	Since the functions $\theta_{j,k}$ are supported on $10B_{j,k}$,which have bounded overlap by Lemma \ref{l:bounded-overlap}, we use this and \eqref{e:before} to get
	\begin{align*}
		\|\sigma_k(X,t)-(X,t)\| \leq \sum_{i \in J_k} |\theta_{i,k}(X,t)| \|\pi_{i,k}(X,t) - (X,t)\| \leq C(1 + C_1 + C_2) \ve r_k, 
	\end{align*} 
	which finishes the proof. 
\end{proof}

Now, fix $j \in J_k$. Our ultimate goal is to define a map $F = F_{j,k}$ such that \eqref{e:Fnorm}-\eqref{e:graphF} hold. By translation, we may assume $(X_{j,k},t_{j,k}) = (0,0).$ Let $P_{j,k}, \ B_{j,k}$ and $A_{j,k}$ be the plane, ball and function satisfy \eqref{e:Anorm}-\eqref{e:graphA}. For brevity, we write 
\[ P=P_{j,k}, \ A = A_{j,k} \mbox{ and } \pi = \pi_{j,k}. \] 
Before we can define $F$, we need to prove some preliminary estimates and define an auxiliary function.

\begin{lemma}
	Let $(X,t),(Y,s) \in \Sigma_k \cap 45B_{j,k}.$ Then,
	\begin{align}\label{e:siglip}
		\|\sigma_k(X,t)  - \sigma_k(Y,s) - (X,t) + (Y,s)\| \leq C(1 + C_1 + C_2)\ve \pd(X,t,Y,s).
	\end{align}
\end{lemma}

\begin{proof}
	Adding and subtracting $\sum_{i \in J_k} \theta_{i,k}(Y,s)[\pi_{i,k}(X,t) - (X,t)],$ and applying the triangle inequality, the left-hand side of \eqref{e:siglip} is at most 
	\begin{align*}
		S_1 + S_2 &\coloneqq\sum_{i \in J_k}|\theta_{i,k}(X,t) - \theta_{i,k}(Y,s)|  \|\pi_{i,k}(X,t) - (X,t)\| \\
		&\hspace{4em}+ \sum_{i \in J_k} |\theta_{i,k}(X,t)| \|\pi_{i,k}^\perp((X,t) - (Y,s))\| \\ 
	\end{align*}
	Since $(X,t),(Y,s) \in 45B_{j,k}$ and $(X_{j,k},t_{j,k}) = (0,0),$ if $i \in J_k$ is such that $\theta_{i,k}(X,t) \not=0$ or $\theta_{i,k}(Y,s) \not=0$ then $\|(X_{i,k},t_{i,k})\| \leq 100r_k.$ For such an $i$, Definition \ref{d:CCBP} \eqref{CCBP3} implies
	\begin{align}\label{e:PPik}
		d_{100B_{j,k}}(P,P_{i,k}) \leq \ve
	\end{align}
	so that
	\begin{align}
		\|\pi_{i,k}(X,t) - (X,t)\|  \leq C(1 + C_1 + C_2)\ve r_k
	\end{align}
	by exactly the same estimate as in \eqref{e:before}. Since the balls $10B_{i,k}$ have bounded overlap, using the estimates on $\theta_{j,k}$ from \eqref{e:lip-theta}, it follows that 
	\begin{align}
		S_1 \leq C(1 + C_1 + C_2)\ve \pd(X,t,Y,s).
	\end{align}
	We now estimate $S_2.$ By \eqref{e:graphA}, $\Sigma_k \cap 45B_{j,k}$ is contained in the graph of $A$ over $P \cap 49B_{j,k}$. In particular, $(X,t)= \pi(X,t) + A ( \pi(X,t) ),$ with a similar expression holding for $(Y,t)$. Again, if $\theta_{i,k}(X,t)\not=0$ then \eqref{e:PPik} holds, so that $d_{100B_{j,k}}(P,P_{i,k}') \lesssim \ve$ by Lemma \ref{l:angAB}, where $P_{i,k}'$ is the linear space parallel to $P_{i,k}$. We use these facts, with Lemma \ref{l:proj}, Lemma \ref{l:bounded-overlap}, and \eqref{e:Alip} to get
	\begin{align*}
		S_2 &=  \sum_{i \in J_k} | \theta_{i,k}(X,t) |  \| \pi_{i,k}^\perp\left[ \pi(X,t) + A( \pi(X,t) ) - \pi(Y,s) + A( \pi(Y,s) ) \right] \| \\
		&\leq \sum_{i \in J_k} |\theta_{i,k}(X,t)|\left( \| \pi_{i,k}^\perp [\pi(X-Y,t-s) ] \| + \|A(\pi(X,t)) -  A(\pi(Y,s))\| \right) \\
		&\leq \sum_{i \in J_k} |\theta_{i,k}(X,t)| \left( C\ve|X-Y|+ C_2\ve \pd(\pi(X,t),\pi(Y,s))\right) \\
		&\leq C(1 + C_2)\ve \pd(X,t,Y,s).
	\end{align*}
	This completes the proof. 
\end{proof}

For $(X,t) \in P \cap 41B_{j,k},$ define
\begin{align}\label{d:h}
	 h(X,t) = \pi[\sigma_k((X,t)+A(X,t))]. 
\end{align}

Since $\sigma_k$ and $A$ are both $C^2$ in space, the same is true of $h.$ Let $(X,t),(Y,s) \in P \cap 41B_{j,k}$ and write $(\tilde{X},t) = (X,t) + A(X,t)$ and $(\tilde{Y},s) = (Y,s) + A(Y,s).$ Unravelling the definition and applying \eqref{e:sig-y}, we have
\begin{align}\label{e:h-y}
	\|h(X,t) - (X,t)\|  = \| \pi (\sigma_k(\tilde{X},t)) - \pi(\tilde{X},t)\| + \|A(X,t)\| \leq C(1+C_1 + C_2) \ve r_k
\end{align}
Furthermore, by \eqref{e:Alip} and \eqref{e:siglip},
\begin{align}
	\begin{split}\label{e:liph}
	|\pd(h(X,t),h(Y,s)) - \pd(X,t,Y,s)|&\leq \|\sigma_k(\tilde{X},t) -\sigma_k(\tilde{Y},s) -  (\tilde{X},t) + (\tilde{Y},s) \| \\
    &+\|A(X,t)-A(Y,s)\|\\
	&\leq C(1+C_1+C_2) \ve \|(\tilde{X},t) - (\tilde{Y},s)\|\\
    &+\|A(X,t)-A(Y,s)\|\\
	&\leq C(1+C_1+C_2) \ve \pd(X,t,Y,s)
    \end{split}	
\end{align}
Inequality \eqref{e:liph} implies that $h : P \cap 41B_{j,k} \rightarrow U$ is bijective, where $U = h(P \cap 41B_{j,k}).$ Denote its inverse by $h^{-1}: U \rightarrow P \cap 41B_{j,k}.$ Then \eqref{e:h-y} and \eqref{e:liph} imply $h^{-1}$ is $\text{Lip}(1,1/2,1+C\ve)$ with
\begin{align}\label{e:h}
	\|h^{-1}(X,t) - (X,t)\| \lesssim \ve r_k. 
\end{align}
It now follows from Lemma \ref{l:topology} that 
\begin{align}\label{e:U2}
	U \supset P \cap 40B_{j,k}.
\end{align}
This allows us to define the function $F: P \cap 40B_{j,k} \rightarrow P^\perp$ by setting
\begin{align}\label{e:defF}
	 F(X,t) = \pi^\perp ( \sigma_k ( h^{-1}(X,t) + A(h^{-1}(X,t)))). 
\end{align}

We will now verify \eqref{e:Fnorm}-\eqref{e:graphF} for $F$. Let us start with \eqref{e:graphF}, which is the following.

\begin{lemma}
	Let $\Gamma_F$ be the graph of $F$ over $U$. We have
	\[\Sigma_{k+1} \cap D(0,0,P,40r_k) = \Gamma_F \cap D(0,0,P,40r_k). \] 
\end{lemma}

\begin{proof}
	Let $(X,t) \in \Gamma_F \cap D(0,0,P,40r_k)$. By \eqref{e:U2} and the fact that $\Gamma_F$ is the graph of $F$ over $U$, there exists $(Y,t) \in P \cap 40B_{j,k}$ such that $(X,t) =  (Y,t) + F(Y,t)$. Set $(Z,t) = \sigma_k(h^{-1}(Y,t) + A(h^{-1}(Y,t)))$. By \eqref{e:h-y} and $(Y,t) \in P \cap 40B_{j,k}$, we have $h^{-1}(Y,t) \in P \cap 41B_{j,k}$. Thus, by \eqref{e:graphA}, $h^{-1}(Y,t) + A(h^{-1}(Y,t)) \in \Sigma_{k}$. It follows that $(Z,t) \in \Sigma_{k+1}$. Then, by \eqref{d:h} and \eqref{e:defF}, 
	\begin{align}
		(X,t) = h(h^{-1}(Y,t)) + F(Y,t) = \pi(Z,t) + \pi^\perp(Z,t) = (Z,t) \in \Sigma_{k+1}.
	\end{align} 
	
	We are left to show the forward inclusion.	Let $(X,t) \in \Sigma_{k+1} \cap D(0,0,P,40r_k)$ and let $(Y,t) \in \Sigma_k$ such that $(X,t) = \sigma_{k}(Y,t).$ By \eqref{e:sig-y}, we have $\pd(X,t,Y,t) \lesssim \ve r_k$, so that $\pi(Y,t) \in P \cap 41B_{j,k}.$ By \eqref{e:graphA}, $\Sigma_k \cap D(0,0,P,41r_k)$ is contained in the graph of $A$ over $P \cap 41B_{j,k}$, hence, $(Y,t) = \pi(Y,t) + A(\pi(Y,t))$ and 
	\[ (X,t) = \sigma_k(Y,t) = \sigma_k( \pi(Y,t) + A(\pi(Y,t))). \] 
	Since $h^{-1} : U \rightarrow 41B_{j,k}$ is bijective there exists $(Z,t) \in U$ such that $h^{-1}(Z,t) = \pi(Y,t).$ Then, by the definitions of $h$ and $F$ (see \eqref{d:h} and \eqref{e:defF}), we have
	\begin{align*} 
		(X,t) &= \sigma_k( h^{-1}(Z,t) + A(h^{-1}(Z,t))) \\
		&= \pi( \sigma_k( h^{-1}(Z,t) + A(h^{-1}(Z,t)))  ) + \pi^\perp ( \sigma_k( h^{-1}(Z,t) + A(h^{-1}(Z,t)))  ) \\
		&= h(h^{-1}(Z,t)) + F(Z,t) = (Z,t) + F(Z,t).
	\end{align*}
	This finishes the proof. 
\end{proof}

We now focus on proving \eqref{e:Fnorm}-\eqref{e:Flip1} for $F.$ Let us start with \eqref{e:Flip}. 
\begin{lemma}
	We have $\|F(X,t) - F(Y,s)\| \leq C_4 \ve \pd(X,t,Y,s)$ for all $(X,t),(Y,s) \in P \cap 40B_{j,k}.$ 
\end{lemma}
\begin{proof}
	Let $(X,t),(Y,s) \in P \cap 40B_{j,k}.$ To simplify notation we will write
	\[ (\tilde{X},t) = h^{-1}(X,t) + A(h^{-1}(X,t)), \]
	with $(\tilde{Y},s)$ defined analogously. First, by the definition of $F$ and  \eqref{e:siglip}, we have 
	\begin{align*}
		\|F(X,t) - F(Y,s)\| &= \|\pi^\perp(\sigma_k(\tilde{X},t)) -\pi^\perp(\sigma_k(\tilde{Y},s))\| \\
		&\leq\|\pi^\perp ( \sigma_k(\tilde{X},t) - \sigma_k(\tilde{Y},s)  - (\tilde{X},t) + (\tilde{Y},s))\| \\
		&\hspace{2em} + \|\pi^\perp ( (\tilde{X},t) - (\tilde{Y},s))\| \\
		&\leq C(1+C_1+C_2) \ve \pd(\tilde{X},t,\tilde{Y},s) + \|\pi^\perp ( (\tilde{X},t) - (\tilde{Y},s))\| 
	\end{align*}
We estimate $\pd(\tilde{X},t,\tilde{Y},s)$ using \eqref{e:Alip} and \eqref{e:liph} to get 
\begin{align*}
	\pd(\tilde{X},t,\tilde{Y},s) &= \|(\tilde{X},t) - (\tilde{Y},s)\| \\
	&\leq \|h^{-1}(X,t) - h^{-1}(Y,s)\| + \|A(h^{-1}(X,t)) - A(h^{-1}(Y,s))\| \\
	&\leq (1+C_2 \ve) \|h^{-1}(X,t) - h^{-1}(Y,s)\| \\
	&\leq C(1+C_2\ve) \pd(X,t,Y,s).
\end{align*}
Since $P$ is a linear subspace, $h^{-1}(X,t) \in P$ and $A(h^{-1}(X,t)) \in P^\perp$, it follows that $\pi^\perp(\tilde{X},t) = A(h^{-1}(X,t)).$ Similarly, $\pi^\perp(\tilde{Y},s) = A(h^{-1}(Y,s)).$ Using this with the estimates above, we have  
\begin{align*}
	\|\pi^\perp ( (\tilde{X},t) - (\tilde{Y},s))\| = \|A(h^{-1}(X,t)) - A(h^{-1}(Y,s))\| \leq CC_2\ve \pd(X,t,Y,s).
\end{align*}
Putting everything together gives
	\begin{align*}
		\|F(X,t) - F(Y,s)\| \leq C(1+C_1+C_2)\ve \pd(X,t,Y,s).
	\end{align*}
	If $C_4$ is chosen such that 
	\[ C_4 \geq C(C_1 + C_2 +1) \]
	then \eqref{e:Flip} follows.
\end{proof}

Before we can prove \eqref{e:Fnorm} and \eqref{e:Flip1} we need two preliminary results, the first of which is Theorem \ref{t:DT} \eqref{DT9}.

\begin{lemma}\label{l:sig-pi}
	Let $k \geq 0,$ suppose $(X,t) \in \Sigma_k \cap V_k^8$, and let $i \in J_k$ be such that $(X,t) \in 10B_{i,k}.$ Then,
	\[ \|\sigma_k(X,t) - \pi_{i,k}(X,t)\| \leq  C \ve r_k.\]
\end{lemma}

\begin{proof}
	Since $(X,t) \in \Sigma_k \cap V_k^8$, \eqref{e:=1} implies
	\[\sum_{j \in J_k} \theta_{j,k}(X,t) = 1.\]
	Hence,
	\begin{align*}
		\begin{split}
			\|\sigma_k(X,t) - \pi_{i,k}(X,t)\| &= \| (X,t) + \sum_{j \in J_k} \theta_{j,k}(X,t)[\pi_{j,k}(X,t)-(X,t)] - \pi_{i,k}(X,t) \|  \\
			&= \| \sum_{j \in J_k} \theta_{j,k}(X,t)(\pi_{j,k}(X,t) - \pi_{i,k}(X,t)) \| \\
			&\leq \sum_{j \in J_k} |\theta_{j,k}(X,t)|   \|\pi_{j,k}(X,t) - \pi_{i,k}(X,t)\|. 
		\end{split}
	\end{align*}
	Since $\supp \theta_{j,k} \subseteq 10B_{j,k}$, if $j \in J_k$ is such that $\theta_{j,k}(X,t) \not=0$, then $(X,t) \in 10B_{j,k}$. Since $(X,t) \in 10B_{i,k},$ we have $\pd(X_{j,k},t_{j,k},X_{i,k},t_{i,k}) \leq 40r_k.$ Then, Definition \ref{d:CCBP} \eqref{CCBP3} implies
	\begin{align}\label{e:sig-pi1}
		d_{100B_{i,k}}(P_{i,k},P_{j,k}) \leq \ve
	\end{align}
	so that
	\begin{align}\label{e:sig-pi2} 
		\|\pi_{j,k}(X,t) - \pi_{i,k}(X,t)\| \leq C\ve r_k.
	\end{align}
	Finally, by Lemma \ref{l:bounded-overlap}, we get
	\[ \|\sigma_k(X,t) - \pi_{i,k}(X,t)\| \leq \sum_{j \in J_k} |\theta_{j,k}(X,t)| C\ve r_k \leq C \ve r_k, \]
	as required. 
\end{proof}

\begin{lemma}\label{l:sig-pi-lip}
	Let $k \geq 0$, $(X,t),(Y,s)\in \Sigma_k \cap V_k^8$, and let $i \in J_k$ be such that $(X,t),(Y,s) \in 10B_{i,k}.$ Then, 
	\begin{align}\label{e:sig-pi-lip}
		\|\sigma_k(X,t) - \sigma_k(Y,s) - \pi_{i,k}(X,t) + \pi_{i,k}(Y,s)\| \lesssim \ve \pd(X,t,Y,s).  
	\end{align}
\end{lemma}

\begin{proof}
	Since $(X,t),(Y,s) \in V_k^8,$ we have $\sum_{j \in J_k} \theta_{j,k}(X,t) = 1$ by \eqref{e:=1}, with the analogous statement holding at the point $(Y,s).$ For $j \in J_k$ let $\pi_{j,k}'$ denote the orthogonal projection onto $P_{j,k}'.$ Adding and subtracting
	\[ \sum_{j \in J_k} \theta_{j,k}(X,t) ( \pi_{j,k}(Y,s) - \pi_{i,k}(Y,s))\]
	and applying the triangle inequality, the left-hand side of \eqref{e:sig-pi-lip} is at most  
	\begin{align*}
		&\sum_{j \in J_k} |\theta_{j,k}(X,t)| \| \pi_{j,k}'(X-Y,t-s) - \pi_{i,k}'(X-Y,t-s)\|  \\
		&\hspace{4em} +  \sum_{j \in J_k} |\theta_{j,k}(X,t) - \theta_{j,k}(Y,s)| \|\pi_{j,k}(Y,s) - \pi_{i,k}(Y,s)\| \\
		&= S_1 + S_2.
	\end{align*}
	We start by estimating $S_1.$ For each $j \in J_k$ such that $\theta_{j,k}(X,t) \not= 0,$ we know both \eqref{e:sig-pi1} and \eqref{e:sig-pi2} hold. In particular, using \eqref{e:sig-pi1} along with \eqref{e:d_B2} gives
	\begin{align}\label{e:small-angle}
		d_{B_{100r_k}(0,0)}(P_{i,k}',P_{j,k}') \lesssim \ve
	\end{align}
	Let $S_{1,1} = \| \pi_{j,k}'(X-Y,t-s) - \pi_{i,k}'(X-Y,t-s)\|.$ Then, by \eqref{e:small-angle} and Lemma \ref{l:angAB2},
	\begin{align*}
		S_{1,1} &\leq \|\pi_{j,k}'(X-Y,t-s) - \pi_{i,k}'(\pi_{j,k}'(X-Y,t-s)) \| \\
		&\hspace{4em}+ \|\pi_{i,k}'(\pi_{j,k}'(X-Y,t-s)) - \pi_{i,k}'(X-Y,t-s)\| \\
		&\lesssim \ve |X-Y| \leq \ve \pd(X,t,Y,s).
	\end{align*}
	Then, by Lemma \ref{l:bounded-overlap}, we have
	\[S_1 \lesssim \ve \pd(X,t,Y,s).\] 
	To estimate $S_2,$ we apply \eqref{e:lip-theta} and \eqref{e:sig-pi2} to get
	\begin{align*}
		|\theta_{j,k}(X,t) - \theta_{j,k}(Y,s)|  \|\pi_{j,k}(X,t) - \pi_{i,k}(Y,s)\| &\lesssim r_k^{-1}\pd(X,t,Y,s) \ve r_k,
	\end{align*}
	then Lemma \ref{l:bounded-overlap} implies
	\[S_2 \lesssim \ve \pd(X,t,Y,s).\]	
\end{proof}

Now for \eqref{e:Fnorm} and \eqref{e:Flip1}. We will be able to choose the constants $C_3$ and $C_5$ independent of all other constants $C_i$ and this closes the induction.
\begin{lemma}
	We have $\|F(X,t)\| \leq C_3 \ve r_k$ for all $(X,t) \in P \cap 40B_{j,k}.$ In particular, \eqref{e:Fnorm} holds.
\end{lemma}

\begin{proof}
	Let $(X,t) \in P \cap 40B_{j,k}$ and let $(\tilde{X},t) = h^{-1}(X,t) +A(h^{-1}(X,t)).$  By \eqref{e:h}, we have $h^{-1}(X,t) \in P\cap41B_{j,k}.$  So, by \eqref{e:Anorm}, \eqref{e:Alip}, \eqref{e:h} and the triangle inequality, we have 
	\begin{align}
		\begin{split}\label{e:above}
		\|(\tilde{X},t) - (X,t)\| &\leq \|h^{-1}(X,t) - (X,t)\| + \|A(X_{j,k},t_{j,k})\| \\
		&\hspace{4em} + \|A(h^{-1}(X,t)) - A(X_{j,k},t_{j,k})\| \\
		&\leq C(1 + C_1 + C_2)\ve r_k.
		\end{split}
	\end{align}
	Combining this with \eqref{e:sig-y} gives
	\begin{align*}
		\|\sigma_k(\tilde{X},t) - (X,t)| &\leq \|\sigma_k(\tilde{X},t) - (\tilde{X},t)\| + \|(\tilde{X},t) -(X,t)\|  \leq C(1+C_1 + C_2)\ve r_k
	\end{align*}
	so that $\sigma_k(\tilde X,t) \in 41B_{j,k}$ for $\ve$ small enough. By Lemma \ref{l:sig-pi}, this implies
	\begin{align*}
		&\|F(X,t)\| = \pd(\sigma_k(\tilde X,t),P) =  \| \sigma_k(\tilde X,t) - \pi(\sigma_k(\tilde X,t)) \|\\
       = & \|\sigma_k(\tilde X,t) - \pi(\tilde X,t)\| 
       +\|\pi(\tilde X,t) - \pi(X,t)\| + \|\pi(X,t) - \pi(\sigma_k(\tilde X,t))\|\\
        \leq& C(1+C_1+C_2) \ve r_k.
	\end{align*}
	Choosing $C_3$ large so that
    $$C(1+C_1+C_2) \leq C_3$$
    finishes the proof.
\end{proof}

\begin{lemma}
	We have $\|F(X,t) - F(Y,s)\| \leq C_5 \ve \pd(X,t,Y,s)$ for all $(X,t),(Y,s) \in P \cap 7B_{j,k}.$	
\end{lemma}

	\begin{proof}	
	Let $(X,t),(Y,s) \in 7B_{j,k}$ and set $(\tilde{X},t) = h^{-1}(X,t) + A(h^{-1}(X,t))$ and $(\tilde{Y},s) = h^{-1}(Y,s) + A(h^{-1}(Y,s)).$ By \eqref{e:above}, we see $(\tilde{X},t),(\tilde{Y},s) \in 8B_{j,k}$ for $\ve$ small enough. By Lemma \ref{l:sig-pi-lip}, \eqref{e:liph}, and \eqref{e:Alip}, we have 
	\begin{align*}
		\|F(X,t) - F(Y,s)\| &= \|\pi^\perp ( \sigma_k(\tilde{X},t) - \sigma_k(\tilde{Y},s) ) \| \\
		&=\|\pi^\perp ( \sigma_k(\tilde{X},t) - \sigma_k(\tilde{Y},s) - \pi(\tilde{X},t) + \pi(\tilde{Y},s) ) \| \\
		&\leq C\ve \pd(\tilde{X},t,\tilde{Y},s) \\
        &\leq C\ve (\|h^{-1}(X,t)-h^{-1}(Y,s)\| +\|A(h^{-1}(X,t))-A(h^{-1}(Y,s)) \|)\\
        &\leq C\ve[ 1 + (1+C_1+C_2)\ve + C_2(1+(1+C_1+C_2)\ve)]\pd(X,t,Y,s) \\
		&=: \tilde C(C_1,C_2,\ve) \pd(X,t,Y,s).\\
	\end{align*} 
	Taking $\ve$ small enough so that $\tilde C(C_1,C_2,\ve) \leq \tfrac{1}{2}$ finishes the proof of \eqref{e:Flip} as long as $C_5$ is taken large enough. 
\end{proof}

\subsection{The remaining items} We now focus on proving the remaining items in Theorem \ref{t:DT}. As a reminder, these are \eqref{DT3}, \eqref{DT6}, \eqref{DT7} and \eqref{DT11}. We start \eqref{DT3} below.

\begin{lemma}
Let $(X,t) \in P_0.$ Then 
\begin{align}\label{e:distf_k}
	\| f_k(X,t) - f(X,t) \| \lesssim \ve r_k
\end{align}
for all $k \geq 0.$ In particular, $\| f(X,t) - (X,t) \| \lesssim \ve.$ 
\end{lemma}

\begin{proof}
	Let $(X,t) \in P_0$ and $k \geq 0 .$ Then, recalling \eqref{e:deff} and \eqref{e:sig-y}, we have
	\begin{align}
		\| f_k(X,t) - f(X,t) \| &\leq \sum_{i = k}^\infty \| f_i(X,t) - f_{i+1}(X,t) \| = \sum_{i = k}^\infty \| f_i(X,t) - \sigma_i(f_i(X,t)) \| \\
		&\lesssim \ve \sum_{i=k}^\infty r_i \lesssim \ve r_k.
	\end{align}
\end{proof}

\begin{lemma}[{\cite[Lemma 6.1]{david2012reifenberg}}]\label{l:traj}
	Let $(X,t) \in \R^{\adim+1}$ and set $(X_k,t) = f_k(X,t)$ for $k \geq 0 .$ Then, if $(X_k,t) \in \R^{\adim+1}\setminus V_k^{10}$ for some $k \geq 0$ then $(X_\ell,t) = (X_k,t) \in \R^{\adim+1} \setminus V_\ell^{10}$ for all $l \geq k.$ If $(X,t) \in P_0$ and $(X_k,t) \in V_k^{10}$ for some $k \geq 0$ then $(X_\ell,t) \in V_\ell^4$ for all $0 \leq \ell \leq k-1.$ 
\end{lemma}

\begin{lemma}\label{l:local-graph}
	Let $k \geq 0$ and $(X,t) \in \Sigma_k.$ There is a plane $P \in \verticalplanes$ through $(X,t)$ and a $\emph{Lip}(1,\tfrac{1}{2},C\ve)$ function $A \colon P \rightarrow P^\perp$ which is $C^2$ in space and such that 
	\[ \Sigma_k \cap B_{19r_k}(X,t) = \Gamma \cap B_{19r_k}(X,t)\] 
	where $\Gamma$ denotes the graph of $A$ over $P$. 
\end{lemma}

\begin{proof}
	Suppose first that $(X,t) \in V_k^{30}$ and let $j \in J_k$ so that $(X,t) \in 30B_{j,k}.$ Then, $B_{19r_k}(X,t) \subseteq 49B_{j,k}$, and the plane $P_{j,k}$ and map $A_{j,k}$ from Proposition \ref{p:DT} (after translating slightly) satisfy the desired properties. 
	
	Suppose then that $(X,t) \in \Sigma_k \setminus V_k^{30}.$ Let $m \geq 0$ be the smallest integer so that $(X,t) \in \Sigma_k \setminus V^{30}_m.$ Then $B_{19r_k}(X,t) \subseteq \Sigma_k \setminus V_m^{11}$ by the triangle inequality, and Lemma \ref{l:traj} gives
	\[ \Sigma_k \cap B_{19r_k}(X,t) = \Sigma_m \cap B_{19r_k}(X,t). \]  
	If $m =0$ then $\Sigma_k = P_0$ and the statement is obvious. Suppose then that $m > 0.$ By minimality, we know $(X,t) \in V_{m-1}^{30}.$ Let $j \in J_{m-1}$ so that $(X,t) \in 30B_{j,m-1}.$ As before $B_{19r_k}(X,t) \subseteq 49B_{j,m-1}$ and the result follows Proposition \ref{p:DT} with $P = P_{j,m-1}$ and $A = A_{j,m-1}$ after translating.   
\end{proof}

\begin{lemma}
	Let $(Z,t) \in \Sigma$ and $0 < r \leq 1.$ Then there is a plane $P  \in \verticalplanes$ through $(Z,t)$ so that 
	\begin{align}\label{e:RF}
		 d_{Z,t,r}(\Sigma,P) \leq C \ve. 
	\end{align}
\end{lemma}

\begin{proof}
	Let $k \geq 0$ be such that $10^{-1}r < r_k \leq r.$ Let $(X,t) \in P_0$ be such that $(Z,t) = f(X,t)$ and let $(X_k,t) = f_k(X,t).$ By \eqref{e:distf_k} we have $\pd(Z,t,X_k,t) \leq C \ve r_k$ so that $B_r(Z,t) \subseteq B_{11r_k}(X_k,t)$ for $\ve$ small enough. At this point we let $P \in \verticalplanes$ be the plane through the point $(X_k,t)$ obtained from Lemma \ref{l:local-graph}. We claim $P$ satisfies \eqref{e:RF}. 
	
	Indeed, let $(Q,s) \in \Sigma \cap B_r(Z,t)$. Let $(Y,s) \in P_0$ be such that $(Q,s) = f(Y,s)$ and set $(Y_k,s) = f_k(Y,s).$ By \eqref{e:distf_k} we have $\pd(Q,s,Y_k,s) \leq C \ve r_k.$ Since $(Q,s) \in B_r(Z,t) \subseteq B_{11r_k}(X_k,t)$, this implies $(Y_k,s) \in \Sigma_k \cap B_{12r_k}(X_k,t).$ Since $(X_k,t) \in P$ and $\Sigma_k \cap B_{19r_k}(X_k,t)$ is the graph of a $\text{Lip}(1,\tfrac{1}{2},C\ve)$ function over $P$ by Lemma \ref{l:local-graph}, we have 
	\[\pd(Q,s,P) \leq \pd(Q,s,Y_k,s) + \pd(Y_k,s,P) \leq C \ve r_k.\]
	
	For the other direction, suppose $(Q,s) \in P \cap B_r(Z,t) \subseteq B_{11r_k}(X_k,t)$. By Lemma \ref{l:local-graph} there exists a point $(Y_k,s) \in \Sigma_k$ such that $\pdist(Q,s,Y_k,s) \leq C \ve r_k.$ Let $(Y,s) \in P_0$ be such that $(Y_k,s) = f_k(Y,s)$ and let $(W,s) \coloneqq f(Y,s).$ Then, by \eqref{e:distf_k}, 
	\[ \pd(Q,s,\Sigma) \leq \pd(Q,s,W,s) \leq \pd(Q,s,Y_k,s) + \pd(Y_k,s,W,s) \leq C \ve r_k\]
	as required.
\end{proof}

We finish with Theorem \ref{t:DT} \eqref{DT11}. We shall prove something slightly more general (which will be useful later on), which will immediately gives Theorem \ref{t:DT} \eqref{DT11} (since projections are 1-Lipschitz). 

\begin{lemma}\label{l:big-projection-RF}
	Let $\Sigma' \subseteq \R^{\adim+1}$ be $\ve$-Reifenberg flat. Then, provided $\ve$ is small enough, we have 
	\begin{align}
		\pHmeasure^{\idim+2}(\pi_P(\Sigma' \cap B_r(X,t))) \geq (1-C\ve)r^{\idim+2}
	\end{align}
	for all $(X,t) \in \Sigma'$, $0 < r < \infty$ and $P \in \verticalplanes$ satisfying $\bilatbeta_{\Sigma'}(B_r(X,t),P) \leq 2 \ve$. In particular, 
	\begin{align}
		\pHmeasure^{\idim+2}(\Sigma' \cap B_r(X,t)) \geq (1-C\ve)r^{\idim+2}
	\end{align}
	for all $(X,t) \in \Sigma'$ and $0 < r < \infty$. 
\end{lemma}

\begin{proof}
	Without loss of generality, we can assume 
	\[P = \R^{n} \times \{0\} \times \dots \times \{0\} \times \R \quad \mbox{ and } B_r(X,t) = B_1(0,0).\]
	Then, if $\pi$ denotes the projection onto $P$, it suffices to show 
	\begin{align}
		\pHmeasure^{\idim+2}( \pi(\Sigma' \cap B_1(0,0))) \geq 1-C\ve. 
	\end{align} 
	
	To prove this, we will apply Theorem \ref{t:DT}. For $k \geq 0$, let $\{(X_{j,k},t_{j,k})\}_{j \in J_k}$ be a maximal collection of points in $\Sigma' \cap B_2(0,0)$ satisfying \eqref{e:net} and $\{P_{j,k}\}_{j \in J_k}$ a collection of planes such that $\bilatbeta_{\Sigma'}(100B_{j,k},P_{j,k}) \leq \ve$. It follows from Lemma \ref{l:beta-angle} that $(\R^\idim,\{B_{j,k}\},\{P_{j,k}\})$ defines a $C\ve$-CCBP. Assume that $\ve$ is small enough such that Theorem \ref{t:DT} can be applied to $(\R^{\idim},\{B_{j,k}\},\{P_{j,k}\})$ and let $f \colon \R^{\idim+1} \to \R^{\adim+1}$ be the resulting map. Fix also a constant $C_1 > 1$ such that 
	\begin{align}\label{e:C_1}
		\pd((X,t) , f(X,t)) \leq C_1 \ve 
	\end{align}
	for all $(X,t) \in \R^{\idim}$, which exists by Theorem \ref{t:DT} \eqref{DT3}. 
	
	There exists $C_2 > 1$ such that 
	\begin{align}\label{e:sigma-contains-f}
		{\Sigma'} \cap B_{1}(0,0)\supseteq f(\R^{\idim+1} \cap B_{1-C_2\ve}(0,0))
	\end{align}
	Indeed, suppose towards a contradiction that there exists $(X,t) \in \R^{\idim+1} \cap B_{1-C_2\ve}(0,0)$ such that $f(X,t) \not\in {\Sigma'} \cap B_1(0,0)$, for some $C_2$ to be chosen large enough. By \eqref{e:C_1}, we have $f(X,t) \in B_1(0,0)$. Thus, it must be that $f(X,t) \in (f(\R^{\idim+1}) \setminus {\Sigma'}) \cap B_1(0,0)$. Let $k = k(X,t)$, with $k(X,t)$ as in Lemma \ref{l:locally-sigma-k}. By the argument there, we have 
	\begin{align}
		\pdist(X,t,{\Sigma'} \cap B_{1}(0,0)) \lesssim \ve r_k.
	\end{align}
	Thus, there exists a point $(X',t') \in {\Sigma'} \cap B_1(0,0)$ such that $\pd(X,t,X',t') \lesssim \ve r_k$. Since $\{(X_{j,k},t_{j,k})\}_{j \in J_k}$ is a maximal $r_{k}$-separated net in ${\Sigma'} \cap B_2(0,0)$, there exists $j \in J_{k}$ such that $(X',t') \in B_{j,k}$. However (for $\ve$ small enough), it then follows that $(X,t) \in 11B_{j,k}$, which contradict the maximality of $k$.  
	
	Define a map $\rho \colon \R^{\adim+1} \to \R^{\adim+1}$ by $\rho(X,t) = (X/(1+\Lambda),t)$, $\Lambda := C_2\ve/(1-C_2\ve)$ and set $g \coloneqq \rho \circ \pi \circ f$. Notice that $g$ is H\"older continuous (hence, continuous), preserves the $t$-component and $\pd((X,t) , g(X,t)) \leq C_3\ve$ for some $C_3 > 1$ and all $(X,t) \in \R^{\idim+1} \cap B_1(0,0)$. Thus, by Lemma \ref{l:topology},
	\begin{align}\label{e:g-contains-ball}
		g(\R^{\idim+1} \cap B_{1-C_2\ve}(0,0)) \supseteq \R^{\idim+1} \cap B_{1-C_4\ve}(0,0), 
	\end{align}
	where $C_4 \coloneqq 2(C_2 + C_3)$. To finish the proof, we use \eqref{e:sigma-contains-f} and \eqref{e:g-contains-ball}, along with the fact $\rho$ and $\pi$ are 1-Lipschitz, to get 
	\begin{align}
		\pHmeasure^{\idim+2}(\pi({\Sigma'} \cap B_1(0,0))) &\geq \pHmeasure^{\idim+2}(\pi \circ f(\R^{\idim+1} \cap B_{1-C_2\ve}(0,0)))\\
		&\geq (1-C_2\ve) \pHmeasure^{\idim+2}(g(\R^{\idim+1} \cap B_{1-C_2\ve}(0,0))) \\
		&\geq (1-C_2\ve) \pHmeasure^{\idim+2}(\R^{\idim+1} \cap B_{1-C_4\ve}(0,0)) \geq (1-C\ve)2^{\idim+2},
	\end{align}
	as required. 
\end{proof}

\section{Corona decompositions and the BWGL}\label{s:BWGL-corona}

In this section we will show the following (we also show the converse at the very end of the section).

\begin{theorem}\label{t:BWGL-Corona}
	If $E \subseteq \R^{\adim+1}$ is parabolic Ahlfors $(\idim+2)$-regular and satisfies the $\BWGL$ then it admits a corona decomposition. The packing constants in Definition \ref{d:corona-decomposition} can be chosen to depend only on the $A, \eta, \theta, n $, the regularity constant for $E$, and the packing constants given by the $\BWGL$. 
\end{theorem}

In fact, we will show something stronger than Theorem \ref{t:BWGL-Corona}. We will prove a corona-type decomposition in which the graphs bilaterally approximate the set inside the stopping-time regions and the approximations will be stated in a slightly different (and stronger) form than that in Definition \ref{d:corona-decomposition}. To state it, we need some notation. \\

For a system of Christ-David cubes $\cD$ (on some underlying set) and a stopping-time region $S \subseteq \cD,$ define the \textit{distance to the stopping-time region} $d_S \colon \R^{\adim+1} \to \R$ by 
\begin{align}
	d_S(X,t) \coloneqq \inf_{Q \in S} \left[ \ell(Q) + \pdist(X,t,Q) \right]. 
\end{align}
For a plane $P \in \verticalplanes,$ define a similar function $D_{S,P} \colon P \to \R$ by 
\begin{align}
	D_{S,P}(X,t) \coloneqq \inf_{(Y,t) \in \pi_P^{-1}(X,t)} d_S(Y,t) = \inf_{Q \in S} \left[\ell(Q) + \pdist(X,t,\pi_P(Q))  \right]. 
\end{align}
It will be useful to recall from Section \ref{s:notation} (see just below \eqref{e:E^t}) the notation $V_x$ for a plane $V \in \verticalplanes$.

\begin{theorem}\label{t:BWGL-bilat-Corona}
	Let $A > 1$ and $\eta , \theta > 0$. Suppose $E \subseteq \R^{\adim+1}$ is parabolic Ahlfors $(\idim+2)$-regular and satisfies the $\BWGL$, and $\cubes$ is a system of Christ-David cubes on $E$. Then, there is a coronization $(\cB,\cG,\cF)$ of $\cubes$ and, for each $S \in \cF$, a plane $P_S \in \cP$ with $(X_{Q(S)},t_{Q(S)}) \in P_S$ and a $\Lip(1,1/2,\eta)$ map $\psi_S \colon P_S \to (P_S)^\perp$ such that the following conditions hold. {First, 
    \begin{align}
       D_{S,P_{S}}(X,t)\left( |\nabla_X^2 \psi_S(X,t)| + |\partial_t \psi_S(X,t)| \right)\leq \theta  .
    \end{align}}
    Second, if $\Gamma_S$ denotes the graph of $\psi_S$ over $P_S$, then
	\begin{align}\label{e:corona-d}
		\pdist(X,t,\Gamma_S) \leq \theta d_S(X,t)
	\end{align}
	for all $(X,t) \in E \cap AB_{Q(S)}$. Finally, 
	\begin{align}\label{e:corona-D}
		\pdist((X,t) + \psi_S(X,t), E) \leq \theta D_{S,P_S}(X,t) 
	\end{align}
	for all $(X,t) \in P_S \cap AB_{Q(S)}$. The packing constants in Definition \ref{d:coronization} can be chosen to depend only on $A,\eta,\theta,n$, the regularity constant for $E$, and the packing constants given by the $\BWGL$. 
\end{theorem}

Theorem \ref{t:BWGL-Corona} follows easily from Theorem \ref{t:BWGL-bilat-Corona} since $d_S(X,t) \lesssim A\ell(Q)$ for any $Q \in S \in \cF$ and $(X,t) \in AB_{Q}$. Thus, the goal of this section will be to prove Theorem \ref{t:BWGL-bilat-Corona}. By Lemma \ref{p:graph-construction}, it suffices to construct a coronization for $\cubes$ such that every good cube is flat and the best approximating planes in each stopping-time region make a small angle with the top plane from that region. This coronization is built in stages as follows. In Proposition \ref{p:BWGL-RF-corona}, we show that any set satisfying the BWGL admits a corona decomposition by parabolic Ahlfors regular Reifenberg flat sets i.e., one replaces $\Lip(1,1/2)$ graphs in Definition \ref{d:corona-decomposition} by parabolic Ahlfors regular Reifenberg flat sets. In Proposition \ref{t:corona-RF-large}, we show that parabolic Ahlfors regular Reifenberg flat sets admit (bilateral) corona decompositions by $\Lip(1,1/2)$-graphs with possibly large constant. In Proposition \ref{t:corona-RF-small}, we improve this large constant condition by showing that Reifenberg flat sets admit a coronization in which the best approximating planes in any stopping-time region make small angle with the top plane. We use these facts along with transitivity type property of coronizations (Corollary \ref{c:corresponding-stopping-time}) to push this coronization back to the original set.

\subsection{Approximating BWGL by Reifenberg flat}
In this section, we show that any set satisfying the BWGL admits a corona-type decomposition by Reifenberg flat sets. The precise statement is the following.

\begin{proposition}\label{p:BWGL-RF-corona}
	Let $A,M > 1$ and $\ve > 0.$ Suppose $E \subseteq \R^{\adim+1}$ is parabolic Ahlfors $(\idim+2)$-regular set which satisfies the $\BWGL$, $\cubes$ is a system of Christ-David cubes on $E$, $(\cB,\cG,\cF)$ is a coronization of $\cD$ such that $\bilatbeta_E(MB_Q) < \ve$ for each $Q \in \cG$, and $\{P_Q\}_{Q \in \cG}$ is an assignment of planes such that 
	\begin{align}\label{e:flatness-condition-2}
		(X_Q,t_Q) \in P_Q \quad \mbox{and} \quad \bilatbeta_E(MB_Q,P_Q) \leq 2\ve. 
	\end{align} 
	Then, provided $M/A$ is large enough and $\ve$ is small enough, for each $S \in \cF$ there exists a parabolic Ahlfors $(\idim+2)$-regular $C\ve$-Reifenberg flat set $\Sigma_S \subseteq \R^{\adim+1}$ such that 
	\begin{align}\label{e:dist-to-Sigma_S}
		\pdist(X,t,\Sigma_S) \lesssim \ve M \ell(Q)
	\end{align}
	for all $Q \in S$ and $(X,t) \in E \cap AB_Q$, and 
	\begin{align}
		\pdist(X,t,E) \lesssim \ve M \ell(Q)
	\end{align}  
	for all $Q \in S$ and $(X,t) \in \Sigma_S \cap AB_Q$. The regularity constant for $\Sigma_S$ depends only on $n$ and the regularity constant for $E$. 
\end{proposition}

\begin{remark}
	To ease notation, we shall prove Proposition \ref{p:BWGL-RF-corona} only in the case that $A=3$. The general case is identical up to notational changes. 
\end{remark}

Let $M > 1$ (to be chosen sufficiently large) and $\ve > 0$ (to be chosen sufficiently small depending on $M$, see Remark \ref{r:first-constants}). Fix $E \subseteq \R^{\adim+1}$, $\cubes$, $(\cB,\cG,\cF)$ and $\{P_Q\}_{Q \in \cG}$ as in the statement of Proposition \ref{p:BWGL-RF-corona}, and fix some $S \in \cF$. We can assume without loss of generality that $Q(S) \in \cubes_0.$ The main tool for constructing the set $\Sigma_S$ will be Theorem \ref{t:DT}, for which we need to define a CCBP.\\

Let $P_S = P_{Q(S)}.$  Without loss of generality, we suppose that $\ell(Q(S))=1$.  For $k \geq 0,$ let $s(k)$ be such that $\Couter \rho^{s(k)} < r_k \leq \Couter \rho^{s(k)-1},$ with $\Couter$ as in Lemma \ref{cubes}. Let $\mathscr{X}_k = \{(X_{j,k},t_{j,k})\}_{j \in J_k}$ be a maximal net satisfying \eqref{e:net} in 
\begin{align}
	\mathscr{C}_k = \{ (X_Q,t_Q) \colon Q \in S \cap \cubes_{s(k)}\}.
\end{align} 
For $j \in J_k,$ let $B_{j,k} = B_{r_k}(X_{j,k},t_{j,k})$ and $P_{j,k} = P_Q,$ where $Q \in S \cap \cubes_{s(k)}$ is such that $(X_{j,k},t_{j,k}) = (X_Q,t_Q).$ For $M$ large enough, it follows from \eqref{e:flatness-condition-2} that
\begin{align}\label{e:close-to-P}
	\pdist(X,t,P_{j,k}) \lesssim  \ve M r_k \quad \mbox{and} \quad \pdist(Y,s,E) \lesssim \ve M r_k
\end{align}
for all $(X,t) \in E \cap 100B_{j,k}$ and $(Y,s) \in P_{j,k} \cap 100B_{j,k}, \ j \in J_k.$ 

\begin{lemma}\label{l:CCBP}
	For $M$ large enough, the triple $(P_S,\{B_{j,k}\},\{P_{j,k}\})$ defines a $\CCCBP \ve$-CCBP for some constant $\CCCBP \sim M$. 
\end{lemma}

\begin{proof}
	We start with Definition \ref{d:CCBP} \eqref{CCBP1}. Let $k \geq 0$ and $j \in J_k.$ If $s(k) = s(k-1)$ then $\mathscr{C}_k = \mathscr{C}_{k-1},$ so that $(X_{j,k},t_{j,k}) \in \mathscr{C}_{k-1}.$ Since $\mathscr{X}_{k-1}$ is a maximal net in $\mathscr{C}_{k-1}$ satisfying \eqref{e:net}, there exists $i \in J_{k-1}$ such that $(X_{j,k},t_{j,k}) \in 2B_{i,k-1}.$ Suppose instead that $s(k) > s(k-1),$ so that $s(k-1) = s(k) - 1.$ Let $Q \in S \cap \cubes_{s(k)}$ be such that $(X_{j,k},t_{j,k}) = (X_Q,t_Q).$ Then $R \coloneqq \Parent(Q) \in \cubes_{s(k)-1} = \cubes_{s(k-1)}$ so there exists $i \in J_{k-1}$ such that $(X_{R},t_{R}) \in B_{i,k-1}.$ Since 
	\[ \ell(R) = C_1\rho^{s(k)-1} = C_1\rho^{s(k-1)} \leq r_{k-1},\]
	we have 
	\begin{align}
		(X_{j,k},t_{j,k}) \in R &\subseteq B_{R} \subseteq B_{r_{k-1} + \ell(R)}(X_{i,k-1},t_{i,k-1}) \\
		&\subseteq B_{2r_{k-1}}(X_{i,k-1},t_{i,k-1}) = 2B_{i,k-1}. 
	\end{align}
	
	It is easy to check that $J_0 = \{j_0\}$ where $(X_{j_0,0},t_{j_0,0}) = (X_{Q(S)},t_{Q(S)}).$ Definition \ref{d:CCBP} \eqref{CCBP2} is then immediate since $P_{j_0,0} = P_S.$  
	
	We verify Definition \ref{d:CCBP} \eqref{CCBP3} and \eqref{CCBP4} simultaneously. Let $i \in J_k$ and let $j \in J_k$ (resp. $J_{k+1}$) satisfy the hypothesis of Definition \ref{d:CCBP} \eqref{CCBP3} (resp. \eqref{CCBP4}). Let $Q \in S \cap \cubes_{s(k)}$ such that $(X_Q,t_Q) = (X_{i,k},t_{i,k})$ and let $R \in S \cap \cubes_{s(k)}$ (resp. $S \cap \cubes_{s(k+1)}$ such that $(X_R,t_R)  = (X_{j,k},t_{j,k})$ (resp. $(X_{j,k+1},t_{j,k+1})$). Since $\pd(X_Q,t_Q,X_R,t_R) \leq 100r_k,$ and $\ell(R) \sim \ell(Q) \sim r_k,$ we have 
	\begin{align}
		20B_{i,k} \subseteq 100B_{i,k} \subseteq MB_R
	\end{align}
	for $M$ large enough. It now follows from Lemma \ref{l:beta-angle} and the first inequality in \eqref{e:flatness-condition-2} that  
	\[ d_{20B_{i,k}}(P_Q,P_R) \lesssim d_{100B_{i,k}}(P_Q,P_R) \lesssim M \ve, \]
	as required. 
\end{proof}

\begin{remark}\label{r:first-constants}
	For here on we will assume $M$ is large enough such that \eqref{e:close-to-P} and Lemma \ref{l:CCBP} hold, and that $\ve > 0$ is chosen sufficiently small (depending on $M$) such that $\CCCBP\ve < \ve_0$ (with $\CCCBP$ as in Lemma \ref{l:CCBP} and $\ve_0$ as in Theorem \ref{t:DT}) and such that $100C\CCCBP\ve < 1$ for any implicit constant $C$ in the statement of Theorem \ref{t:DT}. 
\end{remark}

 Let $\Sigma_S$ be the set constructed by applying Theorem \ref{t:DT} with $(P_S,\{B_{j,k}\},\{P_{j,k}\})$ and let $\Sigma_{S,k}$ denote the sequence of approximating sets as in Theorem \ref{t:DT} \eqref{DT5}. The fact that $\Sigma_S$ is $C\ve$-Reifenberg flat is immediate from Theorem \ref{t:DT} \eqref{DT7}. Before we can check that $\Sigma_S$ is parabolic Ahlfors $(\idim+2)$-regular, we need the following. 
 
\begin{lemma}\label{l:locally-sigma-k}
 	For $(X,t) \in \Sigma_S \setminus E$, let $k=k(X,t)$ be the largest integer with $(X,t) \in V_{k-1}^{11}$. Then, 
 		\begin{align}\label{e:sigma^N to sigma}
 		\pdist(X,t,E) \lesssim \CCCBP \ve r_{k(X,t)}
 	\end{align}
 	and 
 	\begin{align}\label{e:local-sigma_k}
 		\Sigma _S \cap B_{r_{k(X,t)}}(X,t) = \Sigma_{S,k(X,t)} \cap B_{r_{k(X,t)}}(X,t) . 
 	\end{align}
\end{lemma}

\begin{proof}
	Fix some $(X,t) \in \Sigma_S \setminus E$ and let $k \coloneqq k(X,t)$. By maximality of $k$, we have $(X,t) \notin V^{11}_k$. Thus, $B_{r_k}(X,t) \subseteq \R^{\adim+1} \setminus V_l^{10}$ for all $\ell \geq k$. Equation \eqref{e:local-sigma_k} then follows from Theorem \ref{t:DT} \eqref{DT4}. 
	
	We now show \eqref{e:sigma^N to sigma}. For this, let $(X_0,t) \in P_S$ such that $(X,t) = f(X_0,t)$ and, for $m \geq 0$, let 
	\begin{align}
		(X_m,t) = f_m(X_0,t) \in \Sigma_{S,m}. 
	\end{align}
	Let $j \in J_{k-1}$ such that $(X,t) \in 11B_{j,k-1}$, which exists by the definition of $k$. Iterating Theorem \ref{t:DT} \eqref{DT9} and summing over a geometric series (recalling also Remark \ref{r:first-constants}), we have 
	\begin{align}
		\pd(X,t,X_{k-1},t) \lesssim \CCCBP \ve r_k,
	\end{align}
	so that $(X_{k-1},t) \in \Sigma_{S,k-1} \cap 12B_{j,k-1}$. Thus, by \eqref{Sigma_kPlane}, there exists a point $(X_{k-1}',t') \in P_{j,k-1} \cap 13B_{j,k-1}$ such that 
	\begin{align}
		\pd(X_{k-1},t,X_{k-1}',t') \lesssim \CCCBP \ve r_k.
	\end{align}
	 Now, applying the second inequality in \eqref{e:close-to-P}, there exists a further point $(X_{k-1}'',t'') \in E$ such that 
	\begin{align}
		\pd(X_{k-1}',t',X_{k-1}'',t'') \lesssim \CCCBP \ve r_k.
	\end{align}
	Inequality \eqref{e:sigma^N to sigma} now follows from the above three estimates and the triangle inequality. 	
\end{proof}
 
 \begin{lemma}
 	The set $\Sigma_S$ is $C\ve$-Reifenberg flat and Ahlfors $(\idim+2)$-regular with regularity constant depending only on $n$ and the regularity constant for $E$, which we shall denote by $C_0$. 
 \end{lemma}
 
\begin{proof}
	Reifenberg flatness and lower regularity of $\Sigma_S$ are immediate from Theorem \ref{t:DT} \eqref{DT7} and \eqref{DT11}, respectively. Thus, we only need to check the upper regularity. For this, fix some $(X,t) \in \Sigma_S$ and $r > 0.$ 
	
	Recall the definition of $k(Y,s)$ from Lemma \ref{l:locally-sigma-k} and $\{(X_i,t_i)\}_{i \in I}$ be a maximal collection of points in $(\Sigma_S \setminus E) \cap B_r(X,t)$ such that 
	\begin{align}\label{e:separated}
		\pd(X_i,t_i,X_j,t_j) \geq \max\{r_{k(X_i,t_i)}, r_{k(X_j,t_j)}\}/2 \quad \mbox{ for all distinct } i,j \in I.
	\end{align}
	For brevity, we will write $r_{k_i} = r_{k(X_i,t_i)}$ and $B_i = B_{r_{k_i}}(X_i,t_i).$ By maximality, we have 
	\begin{align}\label{e:ur1}
		(\Sigma_S \setminus E) \cap B_r(X,t)\subseteq \bigcup_{i \in I} (\Sigma_S \cap B_i). 
	\end{align}
	
	For $i \in I$, let $(X_i',t_i')$ be a point in $E$ such that $\pdist(X_i,t_i,E) = \pd(X_i,t_i,X_i',t_i')$ and define $B_i' = B_{r_{k_i}/32}(X_i',t_i').$ We claim that 
	\begin{align}\label{e:ur2}
		B_i' \cap B_j' =\emptyset \mbox{ for } i \not=j.
	\end{align}
	Indeed, suppose towards a contradiction that there exists $(Z,u) \in B_i' \cap B_j'.$ By \eqref{e:sigma^N to sigma}, $\pd(X_i,t_i,\Sigma) \lesssim \CCCBP  \ve r_{k_i}$ for each $i \in I.$ Thus (recalling again Remark \ref{r:first-constants}),
	\begin{align}
		\pd(X_i,t_i,X_j,t_j) &\leq \pd(X_i,t_i,X_i',t_i') + \pd(X_i',t_i',Z,u)  \\
		&\hspace{4em} + \pd(Z,u,X_j',t_j') + \pd(X_j',t_j',X_j,t_j) \\
		&\leq C \ve r_{k_i} + r_{k_i}/16 + r_{k_j}/16 + C\ve r_{k_j} \leq \max\{r_{k_i},r_{k_j}\}/4,
	\end{align}
	which contradicts \eqref{e:separated}.  
	
	By Theorem \ref{t:DT} \eqref{DT6} and \eqref{e:local-sigma_k}, we have that $\Sigma_{S} \cap B_i$ is the graph of some $\text{Lip}(1,1/2,C\ve)$ function. Hence, 
	\begin{align}\label{e:ur3}
		\pHmeasure^{\idim+2}(\Sigma_S \cap B_i) \lesssim r_{k_i}^{\idim+2}.
	\end{align}
	Using \eqref{e:ur1}, \eqref{e:ur2} and \eqref{e:ur3} with the fact that $E$ is parabolic Ahlfors $(\idim+2)$-regular, we have 
	\begin{align}
		\pHmeasure^{\idim+2}((\Sigma_S\setminus
         E) \cap B_r(X,t) ) &\leq \sum_{i \in I} \pHmeasure^{\idim+2}(\Sigma_S \cap B_i) \lesssim \sum_{i \in I} r_{k_i}^{\idim+2} \\
		&\lesssim_{C_0} \sum_{i \in I} \pHmeasure^{\idim+2}(E\cap B_i') \\
		&\leq \pHmeasure^{\idim+2}(E \cap B_{2r}(X,t)) \lesssim_{C_0} r^{\idim+2}.
	\end{align}
	Thus, 
	\begin{align}\pHmeasure^{\idim+2}(\Sigma_S \cap B_r(X,t)) &\leq \pHmeasure^{\idim+2}(\Sigma_S \cap E \cap B_r(X,t)) + \pHmeasure^{\idim+2}((\Sigma_S \setminus E) \cap B_r(X,t) ) \\
	&\lesssim_{C_0} r^{\idim+2}.
\end{align}
as required. 
\end{proof}

The following finishes the proof of Proposition \ref{p:BWGL-RF-corona}. 

\begin{lemma}
	For each $Q \in S$, $(X,t) \in E \cap 3B_Q$ and $(Y,s) \in \Sigma_S \cap 3B_Q$ we have 
	\begin{align}\label{e:proof-of-distances}
		\pdist(X,t,\Sigma_S) \lesssim  \ve M \ell(Q) \quad \mbox{ and } \quad  \pdist(Y,s,E) \lesssim  \ve M \ell(Q). 
	\end{align}
\end{lemma}

\begin{proof}
	Fix some $Q \in S$. Let $k \geq 0$ be such that $(X_Q,t_Q) \in \mathscr{C}_k$, let $j \in J_k$ such that $(X_Q,t_Q) \in B_{j,k}$ (which exists by maximality) and let $R \in \cubes_{s(k)}$ be the cube associated to $(X_{j,k},t_{j,k}).$ Since $\ell(Q) < r_k,$ it follows from the triangle inequality that 
	\begin{align}\label{e:ball-containment}
		3B_Q \subseteq 5B_{j,k}.
	\end{align}
	
	To prove the left-most inequality in \eqref{e:proof-of-distances} we proceed as follows. Fix some $(X,t) \in E \cap 3B_Q$. By \eqref{e:close-to-P} and \eqref{e:ball-containment}, there exists $(X',t') \in P_{j,k} \cap 6B_{j,k}$ such that $\pd(X,t,X',t') \lesssim \CCCBP \ve r_k \lesssim \CCCBP \ve \ell(Q).$ By Theorem \ref{t:DT} \eqref{DT5}, since $(X',t') \in P_{j,k} \cap 6B_{j,k}$, there is $(X'',t'') \in \Sigma_{S,k} \cap 7B_{j,k}$ such that $\pd(X',t',X'',t'') \lesssim \CCCBP \ve r_k \lesssim \CCCBP \ve \ell(Q).$ Finally, by Theorem \ref{t:DT} \eqref{DT8}, there exists $(X''',t''') \in \Sigma_S \cap 8B_{j,k}$ such that $\pd(X'',t'',X''',t''') \lesssim \CCCBP \ve r_k \lesssim \ve \ell(Q).$ Combining the previous three estimates with the triangle inequality finishes the proof of the left-most inequality. 
	 
	 The argument for the right-most inequality is similar but in the reverse order. Fix some $(Y,s) \in \Sigma_S \cap 3B_{Q}$. By Theorem \ref{t:DT} \eqref{DT8} and \eqref{e:ball-containment}, there exists $(Y',s') \in \Sigma_{S,k} \cap 6B_{j,k}$ such that $\pd(Y,s,Y',s') \lesssim \CCCBP \ve r_k$. By Theorem \ref{t:DT} \eqref{DT5}, since $(Y',s') \in \Sigma_{S,k} \cap 6B_{j,k}$, there exists $(Y'',s'') \in P_{j,k} \cap 7B_{j,k}$ such that $\pd(Y',s',Y'',s'') \lesssim \CCCBP \ve r_k$. It then follows from \eqref{e:close-to-P} that there exists $(Y''',s''') \in E$ such that $\pd(Y'',s'',Y''',s''') \lesssim \CCCBP \ve r_k$. Combining the previous three estimates with the triangle inequality and the fact that $\CCCBP \sim M$ finishes the proof.  
\end{proof}

\subsection{Approximating Reifenberg flat by Lipschitz graphs with large constant} We show here that Reifenberg flat sets admit corona-type decompositions similar to those stated in Theorem \ref{t:BWGL-bilat-Corona}. The statement will be slightly different than Theorem \ref{t:BWGL-bilat-Corona} in the sense that the Lipschitz is possibly quite large and approximation parameter ($\theta$ in Theorem \ref{t:BWGL-bilat-Corona}) cannot be chosen too small with respect to the Reifenberg parameter. The precise statement is the following. 

\begin{proposition}\label{t:corona-RF-large}
	Let $A,M > 1$ and $\ve,\theta > 0$. Suppose $\Sigma \subseteq \R^{\adim+1}$ is parabolic Ahlfors $(\idim+2)$-regular and $\ve$-Reifenberg flat, $\cubes$ is a system of Christ-David cubes on $\Sigma$ and $\{P_Q\}_{Q \in \cubes}$ is an assignment of $t$-independent planes such that 
	\begin{align}
		(X_Q,t_Q) \in P_Q \quad \mbox{and} \quad \bilatbeta_\Sigma(MB_Q,P_Q) \leq 2 \ve .
	\end{align} 
	Then, provided $M/A$ is sufficiently large and $\ve = \ve(A,M,\theta)$ is sufficiently small, there is a constant $L >1$, a coronization $(\cB,\cG,\cF)$ of $\cubes$ and, for each $S \in \mathscr{F}$, a $\Lip(1,1/2,L)$ map $\psi_S \colon P_{Q(S)} \to (P_{Q(S)})^\perp$ such that the following conditions hold. {First, 
    \begin{align}
       D_{S,P_{Q(S)}}(X,t)\left( |\nabla_X^2 \psi_S(X,t)| + |\partial_t \psi_S(X,t)| \right)\leq \theta  .
    \end{align}}
    Second, if $\Gamma_S$ denotes the graph of $\psi_S$ over $P_{Q(S)}$, then 
	\begin{align}
		\pdist(X,t,\Gamma_S) \leq \theta d_S(X,t)
	\end{align}
	for all $(X,t) \in E \cap AB_{Q(S)}$, and 
	\begin{align}
		\pdist((X,t) + \psi_S(X,t), E) \leq \theta D_{S,P_{Q(S)}}(X,t) 
	\end{align}
	for all $(X,t) \in P_{Q(S)} \cap AB_{Q(S)}$. The packing constants in Definition \ref{d:coronization} can be chosen to depend only on $n$ and the regularity constant for $\Sigma$. 
\end{proposition}

Lemma \ref{p:graph-construction} below will be important for the proof of Proposition \ref{t:corona-RF-large} (and Proposition \ref{t:corona-RF-small} later). It gives a way to construct Lipschitz graphs in stopping-time regions where every cube is flat and the best approximating planes in some stopping-time region make a small angle with the top plane in that region. The first part (assuming only the unilateral condition \eqref{e:unilat-approx}) is proved in \cite[Section 4]{bortz2023corona} in the codimension 1 case, based on the original (Euclidean) argument of David and Semmes \cite{david1991singular}. A version of the second part (assuming the bilateral condition \eqref{e:bilat-approx}) also appears in \cite{bortz2023corona}, however, the conclusion is slightly weaker than what we need here. We note that, in \cite{bortz2023corona}, it is further needed that $\alpha$ be chosen sufficiently small. This is only to utilise the small angle approximation for the sine function, which is then used to show that the Lipschitz constant is comparable to $\alpha.$ The proof still works without this smallest condition, however, it produces instead a Lipschitz graphs with some constant that depends on $\alpha.$ We have included some details in Section \ref{s:graphs} regarding the (minor) adjustments that need to be made in the large angle and general co-dimension case. We prove \eqref{e:close-to-E} in the same section.

 For the statement of Lemma \ref{p:graph-construction}, let us define $L \colon (0,1) \to \R$ by 
\begin{align}\label{d:L}
	L(\alpha) \coloneqq \frac{4\alpha}{\sqrt{1-\alpha^2}}.
\end{align}

\begin{lemma}\label{p:graph-construction}
	Let $A,M > 1$ and $\alpha, \ve \in (0,1).$ Let $\Sigma \subseteq \R^{\adim+2}$ be parabolic Ahlfors $(\idim+2)$-regular, $\cubes$ be a system of Christ-David cubes on $\Sigma$, $S \subseteq \cubes$ a stopping-time region, and suppose $\{P_Q\}_{Q \in S}$ is an assignment of $t$-independent planes such that 
	\begin{align}\label{e:unilat-approx}
		(X_Q,t_Q) \in P_Q, \quad \unilatbeta_\Sigma(MB_Q,P_Q) \leq \ve \quad \mbox{ and } \quad \angle(P_Q,P_{Q(S)}) \leq \alpha
	\end{align}
	for all $Q \in S.$ Then, provided $M/A$ is sufficiently large and $\ve = \ve(M,\alpha)$ is sufficiently small, there exists a $\Lip(1,1/2,L(\alpha))$ map $\psi_S \colon P_{Q(S)} \to (P_{Q(S)})^\perp$ such that the following conditions hold. {First, 
    \begin{align}
       D_{S,P_{Q(S)}}(X,t)\left( |\nabla_X^2 \psi_S(X,t)| + |\partial_t \psi_S(X,t)| \right)\lesssim \ve  .
    \end{align}}
    Second, if $\Gamma_S$ is the graph of $\psi_S$ over $P_{Q(S)}$, then 
	\begin{align}\label{e:E-close-to-gamma}
		\pdist(X,t,\Gamma_S) \lesssim \ve M d_S(X,t)
	\end{align}
	for all $(X,t) \in E \cap AB_{Q(S)}.$ Finally, if 
	\begin{align}\label{e:bilat-approx}
		\bilatbeta_\Sigma(MB_Q,P_Q) \leq \ve 
	\end{align}
	for all $Q \in S$, then 
	\begin{align}\label{e:close-to-E} 
		\pdist((X,t) + \psi_S(X,t),E) \lesssim \ve M D_{S,P_{Q(S)}}(X,t)
	\end{align} 
	for all $(X,t) \in P_{Q(S)} \cap AB_{Q(S)}$. 
\end{lemma}

We can now commence with the proof of Proposition \ref{t:corona-RF-large}. Fix constants $A,M > 1$ and $\ve,\theta > 0,$ and let $\Sigma,\cubes$ and $\{P_Q\}_{Q \in \cubes}$ be as in the statement of Proposition \ref{t:corona-RF-large}. We will construct a coronization $(\cB(Q_0),\cG(Q_0),\cF(Q_0))$ of $\cubes(Q_0)$ satisfying the conclusions of Proposition \ref{t:corona-RF-large} for each $Q_0 \in \cD$. It is well known (see \cite[Page 38]{david1991singular}) that this \textit{local} coronization for each $Q_0$ is sufficient for the \textit{global} corona decomposition as in Proposition \ref{t:corona-RF-large}. 

Let $\alpha \in (0,1)$ (to be chosen sufficiently close to 1). Let $Q_0 \in \cubes$ and assume without loss of generality that $Q_0 \in \cD_0.$ By Lemma \ref{p:graph-construction}, it suffices to show the following. 

\begin{lemma}\label{l:local-corona}
	There exists a coronization $(\cB(Q_0),\cG(Q_0),\cF(Q_0))$ of $\cubes(Q_0)$ such that 
	\begin{align}\label{e:sufficient-condition}
		\angle(P_Q,P_{Q(S)}) \leq \alpha 
	\end{align}
	for all $Q \in S \in \cF(Q_0)$.
\end{lemma} 

Proving Lemma \ref{l:local-corona} is the goal for the remainder of the subsection. We begin by setting 
\[\cB(Q_0) = \emptyset \quad \mbox{and} \quad \cG(Q_0) = \cubes(Q_0).\]
Thus, the only things left to do are to define are the stopping-time regions and prove the packing condition for $\{Q(S)\}_{S \in \cF(Q_0)}$

\begin{definition}\label{d:small-angle}
	For each $Q \in \cubes$, let $S_Q$ be the stopping-time region constructed inductively as follows. First, add $Q$ to $S_Q.$ Then, add sub-cubes $R \subseteq Q$ if $\Parent(R) \in S_Q$ and 
	\begin{align}
		\angle(P_{R'} , P_Q) \leq \alpha
	\end{align}
	for each $R' \in \Sibling(R).$ 
\end{definition}

We construct $\cF(Q_0)$ inductively as follows. 

\begin{definition}
	Let 
	\begin{align}
		\Top_0 = \{Q_0\} \quad \mbox{and} \quad \cF_0(Q_0) = \left\{S_{Q_0}\right\}. 
	\end{align}
	Then, assuming $\Top_k$ and $\cF_k(Q_0)$ have been defined up to some $k \geq 0,$ let 
	\begin{align}
		\Top_{k+1} = \bigcup_{Q \in \Top_k} \bigcup_{T \in \min(S_Q)} \Child(T) \quad \mbox{and} \quad \cF_{k+1}(Q_0) = \left\{S_Q \colon Q \in \Top_{k+1}\right\}.
	\end{align}
	Where $\min(S_Q)$ are the minimal cubes of $S_Q$.  Finally, set  
	\begin{align}
		\cF(Q_0) = \bigcup_{k=0}^\infty \cF_k(Q_0). 
	\end{align}
\end{definition}

It is immediate from Definition \ref{d:small-angle} that the stopping-time regions $\cF(Q_0)$ satisfy \eqref{e:sufficient-condition}. We will now check the Carleson condition for $\{Q(S)\}_{S \in \cF(Q_0)}$. For this we need some preliminary lemmas. The first of these says that each cube has large projection onto its corresponding plane, and is an immediate consequence of Lemma \ref{l:big-projection-RF} and Lemma \ref{cubes} \eqref{cubes2}.

\begin{lemma}\label{l:RF-BP}
	For $\ve > 0$ small enough there is a constant $\Cproj > 0$ such that 
	\begin{align}
		\pHmeasure^{\idim+2}(\pi_Q(Q)) \geq \Cproj \pHmeasure^{\idim+2}(Q)
	\end{align}
	for all $Q \in \cubes$, where $\pi_Q$ denotes the projection onto the plane $P_Q$.
\end{lemma}

The next lemma says that each cube whose best approximating planes makes large angle with the top plane must have small orthogonal projections onto that top plane. 

\begin{lemma}\label{l:big-angle-small-proj} 
	Provided $M$ and $\alpha = \alpha(\Cproj)$ are sufficiently large and $\ve = \ve(M,\Cproj)$ is sufficiently small, we have
	\begin{align}\label{e:big-angle-small-proj}
		\pHmeasure^{\idim+2}(\pi_{Q(S)}(Q)) \leq \frac{\Cproj}{2}\pHmeasure^{\idim+2}(Q), 
	\end{align}
	for each $S \in \cF$ and $Q \in \min(S)$, where $\pi_{Q(S)}$. 
\end{lemma}

\begin{proof}
	Fix some $S \in \cF(Q_0)$, $Q \in \min(S)$ and let $\delta > 0$ (to be chosen sufficiently small depending on $\Cproj$). By construction, there exists $R \in \Child(Q)$ such that 
	\begin{align}
		\angle(P_R,P_{Q(S)}) \geq \alpha. 
	\end{align}
	Thus, for $\alpha = \alpha(\delta)$ large enough, there is a $(d-1)$-dimensional $t$-independent plane $L \subset P_{Q(S)}$ such that 
	\begin{align}
		\pi_{Q(S)}(P_R) \subseteq \cN_{\delta \ell(Q)}(L),
	\end{align}
	where $\cN_{\delta \ell(Q)}(L)$ denotes the $\delta \ell(Q)$-neighbourhood of $L$. Since $R \in \Child(Q)$, there exists some absolute constant $C > 1$ such that $Q \subseteq (C/2)B_R$. In particular, if $M$ is sufficiently large and $\ve = \ve(M)$ is sufficiently small, for each $(X,t) \in Q$ there exists $(X',t') \in P_R \cap CB_R$ such that $\pd(X,t,X',t') \lesssim \ve M\ell(R) \lesssim \ve M\ell(Q).$ In particular, it follows that 
	\begin{align}
		\pi_{Q(S)}(Q) \subseteq B_{C\ell(R)}(\pi_{Q(S)}(X_R,t_R)) \cap \cN_{(\delta + CM\ve)\ell(Q)}(L). 
	\end{align}
	Then, for $\ve = \ve(M,\Cproj)$ and $\delta = \delta(\Cproj)$ sufficiently small, we have 
	\begin{align}
		\pHmeasure^{\idim+2}\left(B_{C\ell(R)}(\pi_{Q(S)}(X_R,t_R)) \cap \cN_{(\delta + CM\ve)}(L)\right) \lesssim (\delta + CM\ve)^{n-1}\ell(Q)^{\idim+2} \leq  \frac{\Cproj}{2}\surfacemeasure(Q),
	\end{align}
	which finishes the proof. 
\end{proof}

The following completes the proof of Lemma \ref{l:local-corona}. 

\begin{lemma}
	With $M,\alpha$ and $\ve$ as in Lemma \ref{l:big-angle-small-proj}, we have 
	\begin{align}
		\sum_{\substack{S \in \cF(Q_0) \\ Q(S) \subseteq R_0}} \ell(Q(S))^{\idim+2} \lesssim \ell(R_0)^{\idim+2}
	\end{align}
	for all $R_0 \subseteq Q_0.$ 
\end{lemma}

\begin{proof}
	Fix some $S \in \cF(Q_0)$ for the moment. By Lemma \ref{l:big-angle-small-proj} we have 
	\begin{align}
		\pHmeasure^{\idim+2}(\pi_{Q(S)}(R)) \leq \frac{\Cproj}{2}\surfacemeasure(R)
	\end{align}
	for all $R \in \min(S)$. Recalling the definition of $z(S)$ from \eqref{d:z}, this means that 
	\begin{align}
		\surfacemeasure(z(S)) &\geq \pHmeasure^{\idim+2}(\pi_{Q(S)}(z(S))) \geq \pHmeasure^{\idim+2}(\pi_{Q(S)}(Q(S))) - \pHmeasure^{\idim+2}\left(\pi_{Q(S)}\left( \bigcup_{R \in \min(S)} R \right)\right) \\
		&\geq \Cproj\pHmeasure^{\idim+2}(Q(S)) - \sum_{R \in \min(S)} \pHmeasure^{\idim+2}(\pi_{Q(S)}(R)) \\
		&\geq \Cproj\pHmeasure^{\idim+2}(Q(S))- \frac{\Cproj}{2}\sum_{R \in \min(S)} \pHmeasure^{\idim+2}(R) \gtrsim \Cproj \pHmeasure^{\idim+2}(Q(S)). 
	\end{align}
	Since $\{z(S)\}_{S \in \cF(Q_0)}$ is a collection of mutually disjoint subset of $\Sigma$, we conclude that 
	\begin{align}
		\sum_{\substack{S \in \cF(Q_0) \\ Q(S) \subseteq R_0}} \ell(Q(S))^{\idim+2} \lesssim \sum_{\substack{S \in \cF(Q_0) \\ Q(S) \subseteq R_0}} \pHmeasure^{\idim+2}(z(S)) \leq \pHmeasure^{\idim+2}(R_0) \lesssim \ell(R_0)^{\idim+2},
	\end{align}
	as required. 
\end{proof}

\subsection{Approximating Reifenberg flat by Lipschitz graphs with small constant}

In this subsection, we show how to construct a coronization for Ahlfors regular Reifenberg flat sets in which, for each stopping-time region, the best approximating planes make small angles with the top plane in that region. Lemma \ref{p:graph-construction} would provide us with a graph for each stopping-time in the coronization, however, this statement in terms of planes will be more useful to us in the next subsection.

\begin{proposition}\label{t:corona-RF-small}
	Let $M > 1$ and $\ve,\eta > 0$. Suppose $\Sigma \subseteq \R^{\adim+1}$ is parabolic Ahlfors $(\idim+2)$-regular and $\ve$-Reifenberg flat, $\cubes$ is a system of Christ-David cubes on $\Sigma$ and $\{P_Q\}_{Q \in \cubes}$ is an assignment of planes such that
	\begin{align}
		(X_Q,t_Q) \in P_Q \quad \mbox{ and } \quad \bilatbeta_\Sigma(MB_Q) \leq 2\ve . 
	\end{align} 
	Then, provided $M$ is sufficiently large and $\ve = \ve(M,\eta)$ is sufficiently small, there is a coronization $(\cB,\cG,\cF)$ of $\cubes$ such that
	\begin{align}\label{e:corona-RF-small}
		\angle(P_Q,P_{Q(S)}) \leq \eta 
	\end{align}
	for each $Q \in S \in \cF$. The packing constants in Definition \ref{d:coronization} can be to chosen to depend only on $M,\eta,\adim,\idim$ and the regularity constant for $\Sigma$.
\end{proposition}

Before starting the proof of Proposition \ref{t:corona-RF-small}, we need a preliminary result. Lemma \ref{p:carleson-corona} below is a parabolic version of Carleson's corona construction for $L^\infty$ functions (Lemma \ref{p:carleson-corona}). A proof of this for Euclidean functions can be found in \cite[Section IV.2.2]{david1993analysis}. The proof given there is rather general and (up to minor notational changes) can be easily adapted to our setting here. \\

\begin{lemma}\label{p:carleson-corona}
	Let $P \in \verticalplanes$ and let $\Delta$ be a system of dyadic parabolic cubes of $P$. For each $h \in L^\infty(P;P^\perp)$, $\|h\|_\infty \leq 1$, and $\eta > 0$, there is a coronization $(\cB,\cG,\cF)$ of $\Delta$ such that for each $S \in \cF$ and each $I \in S$
	\begin{align}
		\left| \avint_I h \, d\mu -  \avint_{I(S)} h \, d\mu \right| \leq \eta 
	\end{align}
	where $I(S)$ denotes the top cube in $S$ and $\mu$ is the measure defined by 
		\begin{align}
			\slicemeasure(E) = \int_{\R} \int_{\R^\adim} \mathds{1}_E(Y,s) \, d\Hmeasure^{\idim-1}(Y) ds = \int_\R \Hmeasure^{\idim-1}(E^s) \, ds. 
		\end{align}
	The packing constants can be chosen such that they depend on $n$ and $\eta$ but not on $h.$ 
\end{lemma}

\begin{remark}\label{r:slice-measure}
	The measure $\mu$ is closely related to $\pHmeasure^{\idim+2}$. Indeed, if $E \subseteq \R^{\adim+1}$ then it is known (see for example \cite[Proposition B.1]{bortz2023corona}) that $\slicemeasure(A) \lesssim \surfacemeasure(A)$ for all $A \subseteq E.$ If $E$ is parabolic Ahlfors $(\idim+2)$-regular then the reverse inequality $\sigma(A) \lesssim \mu(A)$ is also true and can be found in the same reference.
\end{remark}

	Now, fix some $M > 1$ and $\ve,\eta > 0,$ and let $\Sigma$ and $\cubes$ be as in the statement of Proposition \ref{t:corona-RF-small}. Since the statement of Proposition \ref{t:corona-RF-small} becomes stronger for smaller $\eta$, we will allow ourselves to choose $\eta$ as small as we need. Moreover, by adjusting the parameters slightly, it suffices to construct a coronization satisfying \eqref{e:corona-RF-small} with right-hand side equal to $C\eta$ for some absolute constant $C > 1.$

	Let $A > 1$ to be chosen sufficiently large. By Proposition \ref{t:corona-RF-large} (assuming $M/A$ is large enough and $\ve = \ve(A,M,\eta)$ is small enough), there is a constant $L > 1$ and a coronization $(\cB',\cG',\cF')$ of $\cubes$ such that the following holds. For each $S' \in \cF'$ there is a plane $P_{S'} \in \verticalplanes$ containing $(X_{Q(S')},t_{Q(S')})$ and a $\Lip(1,1/2,L)$ function $\psi_{S'} \colon P_{S'} \to (P_{S'})^\perp$ such that, if $\Gamma_{S'}$ denotes the graph of $\psi_{S'}$ over $P_{S'}$, then
	\begin{align}\label{e:close-to-gamma-proof}
		\pdist(X,t,\Gamma_{S'}) \leq \eta d_{S'}(X,t)
	\end{align}
	for all $(X,t) \in E \cap AB_{Q(S')}$, and 
	\begin{align}\label{e:gamma-close-to-plane}
		\pdist((X,t) + \psi_{S'}(X,t), E) \leq \eta D_{S',P_{S'}}(X,t)
	\end{align}
	for all $(X,t) \in P_{S'} \cap AB_{Q(S')}$. 
	
	For $(X,t) \in P_{Q(S')}$, let $g_{S'}(X,t) = (X,t) + \psi_{S'}(X,t)$ and let $\Delta_{S'} = \bigcup_k \Delta_{S',k}$ be a decomposition of $P_{Q(S')}$ into parabolic dyadic cubes with \textit{side length} $\ell(I) = \rho^k$ for each $I \in \Delta_{S',k}$. Then, 
	\begin{align}
		\cubes_{S'} = \{g_{S'}(I) \colon I \in \Delta_{S'}\}
	\end{align} 
	defines a system of Christ-David cubes on $\Gamma_{S'}$ satisfying Lemma \ref{cubes} (1)--(3) with constants $(\Cinner',\Couter',\rho)$ for some $\Cinner',\Couter'$ depending only on $L$.  By Lemma \ref{p:carleson-corona}, there exists a coronization $(\cB_{S'},\cG_{S'},\cF_{S'})$ of $\cubes_{S'}$ such that, for any $T \in \cF_{S'}$ and $Q \in T,$
	\begin{align}\label{e:stopping-time-bounded-functions}
		\left| \avint_{\pi_{S'}(Q)} \nabla \varphi_{S'} \, d\mu -  \avint_{\pi_{S'}(Q(T))} \nabla \varphi_{S'} \, d\mu\right| \leq \eta .
	\end{align}
	
	By \eqref{e:close-to-gamma-proof} we have $\pdist(X_Q,t_Q,\Gamma_{S'}) \leq \eta \ell(Q)$ for each $Q \in S' \in \cF'$. Thus, we can apply Corollary \ref{c:corresponding-stopping-time} with the coronizations $(\cB',\cG',\cF')$ and $\{(\cB_{S'},\cG_{S'},\cF_{S'})\}_{S' \in \cF'}$ to obtain a final coronization $(\cB,\cG,\cF)$ of $\cubes$ such that the following holds. For each $S \in \cF$ and $Q \in S$, there exist unique stopping-time regions $S' \in \cF'$, $T \in \cF_{S'}$ and a cube $Q' \in T$ satisfying 
	\begin{align}\label{e:cubes-nearby}
		S \subseteq S', \quad \ell(Q) \sim \ell(Q') \quad \mbox{and} \quad \pdist(Q,Q') \lesssim \ell(Q). 
	\end{align}
	
	Proposition \ref{t:corona-RF-small} follows immediately from the following.
	
	\begin{lemma}
		Let $S \in \cF$ and let $S' \in \cF'$, $T \in \cF_{S'}$ be the corresponding stopping-time regions as above. For each $Q \in S$, let $A_Q \colon P_{Q(S')} \to (P_{Q(S')})^\perp$ denote the affine function whose graph coincides with $P_Q.$ Then, provided $A$ is large enough, 
		\begin{align}\label{e:prelim-small-angle-1}
			| \nabla A_Q - \nabla A_{Q(S)} | \lesssim \eta \quad \mbox{ for each } Q \in S,
		\end{align}
		where $\nabla$ denotes the spacial gradient with respect to the coordinates generated by $P_{Q(S')}$. 
	\end{lemma}
	
	\begin{proof}
		It suffices to check that 
		\begin{align}\label{e:prelim-small-angle}
			\left| \nabla A_Q - \avint_{I(Q)} \nabla \varphi_{S'} \, d\mu \right| \lesssim \eta, 
		\end{align}
		for each $Q \in S$, where $I(Q) = \pi_{S'}(Q')$ (which is a parabolic cubes in $P_{S'}$). Indeed, if this is true, recalling that $Q'$ and $Q(S)'$ belong to the same stopping-time region $T \in \cF_{\Gamma_{S'}}$, it follows from \eqref{e:stopping-time-bounded-functions} and the triangle inequality that 
		\begin{align}
			| \nabla A_Q - \nabla A_{Q(S)} | \lesssim \eta + \left| \avint_{I(Q)} \nabla \varphi_{S'} \, d\mu - \avint_{I(Q(S))} \nabla \varphi_{S'} \, d\mu  \right| \lesssim \eta.
		\end{align}

		We now verify \eqref{e:prelim-small-angle}. Up to translations and rotations, we can assume 
		\[P_{S'} = \R^{\idim-1} \times \{0\} \times \cdots \times \{0\} \times \R\] 
		and we can check a similar inequality with $\nabla$ replaced by $\partial_1$ -- the partial spacial derivative in the first coordinate direction. Fix some $Q \in S$ and write points in $\R^{\adim+1}$ as $(x_1,x',t) \in \R \times \R^{\adim -1} \times \R$. With this identification, decompose $I \coloneqq I(Q)$ as
		\[I = I_1 \times I' \times I_{t} \subset \R \times \R^{\adim -1} \times \R.\] 
		Let $(x_{I,1},x_{I}',t_I)$ denote the centre of $I$. By \eqref{e:cubes-nearby} (for $A$ sufficiently large), we have $Q' \subseteq AB_Q$. Thus, by \eqref{e:gamma-close-to-plane},
		\begin{align}
			\pdist((x_1,x',t)+ \psi_{S'}(x_1,x',t),P_Q) \lesssim \eta \ell(Q) \quad \mbox{ for all } (x_1,x',t) \in I. 
		\end{align}
		In particular, setting $h \coloneqq A_Q - \varphi_{S'}$, we have $|h(x_1,x',t)| \lesssim \eta \ell(Q) \lesssim\eta  \ell(I)$ for all $(x_1,x',t) \in I$. Combining this with the fundamental theorem of calculus, for each $(x',t) \in I' \times I_t$, we have 
		\begin{align}
			\left|\int_{I_1} \partial_1 h(x_1,x',t) \, dx_n \right| = \left| h(x_{I,1}+\ell(I),x',t) - h(x_{I,1}-\ell(I),x',t) \right| \lesssim \eta \ell(I). 
		\end{align}
		Thus,
		\begin{align}
			\left| \partial_1 A_Q - \avint_I \partial_1 \varphi_{S'} \, d\mu \right| \lesssim \ell(I)^{-\idim-2}	\left| \int_{I_t} \int_{I'} \int_{I_1} \partial_1 h \, dx_1 dx' dt \right| \lesssim \eta,
		\end{align}
		as required. 
			\end{proof}

\subsection{Proof of Theorem \ref{t:BWGL-bilat-Corona}} Let $E \subseteq \R^{\adim+1}$ be a parabolic Ahlfors $(\idim+2)$-regular set satisfying the BWGL and $\cubes$ a system of Christ-David cubes on $E.$ Let $A > 1$ and $\eta , \theta > 0.$ As we have done several times, we can assume without loss of generality that $A$ is sufficiently large and $\eta$ is sufficiently small when needed. By adjusting the constants slightly, it suffices to construct a function with Lipschitz constant at most $C\eta$ satisfying \eqref{e:corona-d} and \eqref{e:corona-D} with $\theta$ replaced by $C\theta$ for some $C > 1$. 

Let $M > 1$ (to be chosen sufficiently large, depending on $A$) and $\ve > 0$ (to be chosen sufficiently small depending on $M,\eta,\theta$). Let $(\cB',\cG',\cF')$ be a coronization of $\cubes$ such that $\bilatbeta_E(MB_Q) < \ve$ for each $Q \in \cG'$, and let $\{P_Q\}_{Q \in \cG'}$ be an assignment of planes such that 
\begin{align}\label{e:small-beta-proof}
	(X_Q,t_Q) \in P_Q \quad \mbox{and} \quad \bilatbeta_E(MB_Q,P_Q) \leq 2\ve. 
\end{align}

\begin{lemma}\label{l:final-corona}
	For $M$ large enough and $\ve = \ve(M,\eta)$ sufficiently small, there exists a coronization $(\cB,\cG,\cF)$ of $\cubes$ such that $\angle(P_Q,P_{Q(S)}) \leq \eta$ for each $Q \in S \in\cF$. 
\end{lemma}

\begin{proof}
	Fix some $S \in \cF'$ for the moment. By Proposition \ref{p:BWGL-RF-corona} (assuming $M$ is sufficiently large), there exists a parabolic Ahlfors $(\idim+2)$-regular $C\ve$-Reifenberg flat set $\Sigma_{S'}$ such that 
	\begin{align}\label{e:coronas-are-close}
		\pdist(X,t,\Sigma_{S'}) \lesssim \ve M\ell(Q)
	\end{align}
	for all $Q \in S'$ and $(X,t) \in E \cap AB_Q$ and 
	\begin{align}\label{e:coronas-are-close-converse}
		\pdist(X,t,E) \lesssim \ve M\ell(Q)
	\end{align}
	for all $Q \in S'$ and $(X,t) \in \Sigma_{S'} \cap AB_Q$. Let $\cubes_{S'}$ be a system of Christ-David cubes on $\Sigma_{S'}$ and let $\{V_{Q}\}_{Q \in \cubes_{S'}}$ be an assignment of planes such that 
	\begin{align}\label{e:sigma-is-flat}
		(X_Q,t_Q) \in V_Q \quad \mbox{and} \quad \bilatbeta_{\Sigma_{S'}}(M B_Q,V_Q) \lesssim \ve,
	\end{align}
	which exists since $\Sigma_{S'}$ is $C\ve$-Reifenberg flat. By Proposition \ref{t:corona-RF-small} (assuming $\ve = \ve(M,\eta)$ is sufficiently small), there exists a coronization $(\cB_{S'},\cG_{S'},\cF_{S'})$ of $\cubes_{S'}$ such that 
	\begin{align}\label{e:angle-in-the-stopping-time}
		\angle(V_Q,V_{Q(T)}) \leq \eta/4 \quad \mbox{for all } Q \in T \in \cF_{S'}. 
	\end{align}
	
	By \eqref{e:coronas-are-close}, we can apply Corollary \ref{c:corresponding-stopping-time} to $(\cB,\cG,\cF)$ and $\{(\cB_{S'},\cG_{S'},\cF_{S'})\}_{S' \in \cF'}$, which provides a further coronization $(\cB,\cG,\cF)$ of $\cubes$ such that, for each $S \in \cF$, there are unique stopping-time regions $S' \in \cF'$, $T \in \cF_{S'}$ such that
    $S \subset S'$, and such that 
    for each $Q \in S$ there exists a cube $Q' \in T$ satisfying
	\begin{align}\label{e:comparable}
		\ell(Q) \sim \ell(Q') \quad \mbox{and} \quad \pdist(Q,Q') \lesssim \ell(Q).
	\end{align}
	
	Now, fix some $S \in \cF$ (let $S' \in \cF$, $T \in \cF_{S'}$ be the corresponding stopping-time regions from the preceding paragraph) and let us check that $\angle(P_Q,P_{Q(S)}) \leq \eta$. It suffices to check that 
	\begin{align}\label{e:angles-between-planes}
		\angle(P_Q,V_{Q'}) \leq \eta/4 \quad \mbox{ for each } Q \in S. 
	\end{align}
	Indeed, if \eqref{e:angles-between-planes} holds, combining this with \eqref{e:angle-in-the-stopping-time} (recalling that $Q'$ and $Q(S)'$ lie in the same stopping-time regions $T$), we have 
	\begin{align}
		\angle(P_Q,P_{Q(S)}) &\leq \angle(P_Q,V_{Q'}) + \angle(V_{Q'},V_{Q(T)}) + \angle(V_{Q(T)},V_{Q(S)'}) + \angle(V_{Q(S)'} , P_{Q(S)} ) \\
		&\leq  \eta. 
	\end{align}
	To show \eqref{e:angles-between-planes}, fix some $Q \in S$ and let $(X,t) \in P_Q \cap B_Q$. In the following we shall assume $\ve = \ve(M)$ is sufficiently small so that $\ve M \ll 1$. By \eqref{e:small-beta-proof} there exists $(X',t') \in E \cap 2B_Q$ such that $\pd(X,t,X',t') \lesssim M \ve \ell(Q)$. By \eqref{e:coronas-are-close} and the second property in \eqref{e:comparable}, there exists $(X'',t'') \in \Sigma_{S'} \cap 3B_Q$ such that $\pd(X',t',X'',t'') \lesssim M\ve \ell(Q)$. By \eqref{e:sigma-is-flat} and the second property in \eqref{e:comparable}, there exists $(X''',t''') \in P_{Q'}$ such that $\pd(X',t',X'',t'') \lesssim \ve M \ell(Q)$. Thus, so long as $\ve = \ve(M,\eta)$ is sufficiently small, the previous three estimates imply
	\[\pdist(X,t,P_{Q'}) \lesssim M \ve \ell(Q) \leq \eta/4.\] 
	
	The converse inequality is similar, starting instead from $P_{Q'}$. Indeed, fix some $(Y,s) \in P_{Q'} \cap B_Q$. For $M$ large enough, it follows from \eqref{e:comparable} that $B_Q \subseteq MB_{Q'}$. In particular, by \eqref{e:sigma-is-flat}, there exists a point $(Y',s') \in \Sigma_{S'} \cap 2B_{Q}$ such that $\pd(Y,s,Y',s') \lesssim \ve M\ell(Q)$. By \eqref{e:coronas-are-close-converse}, there exists $(Y'',s'') \in E \cap 3B_Q$ such that $\pd(Y',s',Y'',s'') \lesssim \ve M \ell(Q)$ and finally, by \eqref{e:small-beta-proof}, there exists $(Y''',s''') \in P_Q$ such that $\pd(Y'',s'',Y''',s''') \lesssim \ve M \ell(Q)$. So long as $\ve = \ve(M,\eta)$ is sufficiently small, the previous three estimates imply 
	\begin{align}
		\pd(Y,s,P_{Q'}) \lesssim \ve M \ell(Q) \leq \eta /4. 
	\end{align}
	as required. 
\end{proof}

Now, choose $M/A$ large enough and $\ve = \ve(M,\eta)$ small enough such that Lemma \ref{p:graph-construction} holds with constant $A,M,\ve,\eta$ and Lemma \ref{l:final-corona} holds. Furthermore, we will assume that $\ve < \theta / M$ and that $\eta$ is small enough such that $L(\eta) \sim \eta$. With this choice of constants, it follows from Lemma \ref{p:graph-construction} and Lemma \ref{l:final-corona} that for each $S \in \cF$ there is a plane $P_S$ and $\Lip(1,1/2,C\eta)$ function $\psi_S \colon P_S \to P_{S}^\perp$ satisfying \eqref{e:corona-d} and \eqref{e:corona-D} with constant comparable comparable to $\ve M < \theta$, as required. This finishes the proof of Theorem \ref{t:BWGL-bilat-Corona}.

\subsection{The converse to Theorem \ref{t:BWGL-Corona}} We finish this section by showing the following.

	\begin{lemma}
		Let $E \subseteq \R^{d+1}$ be parabolic Ahlfors $(n+2)$-regular. If $E$ admits a corona decomposition, then it satisfies the \emph{BWGL}. 
	\end{lemma}
	
	\begin{proof}
		Let $\ve ,\eta> 0$ and $A > 1$, where $\eta$ is to be chosen sufficiently small depending on $\ve$. Since $E$ admits a corona decomposition, it admits a corona decomposition with approximation parameter $\eta$ by $\mathcal{E}$ -- the collection of Lip$(1,1/2)$ graphs with Lipschitz constant at most $\eta$. Fix some $\Gamma \in \mathcal{E}$ and a system of Christ-David cubes $\cubes_\Gamma$ on $\Gamma$. Then, $\Gamma$ parabolic Ahlfors $(n+2)$-regular with constant depending only on $\eta$. Since $\Gamma$ is a Lipschitz graph with constant at most $\eta$, we have that $\beta_\Gamma(C'AB_Q) < C\eta$ for all $Q \in \cubes_\Gamma$ and any $C' > 1$. By \cite[Lemma 4.10]{bortz2023corona}, if $C'$ is a suitably chosen constant (depending only on $\eta$), we have $b\beta(AB_Q) < C\eta$ for all $Q \in \cubes_\Gamma$. In particular, $\mathcal{E}$ is a class of uniformly Ahlfors regular sets which satisfies the BWGL with parameter $(A,C\eta)$ with uniform packing constants (equal to zero). By choosing $\eta$ small enough (depending on $C$ and $\ve)$, it then follows directly from \cite[Theorem 1.2]{bortz2023corona} that $E$ satisfies the BWGL with constants $(A,\ve)$ with packing constants depending only on $\ve$ and the packing constant in the corona decomposition. Since $\ve$ and $A$ were arbitrary, this finishes the proof. 
	\end{proof}

\section{BAUP implies BWGL}\label{s:BAUP-BWGL}

In this section we prove the following. 

\begin{theorem}\label{BAUPBWGL}
	Suppose $E \subseteq \R^{\adim+1}$ is parabolic Ahlfors $(\idim+2)$-regular, $1 \leq n \leq d-1$. If $E$ satisfies the BAUP then it satisfies the BWGL.
\end{theorem}

This result is known in the elliptic case, but is new in the parabolic case. However, even in the elliptic case, the known proofs are rather indirect and pass via several difficult properties related to the theory of uniformly rectifiability (WHIP, generalised corona decompositions etc). Since the BAUP is morally very closely related to the BWGL, this is rather unsatisfactory. Here, we provide a more direct proof of the fact that BAUP implies BWGL in both the elliptic and parabolic settings. 

The argument is inspired by arguments in \cite{azzam2022weak} (which is based on the method developed in \cite{schul2007subsets}), which controls how often the density of a collection of cubes can be larger than one. Assume for the sake of exposition that every cube is good for the BAUP condition. If a cube is bad for the BWGL, there should be at least two distinct planes which well-approximate the set in that cube. It would be nice if this implied that the density (of a dilation) of the cube is larger than one (so that we could directly apply the main result of \cite{azzam2022weak}), however, this is simply not true under the assumption of Ahlfors regularity alone. If the set had some topology, a statement like this could be true. The key thing to note is that, in between the scales which are bad for the BWGL, the cubes are good for the BWGL. Thus, by Theorem \ref{t:DT}, one can approximate these scales by a Reifenberg flat set. This imposes a sort of \textit{approximate topology} on the set, which we can then utilize.

\subsection{Schul's cubes}

Before embarking on the proof of Theorem \ref{BAUPBWGL}, we need to recall Schul's cubes from \cite{schul2007subsets}. These are a collection of subsets of a metric space, which are similar to the cubes constructed in Lemma \ref{cubes}. Both collections have the property that they can be divided into different generations with nesting properties. The main difference is that, while Christ-David cubes are partitioned by cubes at the next generation of roughly the same size, the children of Schul's cubes can vary wildly in size. The important property Schul's cubes do have, is that they are approximately like balls.\\

\begin{lemma}\label{l:Schul-cubes}
	Let $(X,d,\mu)$ be a doubling metric measure space, $\{X_k\}_{k \in \mathbb{Z}}$ be a sequence of subsets of $X$ satisfying \eqref{e:sep-net} and $\CSchulprime \coloneqq 1/36$. For $\rho$ small enough, there is a collection of subsets $\cC \coloneqq \bigcup_{k \in \mathbb{Z}} \cC_k$ of $X$, which we call Schul cubes, such that that the following conditions hold.  
	\begin{enumerate}
		\item If $Q \in \cC_k$ and $R \in \cC_{k'}$ for some $k' \geq k$, and $Q \cap R \neq\emptyset$, then $R \subseteq Q$. 
		\item For each $Q \in \cC_k$ there is a point $x_Q \in X_k$ such that, if $B(Q) \coloneqq B_{\CSchulprime \rho^k}(x_Q)$, then 
		\begin{align}
			B(Q) \subseteq Q \subseteq (1+8\rho)B(Q). 
		\end{align}
		\item For each $k \in \mathbb{Z}$, we have 
		\begin{align}
			\mu\left( X \setminus \bigcup_{m \geq k}\bigcup_{Q \in \cC_m} Q \right) = 0 .
		\end{align}
	\end{enumerate}
	Given a cube $Q \in \cC$ and a subset $\cQ \subseteq \cC$, let 
	\begin{align}\label{e:descendents}
		\cC(Q) \coloneqq \{R \in \cC \colon R \subseteq Q \} \quad \mbox{ and } \quad \cC(\cQ) \coloneqq \bigcup_{Q \in \cQ} \cC(Q). 
	\end{align}
\end{lemma}

\begin{proof}
    Properties (1) and (2) are already shown in \cite{azzam2015hausdorff} and require only that the space $X$ is metrically doubling. To show property (3), let us fix some $k \in \mathbb{Z}$ and denote $E = \bigcup_{m \geq k}\bigcup_{Q \in \cC_k} Q$. Since $\mu$ is doubling, by the Lebesgue differentiation theorem it suffices to show that 
    \begin{align}\label{e:positive-density}
        \limsup_{r \to 0} \frac{\mu(E \cap B_r(x))}{\mu(B_r(x))} > 0 \quad \mbox{ for each } x \in X. 
    \end{align}
    To prove \eqref{e:positive-density}, fix some $x \in X$, $r \in (0,\rho^k)$, and let $m \geq k$ such that $\rho^{m+1} \leq r < \rho^m$. By \eqref{e:sep-net}, there exists $z \in X_{m+2}$ such that $d(x,z) < 2\rho^{m+2}$. If $\rho = \rho(\CSchulprime)$ is small enough, then by our choice of $m$, if $Q \in \cC_{m+2}$ is the cube such that $x_Q = z$, we have 
    \begin{align}
    	B(Q) = B_{\CSchulprime \rho^{m+1}}(z) \subseteq B_{\CSchulprime \rho^{m+2} + 2 \rho^{m+2}}(x) \subseteq B_{(\CSchulprime + 2)\rho r}(x) \subseteq B_r(x)
    \end{align} 
    and $\mu(B(Q)) \geq c\mu(B_r(x))$ for some constant $c > 0$ depending only on the doubling constant. It now follows that 
    \begin{align}
        \frac{\mu(E \cap B_r(x))}{\mu(B_r(x))} \geq \frac{\mu(B(Q))}{\mu(B_r(x))}  \geq c > 0. 
    \end{align}
    Since $r > 0$ was arbitrary, \eqref{e:positive-density} now follows. 
\end{proof}

\begin{lemma}\label{l:subsequence}
    Let $\{E_j\}$ be a sequence of uniformly bounded subsets of a doubling metric space $X$ and $\{\cC(j)\}$ be a corresponding sequence of Schul cubes. There is a subsequence $\{E_{j_n}\}$ converging to some set $E$ in the Hausdorff metric and a system of Schul cubes $\cC$ on $E$ such that the following holds. For every $\delta \in (0,1)$ and $k \in \N$ there is $N \in \N$ such that, if $n \geq N$, then there exists a bijection 
    \begin{align}
        g_{j_n} \colon \bigcup_{0 \leq m \leq k} \cC_m \to \bigcup_{0 \leq m \leq k} \cC_m(j_n)
    \end{align}
    satisfying
    \begin{align}\label{e:balls-are-nest}
        (1+\delta/2)^{-1} B(g_{j_n}(Q)) \subseteq B(Q) \subseteq (1+\delta/2) B(g_{j_n}(Q)).
    \end{align}
    Moreover, if $\rho \leq \rho_1(\delta) \coloneqq \delta/32$, then 
    \begin{align}\label{e:comp-g(Q)} 
        (1+\delta)^{-1}\pdiam(Q) \leq \pdiam(g_{j_n}(Q)) \leq (1+\delta)\pdiam(Q) .
    \end{align}
\end{lemma}

\begin{proof}
    For each $j$, let $\{X_k(j)\}_{k}$ be the sequence of points from which $\cC(j)$ is constructed. Since $\{E_j\}$ is uniformly bounded (hence, also the nets), by compactness and a diagonal argument, we can find a sequence $\{j_n\} \subseteq \{j\}$ such that $\{E_{j_n}\}$ converges to some set $E$ and $\{X_k(j_n)\}_n$ converges for each $k$. Let $X_k$ denote the limit of the sequence $\{X_k(j_n)\}_n$. Since each of the $X_k(j)$ satisfy \eqref{e:sep-net} in $E_j$, it is easy to check that $\{X_k\}_k$ satisfies condition \eqref{e:sep-net} in $E$. We let $\cC$ be the system of Schul cubes constructed via Lemma \ref{l:Schul-cubes} with respect to this net. 

    Now, fix some $\delta \in (0,1)$ and $k \in \N$. Since $k$ is finite and $\{X_m(j_n)\}_n$ converges to $X_m$ for each $0 \leq m \leq k$, there exists $N \in \N$ such that 
    \begin{align}\label{e:nets-are-close}
    d_H(X_m,X_m({j_n})) < \delta \CSchulprime \rho^k \leq \delta \CSchulprime \rho^m \quad \mbox{ for all } 0 \leq m \leq k \mbox{ and } n \geq N.
    \end{align}
    
    For $Q \in \cC_m$ with $0 \leq m \leq k$, let $g_{j_n}(Q)$ be \textit{the} cube in $\cC_m(j_n)$ such that 
    \begin{align}\label{e:g_j(Q)}
        |x_Q - x_{g_j(Q)}| < \delta \CSchulprime \rho^m.
    \end{align} 
    The existence of this cube follows from \eqref{e:nets-are-close}, the uniqueness follows from the separation condition \eqref{e:sep-net} since $\CSchulprime < 1/4$. If $Q,R \in \cC_m$ are distinct cubes, then we must have $g_{j_n}(Q) \neq g_{j_n}(R)$ otherwise (by the triangle inequality) we contradict the separation condition $X_m$ i.e., the map $g_j$ is injective. By \eqref{e:nets-are-close}, for each $Q' \in \cC_m(j_n)$ there exists $Q \in \cC_m$ such that $|x_Q - x_{Q'} | < \delta \CSchulprime \rho^m$, thus, the map is surjective. The inclusion \eqref{e:balls-are-nest} follows directly from \eqref{e:g_j(Q)} and the triangle inequality. The second inequality in \eqref{e:comp-g(Q)} follows from Lemma \ref{l:Schul-cubes} (2), \eqref{e:balls-are-nest} and our choice of $\rho$, since 
    \begin{align}
    	\pdiam(g_{j_n}(Q)) &\leq (1+8\rho)\pdiam(B(g_{j_n}(Q)) \leq (1+8\rho)(1+\delta/2)\pdiam(B(Q)) \\
    	&\leq (1+\delta)\pdiam(Q). 
    \end{align}
    The first inequality in \eqref{e:comp-g(Q)} is similar.
\end{proof}

\subsection{Preliminary arguments}\label{s:BAUP-prelim} We start by gather several useful result that will assist us in the proof of Theorem \ref{BAUPBWGL}. We will assume from here on that we are working within parabolic $\R^{\adim +1}.$ Since the distinction between the $X$ and $t$ components will not be very important here, for the sake of brevity, we will typically denote a points $(X,t)$ and $(Y,s)$ instead just by $x$ and $y$ (for example).

We begin with the following which show that, for any $B$ which is sufficiently well approximated by a union of planes, there exists a collection $\cN$ of disjoint Schul cubes contained in $B$ which are flat, relatively large, and satisfy that the sum of $\pdiam(R)^{\idim+2}$ over $R \in \cN$ is at least (up to a small perturbation) $\pdiam(B)^{\idim+2}$. Moreover, if it holds that $\bilatbeta_E(B) \sim \alpha$, then the approximate measure of the cubes in $\cN$ is almost \textit{double} the approximate measure of the top cube. This doubling property be extremely important in our arguments later, where this will be used to distribute mass along these cubes.

\begin{lemma}\label{l:lots-of-flat-children}
	For each $C_0 , M > 1$ and $\alpha , \delta \in (0,1)$ there is $\Nflat > 1$ and $\veflat > 0$ such that the following holds. Suppose $E \subseteq \R^{\adim+1}$ is parabolic Ahlfors  $(\idim+2)$-regular with constant $C_0$ and $\cC$ is a system of Schul cubes on $E$ with $\rho = \rho_1(\delta/(4\idim+8))$ (with $\rho_1(\cdot)$ as in Lemma \ref{l:subsequence}). For every ball $B$ centered on $E$ such that $\baupbeta_E(B) < \ve$ for some $0 < \ve < \veflat$, there exists a disjoint collection of cubes $\cN \subseteq \cC$ contained in $B$ such that 
	\begin{align}\label{e:size}
		\pdiam(R) \geq \rho^{\Nflat} \pdiam(B) \quad \mbox{ and } \quad \bilatbeta_E(MB(R)) < \delta \alpha
	\end{align}
	for all $R \in \cN$, and  
	\begin{align}\label{e:sum}
		\sum_{R \in \cN} \pdiam(R)^{\idim+2} \geq (1+\delta)^{-1} \pdiam(B)^{\idim+2}.
	\end{align}
	Moreover, if $\bilatbeta_E(B) \sim \alpha$, then in fact we have 
	\begin{align}\label{e:big-size}
		\sum_{R \in \cN} \pdiam(R)^{\idim+2} \geq (1+\delta)^{-1}(2-C\alpha) \pdiam(B)^{\idim+2}. 
	\end{align}
\end{lemma}

\begin{proof}
	We begin by showing \eqref{e:size} and \eqref{e:sum}. A small modification of the argument will give \eqref{e:big-size} under the additional assumption, which we will describe at the end. Let us suppose the first part of the lemma is false. Then, there exist constants $C_0,M,\alpha,\delta$ and for each $j \in \N$, a parabolic Ahlfors $(\idim+2)$-regular set $E_j$ with constant $C_0$, a system of Schul cubes $\cC(j)$ on $E_j$ with $\rho = \rho_1(\delta/(4\idim+4))$, and a ball $B_j$ centered on $E_j$ such that 
	\begin{align}\label{e:decaying-baup-2}
		\baupbeta_{E_j}(B_j) < 1/j
	\end{align}
	and 
	\begin{align}\label{e:comp2-as-1}
		\sum_{R \in \cN} \pdiam(R)^{\idim+2} \leq (1+\delta)^{-1} \pdiam(B)^{\idim+2}
	\end{align}
	for any disjoint collection of cubes $\cN \subseteq \cC(j)$ contained in $B_j$ satisfying
	\begin{align}\label{e:comp2-as-2}
		\pdiam(R) \geq 2^{-j}  \pdiam(B_j) \quad \mbox{ and } \quad \bilatbeta_{E_j}(MB(R)) < \delta\alpha \quad \mbox{ for } R \in \cN.  
	\end{align}
	
	By scaling and translating, we may assume that $B_j = B_1(0)$, and since the statement is local, we can assume that the sequence $\{E_j\}$ is uniformly bounded. By Lemma \ref{l:subsequence}, we can assume that $\{E_j\}$ converges to a set $E$ on which there is a system of Schul cubes satisfying the conclusions of Lemma \ref{l:subsequence} with respect to the sequence $\{\cC(j)\}$.

	Let $\eta > 0$ (to be chosen sufficiently small, depending on $\delta)$. By \eqref{e:decaying-baup-2}, we have that $\baupbeta_E(B_1(0)) = 0$. Thus, there is a collection of $t$-independent planes $\mathscr{V}$ such that $E \cap B_1(0) \supseteq V \cap B_1(0)$ for each $V \in \mathscr{V}$. Let $V_0 \in \mathscr{V}$ be the plane which passes through the origin (which exists since $0 \in E$), so that
	\begin{align}\label{e:large-measure-2}
		\pHmeasure^{\idim+2}(V_0 \cap B_{1-\eta}(0)) \geq (1-\eta)^{\idim+2}2^{\idim+2}. 
	\end{align}
	
	 Suppose initially that $k \geq 0$ is large enough such that $\pdiam(R) < \eta/2$ for all $R \in \cC_m$ with $m \geq k$. Let $\cC_m'$ be the set of cubes in $\cC_m$ which intersect $B_{1-\eta}(0)$, and let $\cQ_k$ be a maximal collection of cubes in $\bigcup_{m \geq k} \cC_m'$. By the triangle inequality, 
	\begin{align}\label{e:contained-in-unit-ball}
		R \subseteq B_{1-\eta/2}(0) \quad \mbox{ for all } R \in \cQ_k ,
	\end{align}
	and by Lemma \ref{l:Schul-cubes} (3),
	\begin{align}\label{e:nothing-missing-2}
		\pHmeasure^{\idim+2}\left(V_0 \cap B_{1-\eta}(0) \setminus \bigcup_{R \in \cQ_k} R  \right) = 0 .
	\end{align}
	Since $V_0$ is upper regular with constant $1$, it follows from \eqref{e:large-measure-2} and \eqref{e:nothing-missing-2} that 
	\begin{align}
		\sum_{R \in \cQ_k} \pdiam(R)^{\idim+2} \geq \sum_{R \in \cQ_k} \pHmeasure^{\idim+2}(V_0 \cap R) \geq (1-\eta)^{\idim+2}2^{\idim+2}. 
	\end{align}
	Thus, there exists $\kappa = \kappa(\eta)$ such that, if $\cQ'_k$ denotes the set of cubes in $R \in \cQ_k$ with $\pdiam(R) \geq \kappa$, then  
	\begin{align}\label{e:large-sum-2}
		\sum_{R \in \cQ'_k} \pdiam(R)^{\idim+2} \geq (1-2\eta)^{\idim+2}2^{\idim+2}.
	\end{align}
	
	For $k \geq 0$ large enough, if $\cB_k$ denotes the collection of cubes in $R \in \cQ_k'$ for which $2MB(R)$ intersects more than one plane from $\mathscr{V}$, then we claim that  
	\begin{align}\label{e:B_m}
		\sum_{R \in \cB_k} \pdiam(R)^{\idim+2} \leq \eta 2^{\idim+2}. 
	\end{align}
	
	Indeed, consider any distinct pair of planes $V,W \in \mathscr{V}$. Since there are only finitely many planes in $\mathscr{V}$, it suffices to prove an estimate similar to \eqref{e:B_m} with $\cB_k$ replaced by $\cL_k$, where
	\begin{align}
		\cL_k \coloneqq \{R \in \cQ'_k \colon 2MB(R) \cap V \neq\emptyset \mbox{ and } 2MB(R) \cap W \neq\emptyset\}.
	\end{align}
	
	If $V \cap W = \emptyset$, we have $\cL_k = \emptyset$ for $k$ large enough (since the maximal diameter must go to zero as $k \to \infty$) and \eqref{e:B_m} follows.  Suppose instead that $L \coloneq V \cap W \neq\emptyset$ and let $\theta > 0$ be the angle between $V$ and $W$. Given $\alpha > 0$, by choosing $k$ large enough (depending on $\theta, \alpha$ and $M$) one can guarantee that $\cL_k$ forms a disjoint collection of cubes centered on $V \cup W$ and contained in $\mathcal{N}_{\alpha}(L) \cap B_1(0)$, where $\mathcal{N}_{\alpha}(L)$ denotes the $\alpha$-neighbourhood of $L$.  Since $L^0$ has dimension at most $n-1$ and $V,W$ are vertical planes in $\verticalplanes$ containing $L$, it follows that 
	\begin{align}
		\pHmeasure^{\idim+2}((V \cup W) \cap \mathcal{N}_{\alpha}(L) \cap B_1(0)) \lesssim \alpha 2^{\idim+2}. 
	\end{align}
	Thus, since $V \cup W$ is parabolic Ahlfors regular $(\idim+2)$-regular and the cubes in $\cL_k$ are contained in $\mathcal{N}_{\alpha}(L) \cap B_1(0)$, we have 
	\begin{align}
		\sum_{R \in \cL_m} \pdiam(R)^{\idim+2} &\lesssim \sum_{R \in \cL_m} \pHmeasure^{\idim+2}((V \cup W) \cap R) \\
		&\leq \pHmeasure^{\idim+2}((V \cup W) \cap \mathcal{N}_{\alpha}(L) \cap B_1(0))  \lesssim \alpha 2^{\idim+2}. 
	\end{align}
	The inequality \eqref{e:B_m} now follows by choosing $\alpha = \alpha(\eta)$ sufficiently small (hence, $k = k(\eta)$ sufficiently large). 
	
	By \eqref{e:large-sum-2} and \eqref{e:B_m}, we have 
	\begin{align}\label{e:large-large}
		\sum_{R \in \cQ_k' \setminus \cB_k} \pdiam(R)^{\idim+2} \geq (1-3\eta)^{\idim+2}2^{\idim+2}. 
	\end{align}
	Also, for any $R \in \cQ'_k  \setminus \cB_k$ (so $2MB(R)$ only intersects one plane from $\mathscr{V}$) we have  
	\begin{align}\label{e:not-B-flat}
		\bilatbeta_E(2MB(R)) = 0. 
	\end{align}
	
	By Lemma \ref{l:subsequence}, for every $j$ large enough there is a one-to-one correspondence between the set of cubes in $\cQ'_k$ and the set of cubes in $\cC(j)$ with $\pdiam(R) \geq \kappa$. Moreover, by \eqref{e:balls-are-nest}, it can be made so that if $R \in \cQ'_k$ and $R' \in \cC(j)$ are corresponding cubes, then 
	\begin{align}\label{e:comp-size}
		\left(1+\frac{\delta}{4\idim+8}\right)^{-1} \pdiam(R') \leq \pdiam(R) \leq \left(1+\frac{\delta}{4\idim+8}\right) \pdiam(R').
	\end{align}
	Our choice of $\rho = \rho_1(\delta/(4\idim+8))$ should hopefully now become clear. With this choice, using the inequalities $(1+x/a)^{a} \leq e^x$ and $e^x \leq 1+2x$ for $x \in (0,1)$, we have 
	\begin{align}\label{e:simple-inequality}
		\left(1 + \frac{\delta}{4\idim+8}\right)^{-\idim-2} \geq (1+\delta/2)^{-1}.
	\end{align}
	
	For such a $j$, let $\cN(j)$ be the set of cubes in $\cC(j)$ corresponding to $\cQ'_k \setminus \mathscr{B}_k$. By \eqref{e:contained-in-unit-ball} and \eqref{e:comp-size}, $\cN(j)$ is a disjoint collection of cubes in $B_1(0)$ with $\pdiam(R) \gtrsim \kappa$ for all $R \in \cN(j)$. Moreover, by \eqref{e:large-large},  \eqref{e:comp-size} and \eqref{e:simple-inequality}, if $\eta$ is sufficiently small (depending on $\delta$)
	\begin{align}
		\sum_{R \in \cN(j)} \pdiam(R)^{\idim+2} \geq (1+\delta/2)^{-1} \sum_{R \in \cQ'_k \setminus \mathscr{B}_k} \pdiam(R)^{\idim+2} \geq (1+\delta)^{-1} 2^{\idim+2}.
	\end{align}
	Furthermore (for $j \in \N$ large enough, again), it follows from \eqref{e:not-B-flat} and the fact that $\{E_j\}$ converges to $E$, that 
	\begin{align}
		\bilatbeta_{E_j}(MB(R)) < \delta \alpha \quad \mbox{ for all } R \in \cN(j). 
	\end{align}
	Thus, we have arrived at a contradiction to \eqref{e:comp2-as-1} and \eqref{e:comp2-as-2}, which completes to the of \eqref{e:size} and \eqref{e:sum}. \\
	
	We finish by describing the modifications necessary to obtain \eqref{e:big-size} under the additional assumption that $\bilatbeta(B) \sim \alpha$. For this, one runs a similar compactness argument with the additional condition that $\bilatbeta_{E_j}(B_1(0)) \sim \alpha$ for each $j$. This translates to the fact that $\bilatbeta_E(B_1(0)) \sim \alpha$, which implies that $\mathscr{V}$ contains at least two distinct planes $V_1,V_2$ which pass within distance $C\alpha$ of the origin. In particular, this implies 
	\begin{align}
		\pHmeasure^{\idim+2}(V_i \cap B_{1-\eta}(0)) \geq (1-\eta - C\alpha)^{\idim+2} 2^{\idim+2}, \quad i \in \{1,2\}.
	\end{align}
	By first removing those cubes which intersect more than one plane (whose mass is small), the cubes $\cQ_k' \subseteq \bigcup_{m \geq k} \cC_m'$ we obtain from the argument above \eqref{e:large-sum-2} satisfy the stronger condition that 
	\begin{align}
		\sum_{R \in \cQ'_k} \pdiam(R)^{\idim+2} \geq 2(1-C\eta - C\alpha)^{\idim+2}2^{\idim+2}.
	\end{align}
	Carrying this improved estimated through the remainder of the argument after \eqref{e:not-B-flat}, one obtains (after minor notational changes) that for $\eta = \eta(\alpha,\delta)$ small enough and $j$ the collection of cubes $\cN(j)$ has 
	\begin{align}
		\sum_{R \in \cN(j)} \pdiam(R)^{\idim+2} \geq (1+\delta)^{-1}(2-C\alpha)2^{\idim+2},
	\end{align}
	which is enough to arrive at a contradiction. 
\end{proof}

By necessity, the constant $\Nflat$ in Lemma \ref{l:lots-of-flat-children} will depend on $\alpha$. It will be useful to have a variant of Lemma \ref{l:lots-of-flat-children} in which we do not require the collection $\cN$ to be flat, but we do require it to have large approximate mass. In this way, we can make the lower bound for the sizes of the cubes independent of $\alpha$. This is the purpose of the following lemma. The proof of Lemma \ref{l:lots-of-children} is similar (and easier) than the proof of Lemma \ref{l:lots-of-flat-children} (one can skip the middle part of the proof which controls the cubes which intersect more than one plane), thus, we shall not include it.

\begin{lemma}\label{l:lots-of-children}
	For each $\CAR > 1$ and $\delta \in (0,1)$, there exists $\Nnonflat > 1$ and $\venonflat > 0$ such that the following holds. Suppose $E \subseteq \R^{\adim+1}$ is Ahlfors $(\idim+2)$-regular with constant $C_0$ and $\cC$ is a system of Schul cubes on $E$ with $\rho = \rho_1(\delta/(4\idim+8))$ (with $\rho_1(\cdot)$ as in Lemma \ref{l:subsequence}). For every ball $B$ centered on $E$ such that $\baupbeta_E(B) < \ve$ for some $0 < \ve < \venonflat$, there exists a disjoint collection of cubes $\cN \subseteq \cC$ contained in $B$ such that 
	\begin{align}
		\pdiam(R) \geq \rho^{\Nnonflat} \pdiam(B) 
	\end{align}
	for all $R \in \cN$, and  
	\begin{align}
		\sum_{R \in \cN} \pdiam(R)^{\idim+2} \geq (1+\delta)^{-1} \pdiam(B)^{\idim+2}.
	\end{align}
\end{lemma}

The following is a variant of Lemma \ref{l:lots-of-children}, in which we show how to cover a large portion of a nearby set $F$ with Schul cubes centered on $E$.

\begin{lemma}\label{l:large-coverage-nearby}
	For each $\CAR > 1$ and $\delta > 0$ there exists $\alpha_3 , \ve_3 > 0$ such that the following holds. Suppose $E \subseteq \R^{\adim+1}$ is parabolic Ahlfors $(\idim+2)$-regular with constant $\CAR$ and $\cC$ is a system of Schul cubes on $E$ with $\rho = \rho_1(\delta/(12\idim+24))$. Suppose $B_0$ is a ball centered on $E$ such that $\baupbeta_E(B_0) < \ve$ for some $0 < \ve < \ve_3$, and suppose $F$ is a set such that $d_{B_0}(E,F) < \alpha$ for some $0 < \alpha < \alpha_3$. Furthermore, suppose
	\begin{align}\label{e:coverage-nearby-1}
		\pHmeasure^{\idim+2}(F \cap B_0) \leq (1+\alpha) \pdiam(B_0)^{\idim+2} 
	\end{align}
	and 
	\begin{align}\label{e:coverage-nearby-2}
		\pHmeasure^{\idim+2}(F \cap B) \geq (1+\alpha)^{-1} \pdiam(B)^{\idim+2}
	\end{align}
	for all balls $B$ centered on $F \cap B_0$ with $0 < r_B < r_{B_0}$. Then, there is a disjoint collection of cubes $\{Q_i\}_{i \in I} \subseteq \cC$ with $\pdiam(Q_i) \geq \rho^{\Nnonflat}\pdiam(B_0)$ for all $i \in I$ (with $\Nnonflat$ as in Lemma \ref{l:lots-of-children} i.e., independent of $\alpha$) and 
	\begin{align}
		\pHmeasure^{\idim+2}\left(F \cap B_0 \setminus  \bigcup_{i \in I} Q_i\right) \leq \delta \pdiam(B_0)^{\idim+2}.
	\end{align}
\end{lemma}

\begin{proof}
By Lemma \ref{l:lots-of-children}, so long as $\ve_3 = \ve_3(\CAR,\delta)$ is sufficiently small, there is a collection of disjoint cubes$\{Q_i\}_{i \in I} \subseteq \cC$ contained in $B_0$ such that $\pdiam(Q_i) \geq \rho^{\Nnonflat}\pdiam(B_0)$ and 
\begin{align}\label{e:nearby-1}
	\sum_{i \in I} \pdiam(Q_i)^{\idim+2} \geq (1+\delta/3)^{-1}\pdiam(B_0)^{\idim+2}. 
\end{align}

Let $\eta > 0$ (to be chosen sufficiently small) and, for each $i \in I$, let $k(i) \in \mathbb{Z}$ such that $Q_i \in \cC_{k(i)}$. Since $\pdiam(Q_i) \geq \rho^{\Nnonflat}\pdiam(B_0)$ (with $\Nnonflat$ independent of $\alpha$) and $d_{B_0}(E,F) < \alpha$, by choosing $\alpha_3 = \alpha_3(\eta)$ small enough there exists  a point $x_{Q_i}' \in F$ such that 
\begin{align}
	|x_{Q_i} - x_{Q_i}'| < \eta \CSchulprime \rho^{k(i)}. 
\end{align} 
In particular, we have  
\begin{align}\label{e:nearby-1.5}
	B_i \coloneqq B_{(1-\eta)\CSchulprime \rho^{k(i)}}(x_{Q_i}') \subseteq B_{ \CSchulprime \rho^{k(i)}}(x_{Q_i}) \subseteq Q_i.
\end{align}

For $\eta = \eta(\delta)$ small enough, our choice of $\rho$ implies $\pdiam(Q_i)^{n+2} \leq (1+\delta/3)\pdiam(B_i)^{n+2}$ for each $i \in I$. Thus, by \eqref{e:nearby-1} and \eqref{e:nearby-1.5}, we find that for $\alpha_3 = \alpha_3(\delta)$ sufficiently small, the set $\{B_i\}_{i \in I}$ is a collection of disjoint balls contained in $B_0$ which satisfy 
\begin{align}\label{e:nearby-2}
	\sum_{i \in I} \pdiam(B_i)^{\idim+2} \geq (1+3\delta/4)^{-1}\pdiam(B_0)^{\idim+2}. 
\end{align}
Combing this with \eqref{e:coverage-nearby-2}, we have 
\begin{align}
	(1+3\delta/4)^{-1}\pdiam(B_0)^{\idim+2} \leq (1+\alpha) \sum_{i \in I} \pHmeasure^{\idim+2}(F \cap B_i) = (1+\alpha)\pHmeasure^{\idim+2}\left( F \cap \bigcup_{i \in I} B_i\right) .
\end{align}
Thus, by \eqref{e:coverage-nearby-1} and some simple algebra, we have
\begin{align}
	\pHmeasure^{\idim+2}\left(F \cap B_0 \setminus \bigcup_{i \in I}Q_i \right) &\leq 	\pHmeasure^{\idim+2}\left(F \cap B_0 \setminus \bigcup_{i \in I}B_i \right) \\
	&= \pHmeasure^{\idim+2}(F \cap B_0) - \pHmeasure^{\idim+2}\left( F \cap \bigcup_{i \in I} B_i \right) \\
	&\leq \left[ 1+\alpha - \frac{1}{(1+\alpha)(1+3\delta/4)}   \right] \pdiam(B_0)^{\idim+2} \\
	&\leq \delta \pdiam(B_0)^{\idim+2}
\end{align}
provided $\alpha_3 = \alpha_3(\delta)$ is sufficiently small. 
\end{proof}

\subsection{The main lemma} Fix, for the remainder of this section, a set $E$ which is parabolic Ahlfors $(\idim+2)$-regular with constant $\CAR > 1$ and satisfies the BAUP condition. Theorem \ref{BAUPBWGL} can be deduced from Lemma \ref{l:control-Schul-cubes} below by a variant of the \textit{one-third-trick} (this carried out near the end of the section). 

\begin{lemma}\label{l:control-Schul-cubes}
	If $\rho > 0$ is small enough, then for each $\eta > 0$ there is $C = C(\eta)$ such that, if $\cC$ is a system of Schul cubes on $E$ constructed with parameter $\rho$, we have 
	\begin{align}
		\sum_{\substack{R \in \cC(Q_0) \\ \bilatbeta_E(B(R)) > \eta }}  \pdiam(R)^{\idim+2} \leq C\pdiam(Q_0)^{\idim+2}, \quad Q_0 \in \cC. 
	\end{align}
\end{lemma}

Since $E$ satisfies the BAUP, so that the cubes which $\baupbeta_E(AB(Q)) > \ve$ satisfy a Carleson packing condition for all $A > 1$ and $\ve > 0$, Lemma \ref{l:control-Schul-cubes} follows easily from Lemma \ref{l:control-Schul-cubes-main} below. 

\begin{lemma}\label{l:control-Schul-cubes-main}
	If $\rho > 0$ small enough, then for each $\eta > 0$ there is $A > 1$, $\ve > 0$ and $C = C(A,\ve,\eta)$ such that the following holds. Suppose $\cC$ is a system of Schul cubes on $E$ constructed with parameter $\rho$. Let $Q_0 \in \cC$ such that $\baupbeta_E(AB(Q_0)) < \ve$ and let $\mathscr{T} \coloneqq \mathscr{T}(Q_0)$ be a maximal collection of cubes $Q \in \cC$ contained in $Q_0$ such that $\baupbeta_E(AB(Q)) > \ve$. Then, 
	\begin{align}
		\sum_{\substack{R \in \cC(Q_0) \setminus \cC(\mathscr{T}) \\ \bilatbeta_E(B(R)) > \eta }}  \pdiam(R)^{\idim+2} \leq C\pdiam(Q_0)^{\idim+2}.
	\end{align}
\end{lemma}

We begin the proof of Lemma \ref{l:control-Schul-cubes-main} in earnest. Fix $\eta ,\rho > 0$ (as in the statement of Lemma \ref{l:control-Schul-cubes-main}). We introduce constants $A, M > 1$ and $\alpha, \delta, \ve, \tau > 0$. 

\begin{remark}[Dependencies of the constant]\label{r:constants}
	The constants above will be chosen (suitably large in the case of $A, M$, and suitably small in the case of $\alpha, \delta, \ve , \rho$) with the following dependencies:
	\begin{align}
		\tau = \tau(\delta), \ \rho = \rho(\delta), \ M = M(\rho), \ A = A(M) , \ \alpha = \alpha(M,\delta,\eta) , \ \ve = \ve(A,\alpha,\delta,\eta). 
	\end{align} 
	The constant $\eta$ is given and fixed, and $\delta$ will be chosen independent of all other constants. To begin with, we will assume that these constants are chosen such that the results of Section \ref{s:BAUP-prelim} hold with such constants. We will also assume $\rho$ is small enough such that 
	\begin{align}
		\pdiam(Q)^{\idim+2} \leq (1+\delta)\pdiam(B(Q))^{\idim+2} \quad \mbox{ for all } Q \in \cC. 
	\end{align} 
	We will not specify this dependencies in the statements below, but we will indicate where they are needed in the proofs, other than what we have already mentioned above. 
\end{remark}

Finally, fix a system of Schul cubes $\cC$ on $E$ and let $Q_0 \in \cC$ and $\mathscr{T} = \mathscr{T}(Q_0)$ be as in the statement of Lemma \ref{l:control-Schul-cubes-main}. To simplify notation, we will write $\mathscr{S} \coloneqq \mathscr{C}(\mathscr{T})$ (recall \eqref{e:descendents}) and we will drop the dependency on $E$ in our notation $\bilatbeta$ and $\baupbeta$. Furthermore, let us define 
\begin{align}
	\cG \coloneqq \{Q \in \cC(Q_0) \colon \bilatbeta(MB(Q)) \leq \delta \alpha \} \quad \mbox{ and } \quad \cB \coloneqq \cC(Q_0) \setminus \cG.
\end{align}

Our first goal is, for any given cube in $Q_1 \in \cC \setminus \mathscr{S}$, to find a family $\cN$ of \textit{next} cubes which contain all but boundedly many of the non-flat sub-cubes of $Q_1$ and which has \textit{large total mass} in some suitable sense. This is split into two cases depending on whether $Q_1$ lies in $\cG$ or $\cB$. 

\begin{lemma}\label{l:next-when-not-flat}
	Suppose $Q_1 \in \cC \cap \cB \setminus \mathscr{S}$ satisfies $\baupbeta(AB(Q_1)) \leq \ve$. There are absolute constant $\Cnext, \Cnextlower$ independent of $Q_1$ and a collection of disjoint cubes $\cN \subseteq \cC(Q_1) \setminus \mathscr{S}$ such that the following holds.
	\begin{enumerate}
		\item Every $T \in \cC(Q_1) \setminus \mathscr{S}$ which is not contained in any cube from $\cN$ satisfies $\pdiam(T) \geq \Cnextlower \pdiam(Q_1)$. 
		\item We have 
		\begin{align}
			\pdiam(Q_1)^{\idim+2} \leq \Cnext\pHmeasure^{\idim+2}\left(E \cap Q_1 \setminus \bigcup_{T \in \cN} T  \right) + (1+\Cnext \delta)\sum_{T \in \cN \cap \cG} \pdiam(T)^{\idim+2}. 
		\end{align}
	\end{enumerate} 
\end{lemma}

\begin{proof}
	Let $\cN'$ be the cubes obtained by applying Lemma \ref{l:lots-of-flat-children} to the ball $B(Q_1)$, $\cN'' \coloneqq \cN' \setminus \mathscr{S}$, $\cN'''$ be the maximal set of cubes in $\cC(Q_1) \setminus \mathscr{S}$ which are disjoint from every cube in $\cN'$, and set 
	\begin{align}
		\cN \coloneqq \cN'' \cup \cN'''.
	\end{align}
	
	For Lemma \ref{l:next-when-not-flat} (1), let us fix $T \in \cC(Q_1) \setminus \mathscr{S}$ as in the statement. Since $\cN'''$ is the maximal collection of cubes in $\cC(Q_1) \setminus \mathscr{S}$ which are disjoint from $\cN'$, and $T$ is not contained in any cube from $\cN'''$ (by assumption), it must that $T$ has non-empty intersection with some cube in $T' \in \cN'$. We must have $T' \subseteq T$. Indeed, if $T' \in \cN' \cap \mathscr{S}$, then this follows from the fact that $T \not\in \mathscr{S} (= \mathscr{C}(\mathscr{S}))$. If $T' \in \cN' \setminus \mathscr{S} = \cN''$, then this follows from the fact that $T$ is not contained in any cube from $\cN$. By Lemma \ref{l:lots-of-flat-children}, we know that every cube in $\cN'' \subseteq \cN'$ has diameter comparable to $\pdiam(Q_1)$. Thus, the same must now be true of $T$. 
	
	For Lemma \ref{l:next-when-not-flat} (2), it follows from the definition of $\cN'$ and \eqref{e:size} that $\cN' \subseteq \cG$. By construction, every cube in $\cN' \cap \mathscr{S}$ is disjoint from any cube in $\mathscr{N}$.  Using this and \eqref{e:sum}, we have 
	\begin{align}
		\pdiam(Q_1)^{\idim+2} &\leq (1+\delta)\pdiam(B(Q_1))^{\idim+2} \leq (1+C\delta) \sum_{R \in \cN'} \pdiam(R)^{\idim+2} \\
		&= (1+C\delta)\sum_{R \in \cN' \cap \mathscr{S}} \pdiam(R)^{\idim+2} + (1+C\delta) \sum_{R \in \cN''} \pdiam(R)^{\idim+2}  \\ 
		&\leq C\pHmeasure^{\idim+2}\left(E \cap Q_1 \setminus \bigcup_{T \in \cN} T  \right) + (1+C\delta) \sum_{T \in \cN \cap \cG} \pdiam(T)^{\idim+2}.
	\end{align}
\end{proof}

\begin{lemma}\label{l:next-when-flat}
	Suppose $Q_1 \in \cC \cap \cG \setminus \mathscr{S}$ satisfies $\baupbeta(AB(Q_1)) \leq \ve$. There are absolute constants $\Cnextprime , \Cnextlowerprime$ and collections of disjoint cubes $\cF \subseteq \cC(Q_1) \setminus \mathscr{S}$ and $\cN \subseteq \cC(\cF) \setminus \mathscr{S}$ such that the following holds. 
	\begin{enumerate}
		\item For each $T \in \cC(Q_1) \setminus \mathscr{S}$ such that $\bilatbeta(B(T)) > \eta$, there exists $T' \in \cF$ such that $T \subseteq T'$.
		\item For each $T \in \cC(Q_1) \setminus \mathscr{S}$ such that $\bilatbeta(B(T)) > \eta$, if there does not exist $T'' \in \cN$ such that $T \subseteq T''$, then $\pdiam(T) \geq \Cnextlowerprime\pdiam(T')$, with $T' \in \cF$ as above.
		\item We have 
			\begin{align}\label{e:half}
			\hspace{1em} \pdiam(Q_1)^{\idim+2} \leq \Cnextprime\pHmeasure^{\idim+2}\left(E \cap Q_1 \setminus \bigcup_{T \in \cN} T \right) + \frac{2}{3}\sum_{T \in \cN} \pdiam(T)^{\idim+2}. 
		\end{align} 
	\end{enumerate}
\end{lemma}

We will carry out the proof of Lemma \ref{l:next-when-flat} over several lemmas. We can assume, without loss of generality, that $B(Q_1) = B_{\CSchulprime}(0)$. Since $Q_1 \in \cG$, there exists a $d$-plane $P_0$ through the origin such that 
\begin{align}\label{e:flat-at-the-top}
	\bilatbeta(MB(Q_1),P_0) \leq 2\delta \alpha. 
\end{align} 
Let $\{X_n\}_{n \in \N}$ be the sequence of subsets from which $\cC$ is constructed, and let $\cubes$ denote the collection of Christ-David cubes constructed with respect to this same sequence. By translating, we can assume that $0 \in X_n$ for each $n \in \N$. Finally, let \\
\begin{align}
	A' \coloneqq \frac{A}{148}
\end{align}

\begin{remark}\label{r:A'}
	With this choice of $A'$ as above, it is easy to check from the definitions in Lemma \ref{cubes} and Lemma \ref{l:Schul-cubes} that if $Q \in \cubes_k$ and $R \in \cC_k$ are such that $x_Q = x_R$, then $A'B_Q = AB(R)$. 
\end{remark}

    \noindent\textbf{The stopping-time regions.} We define a collection of cubes $S \subseteq \cubes$ as follows. Let $\ell \in \N$ such that 
    \begin{align}\label{e:choice-of-n}
    	\Couter\rho^{\ell+1} \leq \delta r_{B(Q_1)} < \Couter\rho^{\ell}.
    \end{align}
   	Let $\{Q(S_j)\}$ the maximal collection of cubes $R \in \cubes_\ell$ such that $10B_R \subseteq B(Q_1)$. Then, for each $j \in J$, let $S_j'$ be the stopping-time region constructed by first adding $Q(S_j)$ to $S_j'$ then, assuming $R \in S_j'$, adding $\Child(R)$ to $S_j'$ if 
    \begin{align}\label{e:stoppig-time-BAUP}
    	\bilatbeta(MB_{R'}) \leq \alpha \quad \mbox{ and } \quad \baupbeta(A'B_{R'}) \leq \ve \quad \mbox{ for all } R' \in \Child(R). 
    \end{align}
    Let $\{S_j\}_{j \in J}$ be the set of extended stopping-time regions obtained by applying Lemma \ref{l:smooth} with the constant $\tau$ to the collection $\{S_j'\}_{j \in J}$. Finally, set 
    \begin{align}\label{e:S'}
    	S' \coloneqq \bigcup_{j \in J} S_j' \quad \mbox{ and } \quad S \coloneqq \bigcup_{j \in J} S_j.
    \end{align}
    
    \begin{lemma}
    	For each $Q \in S$ we have 
    	\begin{align}\label{e:properties-of-S}
    		\bilatbeta(M'B_Q) \lesssim \alpha \quad \mbox{ and } \quad  \baupbeta(A''B_Q) \lesssim_\delta \ve. 
    	\end{align}
    	for some $M'$ and $A''$, which are constant multiples of $M$ and $A'$, respectively. 
    \end{lemma}
    
    \begin{proof}
    	If $Q \in S' \setminus \{Q(S_j)\}_{j \in J}$, then \eqref{e:properties-of-S} follows immediately from \eqref{e:stoppig-time-BAUP}. If $Q = Q(S_j)$ for some $j \in J$, then since $B_{Q(S_j)} \subseteq B(Q_1)$, our assumption on $Q_1$ in Lemma \ref{l:next-when-flat} and our choice of $\ell$ from \eqref{e:choice-of-n} imply (with an application of Lemma \ref{l:monotonicity}) \eqref{e:properties-of-S}. We finish with the case that $Q \in S \setminus S'$. By Lemma \ref{l:smooth}, there exists $R \in S'$ such that $B_Q \subseteq CB_R$ and $\ell(R) \sim \ell(Q)$. The first condition implies that $C^{-1}MB_Q \subseteq MB_R$ and $C^{-1}A'B_Q \subseteq A'B_R$. Since $R$ satisfies \eqref{e:properties-of-S} with constants $M$ and $A'$, Lemma \ref{l:monotonicity} implies that $Q$ satisfies \eqref{e:properties-of-S} with constant $M' \coloneqq C^{-1}M$ and $A'' \coloneqq C^{-1}A'$. 
    \end{proof}

    \noindent\textbf{The approximating set.} We will prove Lemma \ref{l:next-when-flat} by first building a Reifenberg flat set which approximates $S$, constructed via Theorem \ref{t:DT}. For this, we will need a CCBP. Given $k \geq 0$, let $s(k)$ to be the integer such that 
    \begin{align}\label{e:s(k)}
    	\Couter \rho^{s(k)} < r_k \leq \Couter \rho^{s(k)-1}
    \end{align}
    (recall that $r_k = 10^{-k}$ and $\Couter$ as in Lemma \ref{cubes}). Let $\mathscr{X}_k = \{x_{j,k}\}_{j \in J_k}$ be a maximal $r_k$-separated net in
    \begin{align}
    	\mathscr{C}_k = \left\{ x_Q \colon Q \in S \cap \cubes_{\max\{\ell, s(k)\}}\right\},
    \end{align} 
    which contains the origin, and let $B_{j,k} = B_{r_k}(x_{j,k})$. If $s(k) < \ell$, we let $P_{j,k} = x_{j,k} + P_0$. Otherwise, we let $P_{j,k} = P_Q$ where $Q \in S \cap \cubes_{s(k)}$ is such that $x_{j,k} = x_Q,$ and $P_Q$ is a plane such that $b\beta(M'B_Q,P_Q) \lesssim \alpha$ (recall that $B_Q = B_{\ell(Q)}(x_Q)$).

    \begin{lemma}
    	The triple $(P_0,\{B_{j,k}\} , \{P_{j,k}\})$ defines a $C\alpha$-CCBP. 
    \end{lemma}
    
    \begin{proof}
    	Definition \ref{d:CCBP} \eqref{CCBP1} is proven exactly as in Lemma \ref{l:CCBP}. Definition \ref{d:CCBP} \eqref{CCBP2} follows from \eqref{e:flat-at-the-top} and the definition of $P_{j,0}$ above (note that $s(0) < \ell$). If $k \geq 0$ is such that $s(k) \geq \ell$, then properties \eqref{CCBP3} and \eqref{CCBP4} in Definition \ref{d:CCBP} can be shown to hold with constant $C\alpha$ by a similar argument to the one employed in the proof of Lemma \ref{l:CCBP} (we shall omit the details). Suppose instead that $s(k) < \ell$. For Definition \ref{d:CCBP} \eqref{CCBP3}, it follows from \eqref{e:flat-at-the-top} and \eqref{e:choice-of-n} that for any $i,j \in J_k$ we have 
    	\begin{align}\label{e:top-planes}
    		d_{100B_{i,k}}(P_{i,k} , P_{j,k}) &\lesssim r_k^{-1} d_{B_{100}(0)}(x_{i,k} + P_0 , x_{j,k} + P_0) \\
    		&\lesssim \rho^{-s(k)} \left(\pdist(x_{i,k},P_0) + \pdist(x_{j,k},P_0) \right)  \lesssim \delta^{-1}\delta \alpha = \alpha. 
    	\end{align}
    	For Definition \ref{d:CCBP} \eqref{CCBP4}, let $i \in J_k$ and $j \in J_{k+1}$. If $s(k+1) < \ell$ then \eqref{CCBP4} follows from a similar estimate to \eqref{e:top-planes}. If $s(k+1) = \ell$, let $Q \in \cubes_\ell$ such that $P_{j,k+1} = P_Q$. Then, by Lemma \ref{l:beta-angle} and \eqref{e:flat-at-the-top}, we have 
    	\begin{align}
    		d_{20B_{i,k}}(P_{i,k},P_{j,k+1}) \lesssim r_k^{-1}\bilatbeta(B_{100}(0) , P_{i,k}) + \bilatbeta(MB_Q,P_Q) \lesssim \alpha. 
    	\end{align} 
    \end{proof}
    
    Let $\Sigma$ be the $C\alpha$-Reifenberg flat set constructed by applying Theorem \ref{t:DT} to the above CCBP.
    
    \begin{lemma}\label{l:properties-of-sigma}
    	There is a constant $\Cball$ such that 
    	\begin{align}\label{e:sigma-large-in-ball}
    		\Sigma \cap (1-\Cball\delta)B(Q_1) \subseteq z(S) \cup \bigcup_{R \in \min(S)} 10 B_R 
    	\end{align}
    	and 
    	\begin{align}\label{e:sigma-locally-upper-reg}
    		\pHmeasure^{\idim+2}(\Sigma \cap B) \leq (1+\Cball \alpha)\pdiam(B)^{\idim+2}
    	\end{align}
    	for any ball $B$ with $B \subseteq 10 B_R$ for some $R \in \min(S)$.
    \end{lemma}
    
    \begin{proof}
    	We will start with \eqref{e:sigma-large-in-ball}, for which we fix some $x \in \Sigma \cap (1-\Cball\delta)B(Q_1)$ with $\Cball$ to be chosen sufficiently large. We can assume that $x \not\in z(S)$ otherwise there is nothing to show. First, we claim that 
    	\begin{align}\label{e:J_k}
    		\mbox{ there exists some } Q \in \cubes_\ell \mbox{ such that } x \in 3B_Q \mbox{ and } 20B_Q \subseteq B(Q_1). 
    	\end{align}
    	Indeed, let $j_0 \in J_0$ such that $x_{j_0,0} = 0$ (which exists by assumption), so that $P_{j_0,0} = P_0$. Recalling that $B(Q_1) \coloneqq B_{\CSchulprime}(0) \subseteq B_{j_0,0}$, it follows from \eqref{Sigma_kPlane} and \eqref{Sigma_kSigma} that $\pdist(x,P_0) \lesssim \alpha$. By choosing $\alpha = \alpha(\delta,M)$ sufficiently small, this and \eqref{e:flat-at-the-top} provides a point $x' \in E \cap (1-\delta/2)B(Q_1)$ such that $|x - x'| \leq \CSchulprime \delta/2$. Let $Q \in \cD_\ell$ such that $|x_Q -x'| < 2 \rho^\ell$. By the previous two inequalities, we have $x \in 3B_Q$. It only remains to show that $20B_Q \subseteq B(Q_1)$ (so that $Q \in S$). Let $y \in 20B_Q$. By \eqref{e:choice-of-n}, and the above inequalities, for $\Cball$ large enough we have 
    	\begin{align}
    		|y| \leq |y-x_Q| + |x_Q - x'| + |x'-x| + |x| \leq C \CSchulprime \delta + (1-\Cball \delta)\CSchulprime  \leq \CSchulprime,
    	\end{align}
    	as required.
    	
    	Let $k \geq \ell$ be the smallest such integer for which there exists $R \in S \cap \cubes_k$ with $x \in (4/\rho)B_R$ (which exists by Lemma \eqref{e:J_k}), and let us fix such a cube. By a similar argument to above ($\Sigma$ is close to a plane from our CCBP at the scale of $R$, which is close to the set $E$), we can find a cube $R'\in \cubes_k$ such that $x \in 3B_{R'}$. Since $3B_{R'} \subseteq \bigcup_{R'' \in \Child(R')} 4\rho^{-1}B_{R''}$, it follows that no child of $R'$ can be in $S$ (otherwise we contradict the minimality of $k$). Let $T \in \cubes_\ell$ such that $R' \subseteq T$. Since $x \in 3B_{R'}$ and no child of $R'$ is in $S$, it is enough to show that $T \in S$ (since then there must exist a cube in $\min(S)$ between $R$ and $T$). We have that $T \in S$ since, for each $y \in 10B_T$, 
    	\begin{align}
    		|y| &\leq |y-x_T| + |x_T - x| + |x - x_Q| + | x_Q | \\
    		&\leq 10\ell(Q) + 3\ell(Q) + 3\ell(Q) + (\CSchulprime-20\ell(Q)) \leq \CSchulprime. 
    	\end{align}

    	To prove \eqref{e:sigma-locally-upper-reg}, let us fix $R \in \min(S)$. Let $k \geq 0$ such that $\ell(R) < r_k \leq \rho^{-1}\ell(R)$ and let $j \in J_k$ such that $x_R \in B_{j,k}$. By construction, $\bilatbeta(B_{j,k},P_{j,k}) \lesssim \alpha$, thus, there exists $z' \in P_{j,k} \cap B_{j,k}$ such that $|x_{j,k} - z'| \lesssim \alpha r_k$. By \eqref{Sigma_kPlane}, there is $z \in \Sigma_k \cap B_{j,k}$ such that $|z-z'| \lesssim \alpha r_k$. By the triangle inequality, 
    	\begin{align}\label{e:equals-graph}
    		\Sigma_k \cap 17B_R \subseteq \Sigma_k \cap B_{19r_k}(z) .
    	\end{align}
    	By \eqref{e:equals-graph} and Theorem \ref{t:DT} \eqref{DT6}, $\Sigma_k$ is upper parabolic Ahlfors regular $(\idim+2)$-regular in $17B_R$ with constant $(1+C\alpha)$ for some $C \geq 1$.  By Lemma \ref{l:smooth} (1), there is an absolute constant $m \geq 0$ (independent of $R$) such that $17B_R \cap V_{k+m'}^{11} = \emptyset$ for all $m' \geq m$. Thus, by the fact that $\sigma_{k+m'}$ is $(1+C\alpha)$-bi-Lipschitz for each $0 \leq m' \leq m$ and the aforementioned fact that holds for each $m' \geq m$, we also have that $\Sigma$ is upper parabolic Ahlfors regular $(\idim+2)$-regular in $10B_R$ with constant $(1+C\alpha)$. This is equivalent to property \eqref{e:sigma-locally-upper-reg}. 
    \end{proof}

    \begin{lemma}\label{l:family-of-cubes}
    	For each $R \in \min(S)$, there exists a collection $\cC_R$ of Schul cubes contained in $10B_R$ such that, for each $T\in \cC_R$
    	\begin{align}\label{e:T-is-large}
    		\pdiam(T) \geq \Cind \pdiam(R), 
    	\end{align}
    	\begin{align}\label{e:T-upper-regular}
    		\pHmeasure^{\idim+2}(\Sigma \cap T) \leq (1+C\alpha)\pdiam(T)^{\idim+2},
    	\end{align}
    	for some constant $\Cind$ independent of $\alpha$, and
    	\begin{align}\label{e:covers-most}
    		\pHmeasure^{\idim+2}\left( \Sigma \cap 10B_R \setminus \bigcup_{T \in \cC_R} T \right) \leq \delta \pdiam(R)^{\idim+2}.
    	\end{align}
    \end{lemma}

    \begin{proof}
    	Fix some $R \in \min(S)$. To construct the set $\cC_R$ satisfying \eqref{e:T-is-large} and \eqref{e:covers-most}, it suffices to check the hypotheses of Lemma \ref{l:large-coverage-nearby} for the ball $B_0 = 10B_R$ and the set $F = \Sigma$ (property \eqref{e:T-upper-regular} is immediate from \eqref{e:sigma-locally-upper-reg} due the fact that $T \subseteq 10B_R$). For this, one deduces from Lemma \ref{l:monotonicity} and \eqref{e:stoppig-time-BAUP} that $\baupbeta(10B_R) \lesssim_{A,\delta} \ve$ and from Lemma \ref{l:properties-of-sigma} that $d_{10B_R}(E,\Sigma) \lesssim \alpha$. The inequality \eqref{e:coverage-nearby-1} holds with constant $C\alpha$ by \eqref{e:sigma-locally-upper-reg}, and the inequality \eqref{e:coverage-nearby-2} holds with constant $C\alpha$ by Lemma \ref{l:big-projection-RF}, since $\Sigma$ is $C\alpha$-Reifenberg flat. The lemma now follows by choosing $\alpha = \alpha(\delta)$ and $\ve = \ve(A,\delta)$ sufficiently small.  
    \end{proof}
    
    \noindent\textbf{The set $\mathscr{F}$.} We are now ready to define the family $\cF$, as in the statement of Lemma \ref{l:next-when-flat}. For this, it will be useful to partition the set $\min(S')$ (recall, these are the cubes obtained via the stopping-time procedure before the smoothing). Let 
    \begin{align}
    	{\min}_1(S') &\coloneqq \{R \in \min(S') \colon \mbox{there exists } R' \in \Child(R) \mbox{ such that } \baupbeta(A'B_{R'}) > \ve \}; \\
    	{\min}_2(S') &\coloneqq \min(S') \setminus {\min}_1(S'). 
    \end{align}
   	
   	With this, we can also partition the cubes in $\min(S)$ depending which family its ancestor lies in. We also wish to identify cubes which live \textit{too close} to the boundary of their ancestor. Indeed, for $R \in \min(S)$, let $\mathcal{A}(R) \in \min(S')$ be the cube such that $R \subseteq \mathcal{A}(R)$. Then, set  
   	\begin{align}
   		\mathcal{B}(S) &\coloneqq \{R \in \min(S) \colon \pdist(R,\mathcal{A}(R)^c) < 20\tau \ell(\mathcal{A}(R))\}; \\
   		{\min}_1(S) &\coloneqq \{R \in \min(S) \setminus \mathcal{B}(S) \colon \mathcal{A}(R) \in {\min}_1(S')\}; \\
   		{\min}_2(S) &\coloneqq \{R \in \min(S) \setminus \mathcal{B}(S) \colon \mathcal{A}(R) \in {\min}_2(S')\}.
   	\end{align}
    
    Finally, define $\cF$ to be the maximal collection of cubes in 
    \begin{align}\label{e:cF}
    	\left\{T \in \cC(Q_1) \setminus \mathscr{S} \colon \bilatbeta(B(T)) > \eta \right\} \cup \bigcup_{R \in {\min}_2(S)} \cC_R,
    \end{align}	
    with $\mathscr{C}_R$ as in Lemma \ref{l:family-of-cubes}.
    
    \begin{lemma}\label{l:flat-above-min}
    	Let $T \in \cC_R$ for some $R \in \min_2(S)$, and suppose that $T' \in \cC(Q_1)$ is such that $T \subseteq T'$. Then,  $\bilatbeta(B(T')) < \eta$.  
    \end{lemma}
    
    \begin{proof}
    	Suppose, towards a contradiction, that $\bilatbeta(B(T')) \geq \eta$. Let $Q \in S$ be the smallest ancestor of $R$ (possibly itself) such that $B(T') \subseteq M'B_Q$, so that $\pdiam(T') \sim_M \pdiam(Q)$. It then follows from Lemma \ref{l:monotonicity} and \eqref{e:properties-of-S} that $\eta \leq \bilatbeta(B(T'))  \lesssim_M \bilatbeta(M'B_Q) \lesssim \alpha.$ This is a contradiction for $\alpha = \alpha(M,\eta)$ small enough. 
    \end{proof}
    
    \begin{lemma}\label{l:flat-near-floor}
    	Let $T \in \cF$ and suppose that $z(S) \cap T \neq\emptyset$. Then, $\bilatbeta(B(T)) < \eta$.
    \end{lemma}
    
    \begin{proof}
    	Let $x \in z(S) \cap T$ and (since $x$ is contained in arbitrarily small cubes from $S$) let $Q \in S$ such that $x \in Q$ and $\rho \pdiam(Q) \leq \pdiam(T) \leq \pdiam(Q)$. By the triangle inequality, $B(T) \subseteq M'B_Q$. Thus, recalling Lemma \ref{l:monotonicity} and \eqref{e:stoppig-time-BAUP}, we have $\bilatbeta(B(T)) \lesssim_M \alpha$. The lemma now follows so long as $\alpha = \alpha(M,\eta)$ is sufficiently small. 
    \end{proof}

    \begin{lemma}\label{l:z(S)}
    	We have 
    	\begin{align}
    		z(S) \subseteq E \cap Q_1 \setminus \bigcup_{T \in \cF} T.
    	\end{align}
    \end{lemma}
    
    \begin{proof}
    	By definition, $z(S) \subseteq E \cap Q_1$. Suppose, towards a contradiction, that there exists $T \in \cF$ such that $z(S) \cap T \neq\emptyset$. By Lemma \ref{l:flat-near-floor}, we have $\bilatbeta(B(T)) < \eta$. Thus, since $\cF$ is a subset of \eqref{e:cF}, there must exists $R \in \min_2(S)$ such that $T \in \cC_R$. However, by construction, $T \subseteq 10B_R$, and by Lemma \ref{l:smooth} (2), $z(S) \cap 10B_R = \emptyset$. This is a contradiction.
    \end{proof}

    For the next lemma we adopt the notation $\partial_\lambda Q = \{x \in Q: \dist(x,\partial Q) <\lambda \}$ for a given Christ-David cube $Q$.
    
    \begin{lemma}
    	We have 
    	\begin{align}\label{e:bad-control}
    		\pHmeasure^{\idim+2}\left( \Sigma \cap \bigcup_{R \in \mathcal{B}(S)} 10B_R \right) \lesssim \delta \pdiam(Q_1)^{\idim+2}. 
    	\end{align}
    \end{lemma}
    
    \begin{proof}
    	Recall from Lemma \ref{l:smooth} (5) that $\ell(R) < \tau \ell(\mathcal{A}(R))$ for all $R \in \min(S)$. In particular, for each $R \in \mathcal{B}(S)$ we have $R \subseteq \partial_{22\tau \ell(R)} \mathcal{A}(R)$. Choosing $\tau = \tau(\delta)$ small enough, Lemma \ref{cubes} and \eqref{e:sigma-locally-upper-reg} show that the left-hand side of \eqref{e:bad-control} is at most 
    	\begin{align}
    		\sum_{R \in \mathcal{B}(S)} \pHmeasure^{\idim+2}(\Sigma \cap 10B_R) &\lesssim_{C_0} \sum_{R \in \mathcal{B}(S)}  \pHmeasure^{\idim+2}(R) \leq \sum_{Q \in \min(S')} \pHmeasure^{\idim+2}(\partial_{22\tau\ell(Q)}Q)   \\
    		&\lesssim_{C_0} \delta  \pdiam(Q_1)^{\idim+2}.
    	\end{align}
    \end{proof}
    
    \begin{lemma}
    	We have 
    	\begin{align}\label{e:min1}
    		\pHmeasure^{\idim+2}\left(\Sigma \cap \bigcup_{R \in {\min}_1(S)} 10B_R\right) \lesssim \pHmeasure^{\idim+2}\left(E \cap Q_1 \setminus \bigcup_{T \in \cF} T \right).
    	\end{align}
    \end{lemma}

    \begin{proof}
    	Fix some $Q \in \min_1(S')$ for the moment. We claim that
    	\begin{align}\label{e:F(Q)}
    		\mbox{there exists $F(Q) \subseteq Q \setminus \bigcup_{T \in \cF} T$ such that $\pHmeasure^{\idim+2}(F(Q)) \gtrsim \pHmeasure^{\idim+2}(Q)$.}
    	\end{align}   
    	
    	We start by defining $F(Q)$. By the definition of $\min_1(S')$, there exists $Q' \in \Child(Q)$ such that $\baupbeta(A'B_{Q'}) > \ve$. Let $k \in \N$ such that $Q' \in \cubes_k$, let $R \in \mathscr{C}_k$ be the Schul cube with the same centre as $Q'$, and let $F(Q) \coloneqq E \cap B(R)$. After some definition chasing (i.e. comparing the constants $\Cinner$ and $\CSchulprime$), it is easy to check that $B(R) \subseteq \Cinner B_{Q'}$. This implies $F(Q) \subseteq Q$. 
    	
    	To see that $F(Q)$ is disjoint from all the cubes in $\cF$ we proceed as follows. Recalling Remark \ref{r:A'}, we have that $\baupbeta(AB(R)) > \ve$, i.e., $R \in \mathscr{S}.$ Since $R \in \mathscr{S}$, it cannot intersect any cube in the left-most set of \eqref{e:cF}. If $R' \in \min_2(S)$, it follows from Lemma \ref{l:smooth} (5), that $\ell(R') < \tau \ell(\mathcal{A}(R'))$. Moreover, since $R' \not\in \mathcal{B}(S)$, we have $\pdist(R',\mathcal{A}(R')) > 20\tau\ell(\mathcal{A}(R'))$. These two fact imply $\pdist(10B_{R'} , \mathcal{A}(R')^c) > 0$ which is enough to show that $R$ does not intersect any $T \in \cC_{R'}$ (since $T \subseteq 10B_{R'}$). 
    	
    	Finally, since by Lemma \ref{l:family-of-cubes} we have $\pdiam(R) \gtrsim \pdiam(Q)$, it follows by Ahlfors regularity that $\pHmeasure^{\idim+2}(E \cap B(R)) \gtrsim \pHmeasure^{\idim+2}(Q)$, which completes the proof of \eqref{e:F(Q)}. 
    	
    	Now, applying \eqref{e:T-upper-regular} and \eqref{e:F(Q)}, the left-hand side of \eqref{e:min1} is at most 
    	\begin{align}
    		\sum_{R \in {\min}_1(S)} \pHmeasure^{\idim+2}(\Sigma \cap 10B_R) &\lesssim_{C_0} \sum_{Q \in \min(S')} \sum_{\substack{R \in {\min}_1(S) \\ \mathcal{A}(R) = Q}} \pHmeasure^{\idim+2}(R) \lesssim_{C_0} \sum_{Q \in \min(S')} \pHmeasure^{\idim+2}(F(Q)) \\
    		&\leq \pHmeasure^{\idim+2}\left( E \cap Q_1 \setminus \bigcup_{T \in \cF} T \right) .
    	\end{align}
    \end{proof}
    
    \begin{lemma}\label{l:min-2}
    	There exists $\Cmintwo > 1$ such that 
    	\begin{align}\label{e:Claim3}
    		\hspace{1em}\pHmeasure^{\idim+2}\left(\Sigma \cap \bigcup_{R \in {\min}_2(S)} 10B_R\right) \leq (1+\Cmintwo\alpha) \sum_{T \in \cF'} \pdiam(T)^{\idim+2} + \Cmintwo\delta \pdiam(Q_1)^{\idim+2},
    	\end{align}
    	where 
    	\begin{align}
    		\cF' = \{T \in \cF \colon \mbox{ there exists } R \in {\min}_2(S) \mbox{ such that } T \in \cC_R\}.
    	\end{align}
    \end{lemma}
    
    \begin{proof}
    	If $R \in \min_2(S)$ and $T \in \cC_R$, then it follows by maximality and Lemma \ref{l:flat-above-min} that there exists some $R' \in \min_2(S)$ and $T' \in \cC_{R'} \cap \cF \subseteq \cF'$ such that $T \subseteq T'$. In particular, 
    	\begin{align}
    		\bigcup_{T \in \cC_R} T\subseteq \bigcup_{T \in \cF'} T \quad \mbox{ for any } \quad R \in {\min}_2(S). 
    	\end{align}
    	Thus, by \eqref{e:T-upper-regular} and \eqref{e:covers-most}, the left-hand side of \eqref{e:Claim3} is at most 
    	\begin{align}
    		&\pHmeasure^{\idim+2}\left(\Sigma \cap \bigcup_{T \in \cF'} T\right) + \pHmeasure^{\idim+2}\left(\Sigma \cap \bigcup_{R \in {\min}_2(S)} 10B_R \setminus \bigcup_{T \in \cF'} T\right) \\
    		&\hspace{3em} \leq (1+C\alpha) \sum_{T \in \cF'} \pdiam(T)^{\idim+2} + \sum_{R \in {\min}_2(S)} \pHmeasure^{\idim+2}\left(\Sigma \cap 10B_R \setminus \bigcup_{T \in \cF'} T\right) \\
    		&\hspace{3em} \leq (1+C\alpha)\sum_{T \in \cF'} \pdiam(T)^{\idim+2} + \delta \sum_{R \in {\min}_2(S)} \pdiam(R)^{\idim+2}\\
    		&\hspace{3em} \leq (1+C\alpha)\sum_{T \in \cF'} \pdiam(T)^{\idim+2} + C\delta \pdiam(Q_1)^{\idim+2}.
    	\end{align}
    \end{proof}
    
    \begin{lemma}\label{l:nearby-non-flat}
    	For each $T \in \cF'$, there is a ball $B$ centered on $E$ such that $T \subseteq KB$, $\pdiam(T) \geq K^{-1}\pdiam(B)$ and
    	\begin{align}\label{e:nearby-not-flat}
    		\bilatbeta(MB) > \alpha, \quad \bilatbeta(2MB) \lesssim \alpha \quad \mbox{ and } \quad \baupbeta(A'B) \lesssim_\delta \ve, 
    	\end{align}
    	for some constant $K$ independent of both $T$ and $\alpha$.
    \end{lemma}
    
    \begin{proof}
    	Fix some $T \in \cF'$. Let $R \in \min_2(S)$ such that $T \in \cC_R$ and let $Q \coloneqq \mathcal{A}(R) \in \min_2(S')$. By the definition of $\min_2(S')$ and an application of Lemma \ref{l:monotonicity} and \eqref{e:properties-of-S}, there exists $Q' \in \Child(Q)$ such that, if $B \coloneqq B_{Q'}$, then $B$ satisfies the estimates in \eqref{e:nearby-not-flat}. Since $T \subseteq 10B_R$, $R \subseteq Q$ and $Q' \in \Child(Q)$, it follows from the triangle inequality that there exists an absolute constant $K_1 > 1$ such that $T \subseteq K_1B$. Since $R \not\in \mathcal{B}(S)$, we have by Lemma \ref{l:smooth} (5) that $\pdiam(R) \gtrsim_\tau \pdiam(Q) \gtrsim \pdiam(Q')$. In particular, recalling from Lemma \ref{l:family-of-cubes} that $\pdiam(T) \geq \Cind \pdiam(R)$, we have $\pdiam(T) \geq \pdiam(B)/K_2$ for some $K_2 \geq 1$ independent of $\alpha$. Taking $K \coloneqq \max\{K_1,K_2\}$ finishes the proof.
    \end{proof}

    \begin{lemma}\label{l:cF'-not-flat}
	For each $T \in \cF'$ we have $\bilatbeta(M''B(T)) \sim \alpha$ and $\baupbeta(A'''B(T)) \lesssim_\delta \ve$ where 
	\begin{align}
		M'' = M/K \quad \mbox{ and } \quad A''' \coloneqq A/K
	\end{align}
	and $K$ is the constant from Lemma \ref{l:nearby-non-flat}.
\end{lemma}

\begin{proof}
	Fix some $T \in \cF'$ and let $B$ and $K$ (independent of $\alpha$) be the ball and the constant obtained by applying Lemma \ref{l:nearby-non-flat} to $T$. By translating, we can assume that $x_B = 0$. Since $B(T) \subseteq KB$, it is immediate from Lemma \ref{l:monotonicity}, \eqref{e:properties-Q} and \eqref{e:nearby-not-flat} that $\bilatbeta(M''B(T)) \lesssim \alpha$ and $\baupbeta(A'''B(T)) \lesssim_\delta \ve$. Thus, we need only check that $\bilatbeta(M''B(T)) \gtrsim \alpha$. For this, let $\kappa > 0$ (to be chosen sufficiently small) and suppose towards a contradiction that there exists a $d$-plane $P_T$ such that
	\begin{align}\label{e:flat-ball}
		\bilatbeta(M''B(T),P_T) < \kappa \alpha.
	\end{align}
	By the properties of $B$ (from \eqref{e:nearby-not-flat}), there exists a vertical plane $P_B\in\verticalplanes$ through $x_B$ such that 
	\begin{align}\label{e:flat-in-B}
		d_{2MB}(E,P_B) \lesssim \alpha 
	\end{align}
	and a collection of vertical planes $\mathscr{P}_B \subseteq\verticalplanes$ such that, if $U \coloneqq \bigcup_{P \in \mathscr{P}_B} P$, then 
	\begin{align}\label{e:small-beta-U}
		d_{2MB}(E,U) \leq (A'/M) d_{A'B}(E,U) \lesssim_\delta A'\ve. 
	\end{align} 
	Since $A'=A/148$, we take $A\geq296M$, so that $A'B\supseteq2MB$. Choose $\alpha$ sufficiently small in \eqref{e:flat-in-B}, and then $\ve$ sufficiently small in \eqref{e:small-beta-U}, that $d_{2MB}(E,P_B)\leq1/4$ and $d_{2MB}(E,U)\leq1/4$, as permitted by Remark \ref{r:constants}. A point of $U\cap(3/2)MB$ or $P_B\cap(3/2)MB$ has distance at most $Mr_B/4$ from $E$, whereas its distance from every point outside $2MB$ is at least $Mr_B/2$. Thus its distance to $E\cap2MB$ equals its distance to $E$. Applying the triangle inequality in both directions and the definition of $d_{(3/2)MB}$ gives
	$$d_{(3/2)MB}(U,P_B)\leq\frac{4}{3}(d_{2MB}(E,U)+d_{2MB}(E,P_B)).$$
	Together with \eqref{e:flat-in-B} and \eqref{e:small-beta-U}, this yields 
	\begin{align}\label{e:U-P_B}
		d_{(3/2)MB}(U,P_B) \lesssim_\delta \alpha + A'\ve. 
	\end{align}

	 Fix some $P \in \mathscr{P}_B$ such that $P \cap (4/3)MB \neq\emptyset$. Denote by $P' , P_{B}',P_T' \subseteq \R^{\adim}$ the $\idim$-dimensional planes such that 
	 \[P = P' \times \R, \quad P_B = P_B' \times \R \quad \mbox{and} \quad P_T = P_T' \times \R. \]
	
	\noindent\textbf{Claim 1:} We have 
	\begin{align}\label{e:close-to-P_B}
		d_{(5/4)MB}(P,P_B) \lesssim_{\delta} \alpha + A'\ve. 
	\end{align}
	
	Indeed, that every point $x \in P \cap (5/4)MB$ satisfies $\pdist(x,P_B) \lesssim_{\delta} (\alpha + A'\ve)r_B$ follows immediately from \eqref{e:U-P_B}. To see that every point $y \in P_B \cap (5/4)MB$ satisfies $\pdist(x,P) \lesssim_{\delta} (\alpha + A'\ve)Mr_B$, we proceed as follows. First, since the planes are $t$-independent, it suffices to show prove \eqref{e:close-to-P_B} with $(P,P_B)$ replaced by $(P',P_B')$. For this, fix some $X_0 \in P' \cap (4/3)MB$ (which exists since $P \cap (4M/3)B \neq\emptyset$) and let $\{X_1,\dots,X_{\idim-1}\} \subseteq P' \cap (3/2)MB$ such that the vectors $\{X_i-X_0\}_i$ are mutually orthogonal in and $|X_i - X_0| \sim Mr_B$.  By \eqref{e:U-P_B}, there are corresponding points $Y \coloneqq\{Y_0,\dots,Y_{\idim-1}\} \subseteq P_B'$ such that $|X_i - Y_i| \lesssim_{\delta} (\alpha + A'\ve)Mr_B$. For $\alpha = \alpha(\delta)$ and $\ve = \ve(A',\delta)$ sufficiently small, this set satisfies
	\begin{align}
		\diam(Y) \sim Mr_B \quad \mbox{ and } \quad \min\{\dist(Y_i,\mbox{span} \{Y \setminus \{Y_i\}\})\} /\diam(Y') \sim 1. 
	\end{align}
	Since $\pdist(Y_i,P') \lesssim_{\delta} (\alpha + A'\ve)Mr_B$ for each $i \in \{0,\dots,\idim-1\}$, this give the desired conclusion by, for example, \cite[Lemma 6.4]{azzam2015characterization}. \\
	
	\noindent\textbf{Claim 2:} We have 
	\begin{align}
		P \cap (1/2)B(T) \neq\emptyset. 
	\end{align}
	
	Indeed, assuming $M$ is large enough such that $B(T) \subseteq MB$, it follows from \eqref{e:flat-in-B} that for $\alpha = \alpha(K,M)$ small enough we have 
	\begin{align}
		\pdist(x_{T},P_B) \lesssim \alpha M'r_B \lesssim \alpha KM\pdiam(T) \leq \pdiam(T)/100. 
	\end{align}
	If $y$ is the point in $P_B$ which is closest to $x_T$, then we have $y \in P_B \cap (5/4)MB$ for $M$ sufficiently large. Thus, by \eqref{e:close-to-P_B} and assuming $\alpha = \alpha(\delta)$ and $\ve = \ve(A',\ve)$ are small enough, there exists $z \in P$ such that $|y-z| \leq \pdiam(T)/100$. The claim now follows by the triangle inequality. \\
	
	\noindent\textbf{Claim 3:} We have  
	\begin{align}
		d_{(M''/2)B(T)}(P,P_T) \lesssim \kappa\alpha. 
	\end{align}
	
	Indeed, for $\ve = \ve(\alpha,\kappa,A',K)$ small enough, it follows directly from \eqref{e:flat-ball} and \eqref{e:small-beta-U} that every $y \in P \cap (M''/2)B(T) \subseteq P \cap 2MB$ satisfies $\pdist(y,P_T) \lesssim M'' \kappa \alpha$ . The other direction is similar to the proof of Claim 1. As in that case, we will consider instead the planes $P'$ and $P_T'$. Let $\{X_0,\dots,X_{\idim-1}\} \subseteq P'\cap (3M''/4)B(T)$ such that $\{X_i - X_0\}_i$ are mutual orthogonal and $|X_i - X_0| \sim M''\pdiam(T)$ (such a collection of points exists by Claim 2). By \eqref{e:small-beta-U} and \eqref{e:flat-ball} there exist a collection of corresponding points $Y = \{Y_0,\dots,Y_{\idim-1}\} \subseteq P_T$ with $|X_i - Y_i| \lesssim \kappa\alpha M''\pdiam(T)$ for all $i \in \{0,\dots,\idim-1\}$, such that 
	\begin{align}
		\diam(Y) \sim M'' \diam(T)\quad \mbox{ and } \quad  \min\{\dist(Y_i,\mbox{span}\{ Y \setminus \{Y_i\}\})\} /\diam(Y) \sim 1.
	\end{align}
	The estimate follows again from \cite[Lemma 6.4]{azzam2015characterization}. \\
	
	Since $\pdiam(T) \gtrsim \pdiam(B)$, there exists $\lambda \geq 1$ such that $2MB \subseteq \lambda (M''/2)B(T)$. It is then easy to check that Claim 3 implies 
	\begin{align}\label{e:last-angle}
		d_{2MB}(P,P_T) \lesssim \kappa \alpha.
	\end{align}
	
	\noindent\textbf{Claim 4:} We have 
	\begin{align}
		d_{MB}(E,P_T) \lesssim \kappa\alpha + A'\ve . 
	\end{align}
	
	For the first direction, fix $y \in E \cap MB$. By \eqref{e:small-beta-U}, there exists $P \in \mathscr{P}_B$ and $z \in P \cap (5/4)MB$ such that $|y-z| \lesssim \ve A' r_B$. By \eqref{e:last-angle}, we have $\pdist(z,P_T) \lesssim M\kappa \alpha$. Thus, $\pdist(y,P_T) \lesssim (\kappa \alpha + A'\ve)Mr_B$. For second direction, fix some $x \in P_T \cap MB$. By \eqref{e:U-P_B}, there is some $P' \in \mathscr{P}_B$ such that $P' \cap (4/3)MB \neq\emptyset$. By \eqref{e:last-angle}, there exists $w \in P'$ such that $|x-w| \lesssim M\kappa \alpha r_B$. By \eqref{e:small-beta-U}, there is $v \in E$ such that $|v-w| \lesssim A'\ve r_B$. These estimates imply $\pdist(x,E) \lesssim (\kappa \alpha + A'\ve)Mr_B$. \\

    Claim 4 contradicts the first property in \eqref{e:nearby-not-flat} if $\kappa \ll 1$ and $\ve \ll \alpha$ (depending on $A'$) This finishes the proof of the lemma.
    \end{proof}
    
    \noindent\textbf{The set $\cN$.} We are now ready to define the set $\cN$, after which, we will complete the proof of Lemma \ref{l:next-when-flat}. Define
    \begin{align}
    	\cN \coloneqq \bigcup_{T \in \cF} \cN(T) 
    \end{align}
    as follows. First, set 
    \begin{align}\label{e:N(T)-1}
    	\cN(T) = \{T\} \quad \mbox{ for all } \quad T \in \cF \setminus \cF'. 
    \end{align}
    For $T \in \cF'$, let $\cN'(T)$ be the set of Schul cubes obtained by applying Lemma \ref{l:lots-of-flat-children} to the ball $B(T)$, $\cN''(T) \coloneqq \cN'(T) \setminus \mathscr{S}$, and $\cN'''(T)$ be the maximal collection of Schul cubes contained in $T$ which are disjoint from the cubes in $\cN'(T)$. Then, let 
    \begin{align}
    	\cN(T) \coloneqq \cN''(T) \cup \cN'''(T).
    \end{align}
    
    \begin{proof}[Proof of Lemma \ref{l:next-when-flat}]
    	Recalling that $\cF$ is chosen to be the maximal collection of cube in \eqref{e:cF}, Lemma \ref{l:next-when-flat} (1) follows immediately by maximality. The proof of Lemma \ref{l:next-when-flat} (2) is similar to the proof of Lemma \ref{l:next-when-not-flat} (2). Fix $T \in \cC(Q_1) \setminus \mathscr{S}$ as in the statement, and let $Q \in \cF$ such that $T \subseteq Q$ (which exists by Lemma \ref{l:next-when-flat} (1)). It must be that $Q \in \cF'$ otherwise $Q \in \cN$ by \eqref{e:N(T)-1} (and $Q$ cannot be in $\cN$ by assumption). Since $\cN'''(Q)$ is the maximal collection of cubes in $Q$ which is disjoint from $\cN'(Q)$, and $T$ is not contained in any cube from $\cN'''(Q)$ (by assumption), it must that $T$ has non-empty intersection with some cube $T' \in \cN'(Q)$. Arguing as in the proof of Lemma \ref{l:next-when-not-flat}, we must $T' \subseteq T$. By Lemma \ref{l:lots-of-flat-children}, we know that every cube in $\cN'(Q)$ has diameter comparable to $\pdiam(Q)$. The same must now be true of $T$, which completes the proof of Lemma \ref{l:next-when-flat} (2). 
    	
    	We finish by showing Lemma \ref{l:next-when-flat} (3). First, by Lemma \ref{l:big-projection-RF} (recall that $\Sigma$ is $C\alpha$-Reifenberg flat), \eqref{e:sigma-large-in-ball} and Lemmas \ref{l:z(S)} -- \ref{l:min-2}, we have 
    	\begin{align}
    		(1+C\alpha)^{-1}(1+C\delta)^{-1}\pdiam(Q_1)^{\idim+2} &\leq \pHmeasure^{\idim+2}(\Sigma \cap (1-C\delta)B(Q_1)) \\
    		&\hspace{-1em}\leq \pHmeasure^{\idim+2}(z(S)) + \pHmeasure^{\idim+2}\left( \Sigma \cap \bigcup_{R \in \min(S)} 10B_R\right) \\
    		&\hspace{-1em}\leq C\pHmeasure^{\idim+2}\left(E \cap Q_1 \setminus \bigcup_{T \in \cF'} T \right) \\
    		&+ (1+C\alpha)\sum_{T \in \cF'} \pdiam(T)^{\idim+2} \\\
    		&\hspace{1em}+ C\delta\pdiam(Q_1)^{\idim+2}.
    	\end{align}
    	Choosing both $\delta$ and $\alpha = \alpha(\delta)$ sufficiently small, and observing that by the definition of $\cN$ we have $\bigcup_{T \in \cN} T \subseteq \bigcup_{T \in \cF} T$, after rearranging the above inequality we get 
    	\begin{align}
    		(1+C\delta)^{-1} \pdiam(Q_1)^{\idim+2} &\leq C\pHmeasure^{\idim+2}\left(E \cap Q_1 \setminus \bigcup_{T \in \cN} T \right) \\
    		&\hspace{2em}+ (1+C\alpha)\sum_{T \in \cF'} \pdiam(T)^{\idim+2}. 
    	\end{align}
    	Given $T \in \cF'$, recall that the set $\cN'(T)$ is obtained by applying Lemma \ref{l:lots-of-flat-children} to the ball $B(T)$. By Lemma \ref{l:cF'-not-flat}, we have the correct assumptions on $B(T)$ to ensure that inequality \eqref{e:big-size} from Lemma \ref{l:lots-of-flat-children} holds for $\cN'(T)$. Also, by construction, every cube in $\cN'(T) \cap \mathscr{S}$ is disjoint from any cube in $\mathscr{N}$.
    	Thus, by choosing $\delta$ and $\alpha$ small enough and applying \eqref{e:big-size}, 
    	\begin{align}
    		\pdiam(Q_1)^{\idim+2} &\leq C\pHmeasure^{\idim+2}\left(E \cap Q_1 \setminus \bigcup_{T \in \cN} T \right) + \frac{2}{3}\sum_{T \in \cF'} \sum_{T' \in \cN'(T)} \pdiam(T')^{\idim+2} \\
    		&= C\pHmeasure^{\idim+2}\left(E \cap Q_1 \setminus \bigcup_{T \in \cN} T \right)  + \frac{2}{3} \sum_{T \in \cF'}\sum_{T' \in \cN(T) \cap \mathscr{S}} \pdiam(T')^{\idim+2} \\\
    		&\hspace{3em} + \frac{2}{3}\sum_{T \in \cF'} \sum_{T' \in \cN''(T)} \pdiam(T')^{\idim+2} \\\
    		&\leq C\pHmeasure^{\idim+2}\left(E \cap Q_1 \setminus \bigcup_{T \in \cN} T \right) + \frac{2}{3} \sum_{T' \in \cN} \pdiam(T')^{\idim+2}.
    	\end{align}
    \end{proof}

\noindent\textbf{A collection of top cubes.} With Lemmas \ref{l:next-when-not-flat} and \ref{l:next-when-flat} in hand, we can now focus more directly on the proof of Lemma \ref{l:control-Schul-cubes-main}. For this, we wish to define a collection of cubes $\Top$ whose volume we will control in Lemma \ref{l:Top-control}. We will need some notation. 

Suppose $Q \in \cC(Q_0) \setminus \mathscr{S}$ with $\baupbeta(AQ) \leq \ve$ is given. If $Q \in \cB$, let $\cF_Q \coloneqq \{Q\}$ and let $\cN_Q$ be the cubes obtained by applying Lemma \ref{l:next-when-not-flat} to $Q$. If $Q \in \cG$, let $\cF_Q$ and $\cN_Q$ denote the collections of cubes obtained by applying Lemma \ref{l:next-when-flat} to $Q$. Then, set 
\begin{align}
	\mathscr{O}_Q \coloneqq E \cap Q \setminus \bigcup_{R \in \cN_Q} T. 
\end{align}
Let 
\begin{align}
	\Top_0 = \{Q_0\}.
\end{align}
Then, supposing $\Top_k$ has been defined up to some integer $k \geq 0$, let 
\begin{align}
	\Top_{k+1} \coloneq \bigcup_{Q \in \Top_k} \cN_Q \quad \mbox { and } \quad \Top \coloneq \bigcup_{k \geq 0} \Top_k. 
\end{align}

Recall from Lemmas \ref{l:next-when-not-flat} and \ref{l:next-when-flat} that $\cN_Q \subseteq \cC(Q_0) \setminus \mathscr{S}$ (so that $\baupbeta(AB(R)) \leq \ve$ for all $R \in \cN_Q$). Thus, the above sequence is well-defined. Furthermore, it is immediate from the construction that 
\begin{align}\label{q:O_R}
	\mathscr{O}_Q \cap \mathscr{O}_{Q'}  = \emptyset \mbox{ for each distinct pair } Q,Q' \in \Top. 
\end{align}

\noindent\textbf{A collection of weight for the top cubes.} Suppose $Q \in \Top$ and $R \in \Top \cap \cC(Q)$ are given. We wish to define a weights $w_Q(R)$ and $w_Q(\mathscr{O}_R)$ with respect to $Q$. For this, let $\Top' = \Top \setminus \{Q\}$ and define 
\begin{align}
	\cT_1 &\coloneqq \{T \in \Top' \colon T \in \cN_{T'} \mbox{ for some } T' \in \Top \cap \cG \};\\
	\cT_2 &\coloneqq \{T \in \Top' \cap \cG \colon T \in \cN_{T'} \mbox{ for some } T' \in \Top \cap \cB\}; \\
	\cT_3 &\coloneqq \{T \in \Top' \cap \cB \colon T \in \cN_{T'} \mbox{ for some } T' \in \Top \cap \cB\}. 
\end{align}
Let $Q \eqqcolon T_1 \supseteq \cdots \supseteq T_N \coloneqq R$ be an enumeration of the cubes in $\Top$ between $Q$ and $R$ (inclusive of the endpoints). For $j \in \{1,2,3\}$, let 
\begin{align}
	N_{Q,j}(R) \coloneqq \#\{T_i \colon i \geq 2 \mbox{ and }  T_i \in \cT_j\}.
\end{align}
Then, define 
\begin{align}\label{e:w(R)}
	w_Q(R) \coloneqq \left(\frac{2}{3}\right)^{N_{Q,1}(R) + N_{Q,3}(R)} (1+\Cnext\delta)^{N_{Q,2}(R)}  \pdiam(R)^{\idim+2}
\end{align}
and 
\begin{align}\label{e:w(O)}
	w_Q(\mathscr{O}_R) \coloneqq \Cnextprime\left(\frac{2}{3}\right)^{N_{Q,1}(R) + N_{Q,3}(R)} (1+\Cnext\delta)^{N_{Q,2}(R)}  \pHmeasure^{\idim+2}(\mathscr{O}_R),
\end{align}
with $\Cnext$ and $\Cnextprime$ as in Lemma \ref{l:next-when-not-flat} and Lemma \ref{l:next-when-flat}, respectively. Note that $w_Q(Q) = \pdiam(Q)^{\idim+2}$. 

\begin{lemma}
Given $Q \in \Top_k$ and $R \in \Top_m$ such that $R \subseteq Q$, we have 
\begin{align}\label{e:properties-of-N}
	\hspace{1em} N_{Q,1}(R) + N_{Q,2}(R) + N_{Q,3}(R) = m-k  \quad \mbox{ and } \quad N_{Q,1}(R) \geq N_{Q,2}(R) -1. 
\end{align}
Thus, by choosing $\delta > 0$ sufficiently small,  
\begin{align}\label{e:decay-of-w}
	w_Q(R) \lesssim \left( \frac{3}{4} \right)^{\frac{m-k-1}{3}}\pdiam(R)^{\idim+2} \quad \mbox{ and } \quad w_Q(\mathscr{O}_R) \lesssim \left( \frac{3}{4} \right)^{\frac{m-k-1}{3}}\pHmeasure^{\idim+2}(\mathscr{O}_R).
\end{align}
\end{lemma}

\begin{proof}
We start with \eqref{e:properties-of-N}. The first assertion is trivial, so we will consider only the second assertion. For this, one notes that if $T_i \in \cT_2$ (so that $T_i \in \cG$), then it must be that $T_{i+1} \in \cT_1$ since $T_{i+1} \in \cN_{T_i}$. Thus, for each $i \in \{1,\cdots,N-1\}$ such that $T_i \in \cT_2$, there is a corresponding cube in $\cT_1$. This is enough to establish the second assertion (one must subtract 1 to account for the fact that, perhaps, $T_N = R \in \cT_2$). 

We will only give the details for the first inequality in \eqref{e:decay-of-w}, the second inequality follows similarly. First note, by \eqref{e:properties-of-N}, that $2N_{Q,1}(R) + N_{Q,3}(R) \geq m-k-1$. Thus, we either have that $N_{Q,1}(R) \geq (m-k-1)/3$ or $N_{Q,3}(R) \geq (m-k-1)/3$. In either case, the second part of \eqref{e:properties-of-N} implies (choosing $\delta = \delta(\Cnext)$ sufficiently small) that 
\begin{align}
	w_Q(R) &= \left(\frac{2}{3}\right)^{N_{Q,1}(R) + N_{Q,3}(R)} (1+\Cnext\delta)^{N_{Q,2}(R)}  \pdiam(R)^{\idim+2} \\
	&\leq \left(\frac{2}{3}\right)^{N_{Q,1}(R) + N_{Q,3}(R)} (1+\Cnext\delta)^{N_{Q,1}(R)+1}  \pdiam(R)^{\idim+2} \\
	&\lesssim \left( \frac{3}{4} \right)^{N_{Q,1}(R) + N_{Q,3}(R)}\pdiam(R)^{\idim+2} \leq \left( \frac{3}{4} \right)^{\frac{m-k-1}{3}}\pdiam(R)^{\idim+2}. 
\end{align}

\end{proof}

\begin{lemma}\label{l:control-with-weights}
	For each $Q \in \Top$, we have 
	\begin{align}\label{e:infty}
		\pdiam(Q)^{\idim+2} \leq \sum_{R \in \Top \cap \cC(Q)} w_Q(\mathscr{O}_R).
	\end{align}
\end{lemma}

\begin{proof}
	Fix some $Q \in \Top$ and let $k \geq 0$ such that $Q \in \Top_k$. It suffices to show that 
    \begin{align}\label{e:weight-control-next}
        w_Q(R) \leq w_Q(\mathscr{O}_R) + \sum_{\substack{ R' \in \Top_{m+1} \\ R' \subseteq R}} w_Q(R') 
    \end{align}
    for each $R \in \Top_m \cap \cC(Q)$ and $m \geq k$. Indeed, if \eqref{e:weight-control-next} holds, it follows by induction that 
	\begin{align}\label{e:m_0}
		\pdiam(Q)^{\idim+2} \leq \sum_{m=k}^{m_0} \sum_{\substack{R \in \Top_{m} \\ R \subseteq Q}} w_Q(\mathscr{O}_R) + \sum_{\substack{R \in \Top_{m_0+1} \\ R \subseteq Q}} w_Q(R)
	\end{align}
	for each $m_0 \geq k$. Using \eqref{e:decay-of-w}, the final term in \eqref{e:m_0} is at most a constant multiple of $(3/4)^{(m_0-k-1)/3}\pdiam(Q)^{\idim+2} \to 0$ as $m_0 \to \infty$. Thus, \eqref{e:infty} follows from \eqref{e:m_0} after taking $m_0 \to \infty$. 
	
	Let us now prove \eqref{e:weight-control-next}. Fix some $m \geq k$, $R \in \Top_m \cap \cC(Q)$ and let $C_R$ be the constant such that $w_Q(R) = C_R \pdiam(R)^{\idim+2}$. If $R \in \cG$, it follows from \eqref{e:w(R)} and \eqref{e:w(O)} that $w_Q(T) = (2C_R/3)\pdiam(T)^{\idim+2}$ for each $T \in \Top_{m+1}$ such that $T \subseteq R$ and $w_Q(\mathscr{O}_R) = C_R\Cnextprime \pHmeasure^{\idim+2}(\mathscr{O}_R)$. Thus, by Lemma \ref{l:next-when-flat},  
    \begin{align}
        w_Q(R) &= C_R\pdiam(R)^{\idim+2} \leq C_R\left[\Cnextprime \pHmeasure^{\idim+2}(\mathscr{O}_R) + \frac{2}{3}\sum_{T \in \cN_R} \pdiam(T)^{\idim+2} \right] \\
        & = w_Q(\mathscr{O}_R) + \sum_{\substack{T \in \Top_{m+1} \\ T \subseteq R}} w_Q(T),
    \end{align}
    which proves \eqref{e:weight-control-next} in this case. Suppose instead that $R \in \mathscr{B}$. By \eqref{e:w(R)}, have that $w_Q(T) = C_R(1+\Cnext\delta) \pdiam(T)^{\idim+2}$ for each $T \in \Top_{m+1} \cap \mathscr{G}$ such that $T \subseteq R$. Thus, by Lemma \ref{l:lots-of-flat-children}, 
    \begin{align}
        w_Q(R) &= C_R\pdiam(R)^{\idim+2} \leq C_R(1+\Cnext\delta)\sum_{T \in \cN_R \cap \cG} \pdiam(T)^{\idim+2} \\
        &\leq w_Q(\mathscr{O}_R) + \sum_{\substack{T \in \Top_{m+1} \\ T \subseteq R}} w_Q(T),
    \end{align}
    which finishes the proof of \eqref{e:weight-control-next}.
\end{proof}

\begin{lemma}\label{l:Top-control}
    We have 
        \begin{align}
            \sum_{Q \in \Top} \pdiam(Q)^{\idim+2} \lesssim \pHmeasure^{\idim+2}(Q_0) \lesssim \ell(Q_0)^{\idim+2}.
        \end{align}
\end{lemma}

\begin{proof}
By applying Lemma \ref{l:control-with-weights}, swapping the order of summation, summing over a geometric series, and using that (recall \eqref{q:O_R}) $\{\mathscr{O}_R\}_{R \in \Top}$ is a collection of disjoint subsets of $E \cap Q_0$, we have 
\begin{align}
	\sum_{Q \in \Top} \pdiam(Q)^{\idim+2} &\leq \sum_{Q \in \Top} \sum_{\substack{R \in \Top \\ R \subseteq Q}} w_Q(\mathscr{O}_R) \lesssim \sum_{R \in \Top} \sum_{\substack{Q \in \Top \\ R \subseteq Q}} w_Q(R) \\
	&=\sum_{m=0}^\infty \sum_{R \in \Top_m} \sum_{k=m}^\infty \sum_{\substack{Q \in \Top_k \\ R \subseteq Q}} w_Q(R) \\
	&\leq \sum_{m=0}^\infty \sum_{R \in \Top_m} \sum_{k=m}^\infty \sum_{\substack{Q \in \Top_k \\ R \subseteq Q}} \left( \frac{3}{4}\right)^\frac{m-k-1}{3}\pHmeasure^{\idim+2}(\mathscr{O}_R) \\
	&\lesssim \sum_{R \in \Top} \pHmeasure^{\idim+2}(\mathscr{O}_R) \leq \pHmeasure^{\idim+2}(E \cap Q_0)\lesssim \ell(Q_0)^{\idim+2}. 
\end{align}
\end{proof}

We now have all the ingredients needed for the proof of Lemma \ref{l:control-Schul-cubes-main}. 

\begin{proof}[Proof of Lemma \ref{l:control-Schul-cubes-main}]
	Let $\Cproof \coloneqq \min\{\Cnextlower,\Cnextlowerprime\}$ and, for each $T \in \cC$, let $\mathscr{R}(T)$ denote the set of $R \in \cC(T)$ such that $\pdiam(R) \geq \Cproof \pdiam(T)$. It suffices to show that
	\begin{align}\label{e:main-inclusion}
		\{R \in \cC(Q_0) \setminus \mathscr{S} \colon \bilatbeta(R) > \eta \} \subseteq \bigcup_{Q \in \Top} \bigcup_{T \in \cF_Q} \mathscr{R}(T). 
	\end{align}
	Indeed, if \eqref{e:main-inclusion} holds, then since $\cF_Q$ is a collection of disjoint sub-cubes of $Q$ and $\sum_{T' \in \mathscr{R}(T)} \pdiam(T')^{\idim+2}\lesssim_{\CAR, \Cproof} \pdiam(T)^{\idim+2}$ for all $T \in \cF_Q$, it follows from Lemma \ref{l:Top-control} that
	\begin{align}
		\sum_{\substack{R \in \cC(Q_0) \setminus \mathscr{S} \\ \bilatbeta(R) > \eta}} \pdiam(R)^{\idim+2} &\leq \sum_{Q \in \Top}  \sum_{T \in \cF_Q} \sum_{T' \in \mathscr{R}(T)} \pdiam(T')^{\idim+2} \\
		&\lesssim_{\CAR , \Cproof} \sum_{Q \in \Top} \pdiam(Q)^{\idim+2} \lesssim \ell(Q_0)^{\idim+2}. 
	\end{align}

	To show \eqref{e:main-inclusion}, fix some $R \in \cC(Q_0) \setminus \mathscr{S}$ such that $\bilatbeta(R) > \eta$. Let $Q$ be the minimal cube in $\Top$ such that $R \subseteq Q$ (such a cube exists since $R \subseteq Q_0$). By minimality and the definition of $\Top$, this means $R$ is not contained in any cube from $\cN_Q$. The inclusion now follows from either Lemma \ref{l:next-when-not-flat} (1) or Lemma \ref{l:next-when-flat} (2). 
\end{proof}

\subsection{Proof of Theorem \ref{BAUPBWGL}}
We now show how to prove Theorem \ref{BAUPBWGL} from Lemma \ref{l:control-Schul-cubes}. For this, we first need a variant of the \textit{one-third trick}. Although not explicitly stated as such, the following is proved in \cite{hytonen2012systems} (see Section 4).  

\begin{lemma}\label{l:one-third}
    For any doubling metric space $(X,d)$, there is a constant $N$ (depending only on the doubling constant) and a collection of finitely many subsets $\{X^n\}_{n=1}^N$ of $X$ such that the following holds. First, for each $n \in \{1,\dots,N\}$, we can write 
    \begin{align}
        X^n = \bigcup_{k \in \mathbb{Z}} X^n_k,
    \end{align}
    where each $X_k^n$ satisfies \eqref{e:sep-net}. Second, for $y \in X$ and $k \in \mathbb{Z}$, there exists $n \in \{1,\dots,N\}$ and $x \in X_k^n$ such that 
    \begin{align}
        d(x,y) \leq \rho^{k+1}. 
    \end{align}
\end{lemma}

\begin{proof}[Proof of Theorem \ref{BAUPBWGL}]
	Let $\cubes$ be a system of Christ-David cubes on $E$, and let $M > 1$ and $\alpha > 0$. Our goal is to show that 
	\begin{align}\label{e:final-claim}
		\sum_{\substack{R \in \cubes(R_0)\\ \bilatbeta(MB_R) > \alpha}} \pdiam(R)^{\idim+2} \lesssim_{M,\alpha} \pdiam(R_0)^{\idim+2} , \quad R_0 \in \cubes. 
	\end{align}
	
	Fix some $R_0 \in \cubes$ (our estimates will not depend on this choice). Since $E$ is Ahlfors regular, it is doubling. Thus, we can apply Lemma \ref{l:one-third} with $X = E$. Let $\{X^n\}_{n=1}^N$ be a sequence of subsets of $E$ satisfying the conclusion of Lemma \ref{l:one-third}, where $N$ depends only on $C_0$, and let $\{\cC(n)\}_{n=1}^N$ be the corresponding systems of Schul cubes on $E$ constructed with respect to these sets. 
    
    For every $R \in \cubes(R_0)$, there exists $n = n(R) \in N$ and $Q = Q(R) \in \cC(n)$ such that 
	\begin{align}\label{e:corresponding}
		MB_R \subseteq Q \quad \mbox{ and } \quad M\pdiam(R) \leq \pdiam(Q) \lesssim M\pdiam(R) . 
	\end{align}
    Indeed, let $k \in \mathbb{Z}$ such that $\rho^{k+2} < M\ell(R) \leq \rho^{k+1}$. By Lemma \ref{l:one-third}, there exists $n \in \{1,\dots,N\}$ and $x \in X_k^n$ such that $|x_R - x| \leq \rho^{k+1}$. Let $Q \in \cC_k(n)$ be the Schul cube such that $x_Q = x$. The second inequality in \eqref{e:corresponding} is immediate. Moreover, for $\rho = \rho(\CSchulprime)$ sufficiently small, we also have  
    \begin{align}
        MB_R\subseteq B_{M\ell(R) + \rho^{k+1}}(x) \subseteq B_{2\rho^{k+1}}(x_Q) \subseteq B_{\CSchulprime\rho^k}(x_Q) \subseteq Q. 
    \end{align}

	It follows from \eqref{e:corresponding} and a standard volume argument that, for each $Q \in \bigcup_{n=1}^N \cC(n)$, we have 
	\begin{align}\label{e:bounded-to-one}
		\#\{R \in \cubes \colon Q(R) = Q\} \lesssim_{C_0,M} 1. 
	\end{align}
    Moreover, it follows from Lemma \ref{l:monotonicity} that for any $R \in \cubes(R_0)$ we have 
    \begin{align}\label{e:cube-not-flat}
    	\bilatbeta(MB_R) > \alpha \implies \bilatbeta(B(Q(R))) \gtrsim \alpha. 
    \end{align} 
    
    Given $n \in \{1,\dots,N\}$, let $\{Q_{0,n}^k\}_{k \in K(n)}$ be the maximal collection of cubes $Q \in \cC(n)$ such that $Q = Q(R)$ for some $R \in \cubes(R_0)$. Since $MB_R \subseteq MB_{R_0}$ for every $R \in \cubes(R_0)$, it follows from \eqref{e:corresponding} that there exists $C = C(M)$ such that $Q_{0,n}^k \subseteq CB_{R_0}$ for all $n \in \{1,\dots,N\}$ and $k \in K(n)$. Combining this with \eqref{e:corresponding}, \eqref{e:bounded-to-one}, \eqref{e:cube-not-flat} and Lemma \ref{l:control-Schul-cubes}, the left-hand side of \eqref{e:final-claim} is at most  
    \begin{align}
        \sum_{\substack{R \in \cubes(R_0) \\ \bilatbeta(MB_R) > \alpha}} \pdiam(Q(R))^{\idim+2} &\lesssim_{M,C_0}\sum_{n=1}^N\sum_{k \in K(n)} \sum_{\substack{Q \in \cC(n) \\ Q\ \subseteq Q_{0,n}^k \\ \bilatbeta(B(Q)) \gtrsim \alpha}} \pdiam(Q)^{\idim+2} \\
        &\lesssim_{C_0,\alpha} \sum_{n=1}^N \sum_{k \in K(n)} \pdiam(Q_{0,n}^k)^{\idim+2} \\
        &\lesssim_{C_0} \pHmeasure^{\idim+2}(E \cap CB_{R_0}) \lesssim_{C_0,M} \ell(R_0)^{\idim+2}. 
    \end{align}
\end{proof}

\section{Construction of approximating Lipschitz graphs} \label{s:graphs}

In this section, we describe how to appropriately modify the construction in \cite{bortz2023corona} to prove Lemma \ref{p:graph-construction}. The following will be useful. 

\begin{lemma}\label{l:planes-make-graphs}
	Let $V,W \in \verticalplanes$ and assume that $\angle(V,W) \leq \alpha$. Then, 
	\begin{align}
		\| \pi_V^\perp(X,t) - \pi_V^\perp(Y,s) \| \leq \frac{L(\alpha)}{4} \| \pi_V(X,t) - \pi_V(Y,s) \| 
	\end{align}
	for all $(X,t) ,(Y,s) \in W$, where $L(\alpha) \coloneqq 4\alpha/\sqrt{1-\alpha^2}$. 
\end{lemma}

\begin{proof}
	We can assume without loss of generality that $(0,0) \in V$. Moreover, since $V$ and $W$ are $t$-independent, we can assume that $t = s = 0$. Now, fix some $X , Y \in \R^\adim$. Since $\angle(V,W) \leq \alpha$ and $(X-Y,0) \in W'$ (the linear space parallel to $W$), it follows that 
	\begin{align}
		\| \pi_V^\perp(X,0) - \pi^\perp(Y,0) \| = \| \pi_V^\perp(X-Y,0) \| \leq \alpha |X -Y|. 
	\end{align}
	Then, by orthogonality, it follows that 
	\begin{align}
		\| \pi_V^\perp(X,0) - \pi^\perp(Y,0) \|^2 &\leq \alpha^2 |X-Y|^2 \\
		&= \alpha^2 \left( \| \pi_V^\perp(X,0) - \pi^\perp(Y,0) \|^2  +  \| \pi_V(X,0) - \pi(Y,0) \|^2 \right).
	\end{align}
	The lemma follows immediately after rearranging the above inequality. 
\end{proof}

With this out of the way, we can begin to discuss the modification. Assume without loss of generality that 
\begin{align}
	P \coloneqq P_{Q(S)} = \R^\idim \times \{0\}^{\adim-\idim} \times \R. 
\end{align}
Let $\pi$ denote the projection on on $P$ and $\pi^\perp$ denote the projection onto $P^\perp$. We will often identify $P$ with $\R^\idim \times\R$ and write points in $\R^{\adim + 1}$ as 
\[(x,x',t) \in \R^\idim \times \R^{\adim-\idim} \times \R.\]
For $(X,t) = (x,x',t) \in \R^{\adim+1}$, we will denote
\begin{align}
	d(X,t) = d_S(X,t) \quad \mbox{ and } \quad D(x,t) = D_{S,P}(x,0,t).
\end{align}
To better follow the argument in \cite{bortz2023corona}, we introduce constants $K, \kappa , \delta ,\ve'$ defined by the following relationships 
\begin{align}\label{e:constants}
	K = M/8, \ \kappa = A/2 ,  \ \delta = L(\alpha) \mbox{ and } \ve' = M\ve. 
\end{align}
Note that 
\begin{align}\label{e:new-flat-constant}
	\pdist(X,t,P_Q) \leq \ve' \ell(Q) \mbox{ for all } (X,t) \in 8KB_Q, \ Q \in S,
\end{align}
as in \cite[Equation 3.9]{bortz2023corona} (with $\ve'$ playing the role of $\ve$ there). Our assumptions do not imply \cite[Equation 3.10 (A)]{bortz2023corona} with constant $\delta$, however, the important point is that the Lipschitz constants of the maps that we introduce below are at most $2\delta$ under our assumption (as we shall see). Eventually, we will have to assume $K/\kappa$ is sufficiently large and $\ve' = \ve(K,\delta)$ is sufficiently small. This corresponds to the choice that $M/A$ is sufficiently large and $\ve = \ve(\alpha,M)$ is sufficiently small, as needed for Lemma \ref{p:graph-construction}. 

Let 
\begin{align}
	F \coloneqq \{(X,t) \in E \colon d(X,t) = 0\}.
\end{align}
The following lemma is analogous to \cite[Lemma 4.2]{bortz2023corona}

\begin{lemma}\label{l:invertible}
	For every $N \geq 1$, supposing $\ve' = \ve'(N)$ is small enough, we have 
	\begin{align}
		\| \pi^\perp(X,t) - \pi^\perp(Y,s) \| \leq 2\delta \pd(\pi(X,t) , \pi(Y,s))
	\end{align}
	whenever $(X,t) , (Y,s) \in E \cap 16\kappa B_{Q(S)}$ are points which satisfy $\min\{d(X,t) , d(Y,s)\} \leq N \pd(X,t,Y,s)$. 
	In particular, the projection $\pi$ is one-to-one on $F$ and the map defined on $\pi(F)$ by
	\begin{align}
		\hat{\psi}(x,t) = \pi^\perp(X,t), 
	\end{align} 
	where $(X,t) \in F$ is such that $\pi(X,t) = (x,t)$, is well-defined and $\Lip(1,1/2,2\delta)$. 
\end{lemma}

\begin{proof}
	Let $N \geq 1$ and let $(X,t) , (Y,s) \in E \cap 16\kappa B_{Q(S)}$ such that $\min\{d(X,t) , d(Y,s)\} \leq N \pd(X,t,Y,s)$. Suppose, without loss of generality, that $d(X,t) \leq N \pd(X,t,Y,s)$. Now, let $Q \in S$ such that 
	\begin{align}
		\pdist(X,t,Q) + \ell(Q) \leq 2N \pd(X,t,Y,s). 
	\end{align}
	By taking an ancestor, if necessary, we can suppose that $\kappa^{-1} \pd(X,t,Y,s) \lesssim \ell(Q) \lesssim N \pd(X,t,Y,s)$. So long as $K/\kappa$ is sufficiently large, the first inequality implies that $(X,t) ,(Y,s) \in 8KB_{Q}$. Thus, recalling \eqref{e:new-flat-constant}, we have 
	\begin{align}
		\pdist(X,t,P_Q) + \pdist(Y,s,P_Q) \leq 2\ve' \ell(Q) \leq C \ve N \pd(X,t,Y,s). 
	\end{align}
	Let $(X',t) = \pi_Q(X,t)$ and $(Y',s) = \pi_Q(Y,s)$. Since $\angle(P_Q,P) \leq \alpha$, it follows from Lemma \ref{l:planes-make-graphs} that 
	\begin{align}
		\| \pi^\perp(X',t) - \pi^\perp(Y',s) \| \leq \frac{L(\alpha)}{2} \| \pi(X,t) -  \pi(Y,s) \|.
	\end{align}
	Then, choosing
	\begin{align}
		\ve' \leq \min\{ 1 , L(\alpha) , L(\alpha)^{-1}  \}/(10CN),
	\end{align}
	it follows from this that 
	\begin{align}
		\| \pi^\perp(X,t) - \pi^\perp(Y,s) \| &\leq C\ve N \pd(X,t,Y,s) + \| \pi^\perp(X',t) - \pi^\perp(Y',s) \| \\
		&\leq C\ve N \pd(X,t,Y,s) + \frac{L(\alpha)}{4} \| \pi(X',t) - \pi(Y',s) \| \\
		&\leq C\ve N \pd(X,t,Y,s) + \frac{L(\alpha)}{4} \big[\| \pi(X,t) - \pi(Y,s) \| \\
		&\hspace{3em} + \ve N \pd(X,t,Y,s) \big] \\
		&\leq \| \pi^\perp(X,t) - \pi^\perp(Y,s) \|/2 + \frac{L(\alpha)}{2} \| \pi(X,t) - \pi(Y,s) \|.
	\end{align}
	Rearranging the above inequality finishes the proof. 
\end{proof}

Following the arguments in \cite{bortz2023corona} verbatim, one then constructs a Whitney-type decomposition $\{I_i\}_{i \in \N}$ of $P \setminus \pi(F)$ such that 
\begin{align}
	\pdiam(I_i) &\sim D(X,t) \quad \mbox{for all } (X,t) \in 10I_i, \ i \in \N; \label{e:I-D} \\
	\pdiam(I_i) &\sim \pdiam(I_j) \quad \mbox{for all } i,j \in \N \mbox{ such that } 10I_i \cap 10I_j \neq\emptyset.  \label{e:I_i-I_j}
\end{align}
Moreover, for each 
\begin{align}\label{lambdaset}
	i \in \Lambda \coloneqq \left\{ i \in \N \colon I_i \cap B_{2\kappa \ell(Q(S))}\left(X_{Q(S)},t_{Q(S)}\right)\right\} \neq \emptyset,
\end{align}
there is a corresponding $Q(i) \in S$ such that 
\begin{align}
	\kappa^{-1} \pdiam(I_i) \lesssim \ell(Q(i)) &\leq 120\pdiam(I_i); \label{e:size-Q(i)} \\
	\pdist(\pi(Q(i))  , I_i) &\leq 120\pdiam(I_i). \label{e:dist-Q(i)}
\end{align}
The cubes $\{I_i\}$ also satisfy (see Lemma 4.6 of \cite{bortz2023corona})
\begin{align}\label{e:10-60}
    10 \diam I_i \leq D(Y,s) \leq 60 \diam I_i, \,\,\, \forall (Y,s) \in 10I_i.
\end{align}
Thus, if $10I_j \cap 10I_i \neq \emptyset$, $\diam I_i \approx \diam I_j$.

For each $i \in \Lambda$, let $A_i \colon P \to P^\perp$ be the affine function whose graph is $P_{Q(i)}$.  For $i \notin \Lambda$, we set $A_i\colon P \to P^\perp$ be the function whose graph is $P_{Q(S)}$. For $i \notin \Lambda$, set $A_i \equiv 0$. Since $\angle(P_Q,P) \leq \alpha$ for each $Q \in S$, it follows from Lemma \ref{l:planes-make-graphs} that 
\begin{align}
	|A_i(X,t) - A_i(Y,s) | \lesssim \delta \pd(X,t,Y,s) 
\end{align}
for all $(X,t) , (Y,s) \in P$ and $i \in \N$. Moreover, following the argument given in the proof of \cite[Lemma 4.18]{bortz2023corona}, we also have 
\begin{align}\label{e:small-change-A}
	|A_i(X,t) - A_j(X,t) | \lesssim \ve' \min\{\ell(I_i),\ell(I_j)\}
\end{align}
whenever $(X,t) \in 100I_i \cap 100I_j$ and $10I_i \cap 10I_j \neq\emptyset.$

Let $\{\varphi_i\}_{i \in \Lambda}$ be a partition of unity adapted to $\{3I_i\}_{i \in \Lambda}$ i.e., each $\varphi_i$ is supported on $3I_i$ and 
\begin{align}
    \sum_{i \in \Lambda} \varphi_i(X,t) = 1 \quad \mbox{ for all} \quad (X,t) \in \bigcup_{i \in \Lambda} 2I_i. 
\end{align}
The function $\psi \coloneqq \psi_S \colon P \to P^\perp$ from Lemma \ref{p:graph-construction} is then defined by 
\begin{align}\label{d:psidef}
	\psi(X,t) \coloneqq
	\begin{cases}
		\hat{\psi}(X,t) & \mbox{ if } (X,t) \in \pi(F); \\
		\sum_{i \in \Lambda} A_i(X,t) \varphi_i(X,t) & \mbox{ otherwise}.
	\end{cases}
\end{align}\\

This ends the construction of the function $\psi$ from Lemma \ref{p:graph-construction}. Given the above properties of $\{I_i\}$, $\{A_i\}$ and $\psi$, the proof of \eqref{e:E-close-to-gamma} proceeds exactly as \cite{bortz2023corona} and we direct the reader there for the details. It only remains to show \eqref{e:close-to-E} under the additional assumption \eqref{e:bilat-approx}. Recalling \eqref{e:constants}, condition \eqref{e:close-to-E} translates to 
\begin{align}\label{e:bilat-approx-a}
	\bilatbeta_E(8KB_Q,P_Q) \lesssim \ve'/K \mbox{ for all } Q \in S,
\end{align}
and it suffices to show the following. 

\begin{lemma}
	For all $(X,t) \in P \cap 2\kappa B_{Q(S)}$ we have 
	\begin{align}
		\pdist((X,t) + \psi(X,t) , E) \lesssim \ve' D(X,t) .
	\end{align}
\end{lemma}

\begin{proof}
	Fix some $(X,t) \in P \cap 2\kappa B_{Q(S)}$. If $(X,t) \in \pi(F)$ and $(Y,t) \in F$ is such that $(X,t) \in \pi(Y,t)$, it is clear that $(X,t) + \psi(X,t) = (Y,t) \in E$, and there is nothing to show. Thus, we will suppose instead that $(X,t) \not\in \pi(F)$. Let $i \in \N$ such that $(X,t) \in I_i$. Recalling that each $\varphi_j$ is supported on $3I_j$, it follows from \eqref{e:I_i-I_j} that there at most $C = C(n)$ indices $j \in \N$ such that $\varphi_j(X,t) \neq \emptyset$. By the left-most inequality in \eqref{e:size-Q(i)} and \eqref{e:dist-Q(i)}, we see that $(X,t) + A_i(X,t) \in P_{Q(i)} \cap C\kappa B_{Q(i)} \subseteq P_{Q(i)} \cap 8KB_{Q(i)}$, so long as $K / \kappa$ is sufficiently large. Combining the above observations with \eqref{e:I-D} and \eqref{e:small-change-A}, we first have
	\begin{align}\label{e:upper-bound-to-plane}
		\hspace{0.5em}\pdist((X,t) + \psi(X,t) , P_{Q(i)} \cap 8KB_{Q(i)}) &\leq  \left| A_i(X,t) - \sum_{j \in \N} A_j(X,t) \varphi_j(X,t) \right| \\
		&\leq \sum_{j \in \N} | A_i(X,t) - A_j(X,t) | \\
		&\lesssim \ve' \ell(I_i) \lesssim \ve' D(X,t). 
	\end{align}
	By \eqref{e:bilat-approx-a}, for each point in $(Y,s) \in P_{Q(i)} \cap 8KB_{Q(i)}$ there exists a corresponding point in $E$ whose distance from $(Y,s)$ is at most $\ve' \ell(Q(i)) \lesssim \ve' D(X,t)$. Inequality \eqref{e:close-to-E} now follows from this, \eqref{e:upper-bound-to-plane} and the triangle inequality. 
\end{proof}

\subsection{Modifications of the approximating graphs} In practice we will need the graphs in the corona decompositions to lie sufficiently close \textit{and} sufficiently far away from the set. This is made precise in the following lemma.

\begin{lemma}\label{l:modification}
	Let $A > 1$ and $\eta, \theta > 0$. Suppose $E \subseteq \R^{\adim+1}$ parabolic Ahlfors $(\idim+2)$-regular, $\cubes$ is a system of Christ-David cubes on $E$ and $S \subseteq \cubes$ is a stopping-time region for which there exists a plane $P \in \verticalplanes$ and an $\Lip(1,1/2,\eta)$ function $\psi \colon P \to P^\perp$ satisfying the conclusions of Theorem \ref{t:BWGL-bilat-Corona}. For $\theta = \theta(A)$ sufficiently small, there is a regular $\Lip(1,1/2)$ mapping $H \colon P \to P^\perp$ with $\|D^{1/2} H\|_* \lesssim 1$ such that, if 
	\begin{align}\label{e:psi-prime}
		\psi' \coloneqq \psi + \theta^\frac{15}{16} H
	\end{align}
	and $\Gamma'$ denotes the graph of $\psi'$ over $P$, then $\psi'$ is $\Lip(1,1/2,\eta + C\theta^{\frac{15}{16}})$,	
	\begin{align}\label{e:dist-estimates-1}
		\pdist(X,t,\Gamma') \sim \theta^\frac{15}{16}d_S(X,t) 
	\end{align}
	for all $(X,t) \in E \cap AB_{Q(S)}$ and 
	\begin{align}\label{e:dist-estimates-2}
		\pdist((X,t) + \psi'(X,t) ,E) \sim \theta^\frac{15}{16} D_{S,P}(X,t)
	\end{align}
	for all $(X,t) \in P \cap AB_{Q(S)}$. 

    Moreover, the map $H$ satisfies, for each $m \in \mathbb N$ and every $(X,t)\in P$ with $D(X,t)>0$,
    \begin{align}\label{e:Hests}
        D(X,t)^{2m-1}|\partial_t^mH(X,t)|
        +D(X,t)^{m-1}|\nabla^mH(X,t)| \leq c_{n,m}.
    \end{align}
\end{lemma}

\begin{proof}
	We can assume, without loss of generality, that $P = \R^\idim \times \{0\}^{\adim-\idim} \times \R$. We will denote by $\pi$ the projection onto $P$. Fix a unit vector $v\in P^\perp$. The function $H$ is defined by $H(X,t) \coloneqq \tilde{H}(X,t) v$, where $\tilde{H} \colon P \to \R$ is the function constructed by applying \cite[Lemma 3.24]{bortz2023carleson} to $D_S \coloneqq D_{S,P}$. It is then immediate that $H$ satisfies the properties listed above and, furthermore, 
	\begin{align}
		D_S(X,t) \sim |\tilde{H}(X,t)|
	\end{align}
	for all $(X,t) \in P$. Defining $\psi'$ as in \eqref{e:psi-prime}, it now only remains to check \eqref{e:dist-estimates-1} and \eqref{e:dist-estimates-2}. Before doing this, we note an important fact that, whenever $\varphi \colon P \to P^\perp$ is $\Lip(1,1/2,L)$, if $\Lambda$ denotes its graph and $(Y,s) = (y,y',s) \in \R^{\adim+1}$, then 
	\begin{align}\label{e:distance-to-graphs}
		\pd((Y,s) , (y,\varphi(y,s),s)) \lesssim_L \pdist(Y,s,\Lambda). 
	\end{align}
	
	We start with \eqref{e:dist-estimates-1}. Fix some $(X,t) = (x,x',t) \in E \cap AB_{Q(S)}$. Here, and throughout the remainder of the proof, $\Gamma$ will denote the graph of $\psi$ over $P$. First, by the triangle inequality, \eqref{e:corona-d}, \eqref{e:distance-to-graphs} and the definition of $\psi'$, we have 
	\begin{align}
		\pdist(X,t,\Gamma') &\leq   \pd\left( (X,t), (x,\psi'(x,t),t) \right) \\
		&\leq \pd((X,t) , (x,\psi(x,t),t)) + \pd(  (x,\psi(x,t),t) , (x,\psi'(x,t),t)  ) \\
		&\lesssim \pdist(X,t,\Gamma) + \theta^\frac{15}{16} |\tilde{H}(\pi(X,t))| \lesssim \theta d_S(X,t) +  \theta^\frac{15}{16} D_S(\pi(X,t)) \\
		&\lesssim \theta^{\frac{15}{16}} d_S(X,t). 
	\end{align}
	For the converse, we apply \eqref{e:corona-d}, \eqref{e:distance-to-graphs} twice, and choose $\theta$ small enough to get 
	\begin{align}
		\pdist(X,t,\Gamma') &\gtrsim \pd((X,t) , (x,\psi'(x,t),t)) \\
		&\geq \pd( (x,\psi'(x,t),t) , (x,\psi(x,t),t) ) - \pd((X,t) , (x,\psi(x,t),t)) \\
		&\gtrsim \theta^\frac{15}{16} D_S(\pi(X,t)) - \theta d_S(X,t)  \gtrsim \theta^\frac{15}{16} D_S(\pi(X,t)) .
	\end{align}
	
	Now we prove \eqref{e:dist-estimates-2}. Fix some $(X,t) = (x,0,t) \in P \cap AB_{Q(S)}$. First, by the triangle inequality, \eqref{e:corona-D} and the definition of $\psi'$, 
	\begin{align}
		\pdist( (x,\psi'(x,t),t) , E) &\leq \pd( (x,\psi'(x,t),t) , (x,\psi(x,t),t) ) \\
		&\hspace{2em}+ \pdist((x,\psi(x,t),t) , E) \\
		&\leq \theta^\frac{15}{16} |\tilde{H}(X,t)| + \theta D_S(X,t) \lesssim \theta^\frac{15}{16} D_S(X,t). 
	\end{align}
	For the converse inequality, we note first by \eqref{e:distance-to-graphs} and the definition of $\psi'$ that 
	\begin{align}\label{e:lower-bound-psi-prime}
		\pdist((x,\psi'(x,t),t),\Gamma) \geq c\theta^\frac{15}{16} D_S(X,t) 
	\end{align}
	for some absolute constant $c > 0$. Now, suppose towards a contradiction that there exists some $(Y,s) = (y,y',s) \in E$ such that 
	\[\pdist((Y,s) , (x,\psi'(x,t),t)) \leq \frac{c \theta^\frac{15}{16} }{4}D_S(X,t).\]
	Combining the above inequality with the fact that $D_S$ is 1-Lipschitz, we have $D_S(\pi(Y,s)) \leq 2D_S(\pi(Y,t))$ for $\theta$ sufficiently small. By \cite[Equation 4.25]{bortz2023carleson}, we have $d_S(Y,s) \lesssim_A D_S(\pi(Y,s))$. In particular, it follows that $d_S(Y,s) \lesssim_A D_S(X,t)$. Then, for $\theta = \theta(A)$ small enough, we have by this and \eqref{e:corona-d} that
	\begin{align}
		 \pdist((x,\psi'(x,t),t) , \Gamma) &\leq \pd((x,\psi'(x,t),t) , (Y,s)  )  +  \pdist(Y,s,\Gamma) \\
		 &\leq \frac{c \theta^\frac{15}{16} }{4}D_S(X,t) + C\theta d_S(Y,s) \\
		 &\leq  \frac{c \theta^\frac{15}{16} }{4}D_S(X,t) + C_A\theta D_S(X,t) \leq  \frac{c \theta^\frac{15}{16} }{2}D_S(X,t), 
	\end{align}
	which contradicts \eqref{e:lower-bound-psi-prime}.

    Finally, we note that \eqref{e:Hests} is a direct consequence of Lemma 3.24 of \cite{bortz2023carleson}.
\end{proof}

\section{Application to the Parabolic David-Semmes Problem}\label{s:Riesz}

\begin{definition}[Riesz kernels, Riesz transforms]
    We define the Riesz kernel with homogeneity $n+2$, $1 \leq n \leq d-1$, as
    \begin{align}\label{rieszkernel}
        K_n(X,t) = K(X,t) := \frac{X}{|t|^{\frac{n+1}{2}+1}} \exp\bigg(-\frac{|X|^2}{4t}\bigg)\mathds{1}_{t>0}.
    \end{align}
    Notice that this kernel is vector valued with $d$ spatial components.  We also define the $\epsilon$-truncations of the above kernels
    \begin{align}\label{truncrieszkernel}
    K_{n,\epsilon}(X,t) = K_{\epsilon}(X,t) =\mathds{1}_{\|(X,t)\|>\epsilon} \frac{X}{|t|^{\frac{n+1}{2}+1}} \exp\bigg(-\frac{|X|^2}{4t}\bigg)\mathds{1}_{t>0}.
    \end{align}
    Given a $(n+2)$-ADR measure $\mu$ on $\mathbb R^{d+1}$, we define the $(n+2)$-homogeneous, $\epsilon$-truncated Riesz transform for $f\in L^2(\mu)$ as 
    \begin{align}\label{truncatedtransforms}
        \mathcal R_{\mu,\epsilon}f(X,t) := \int K_{n,\epsilon}((X,t)-(Y,s)) f(Y,s) d\mu(Y,s).
    \end{align}
    We say that $\mathcal R_\mu : L^2(\mu) \rightarrow L^2(\mu)$ if, for each $f \in L^2(\mu)$,
    \begin{align}\label{boundedtransforms}
        \big\|\sup_{\epsilon>0} |\mathcal R_{\mu,\epsilon}f|\big\|_{L^2(\mu)} \lesssim \|f\|_{L^2(\mu)}.
    \end{align}
    Notice that a simple duality argument shows that, if $\mathcal R^-$ is the singular integral operator associated with the kernel $K(-X,-t)$, then
    \begin{align}\label{boundedduals}
      \big\|\sup_{\epsilon>0} |\mathcal R^-_{\mu,\epsilon}f|\big\|_{L^2(\mu)} \lesssim \|f\|_{L^2(\mu)}.  
    \end{align}
\end{definition}

In this section we prove the following theorem, which is just Theorem \ref{t:riesztheorem} restated for the reader's convenience. 
\begin{theorem}\label{t:riesztheorem-restate}
    Suppose $\mu$ is a parabolic $(n+2)$-ADR measure on $\mathbb R^{d+1}$, $1 \leq n \leq d-1$ which satisfies BAUP, and such that $\mathcal R:L^2(\mu) \rightarrow L^2(\mu)$.  Then $\mu$ is parabolic uniformly rectifiable.
\end{theorem}

For the remainder of this section we fix a measure $\mu$ satisfying the hypotheses of \ref{t:riesztheorem-restate}.  The overall strategy for the proof of Theorem \ref{t:riesztheorem-restate} is as follows.  By Theorem \ref{BAUPBWGL}, we know that $\mu$ satisfies BWGL.  Hence, by Theorem \ref{t:BWGL-Corona}, we also know that $\mu$ admits corona decompositions in terms of Lip-$(1,1/2,\eta)$ graphs, with $\eta$ as small as we like.  Moreover, by Theorem \ref{t:BWGL-bilat-Corona}, we can assume that the corona graphs approximate bilaterally.  Let $\mu_\Gamma$ be the Hausdorff measure restricted to such a graph.  We will prove that $\mathcal R_{\mu_{\Gamma}}:L^2(\mu_{\Gamma}) \rightarrow L^2(\mu_{\Gamma})$, with uniform bounds over all $\Gamma$ in the corona decomposition of $\mu$.  This step crucially uses the bilateral approximation $\mu$ by $\mu_\Gamma$.  By \cite[Theorem 1.1]{hoffman2026riesz}, this implies that the measures $\mu_{\Gamma}$ are uniformly rectifiable with uniform bounds over all $\Gamma$ in consideration.  By Theorem 1.1 and Proposition 2.23 of \cite{bortz2022coronizations}, we will immediately conclude that $\mu$ is itself a parabolic uniformly rectifiable measure.

We now proceed with the proof in full technicality.  Notice that, because $\mu$ is an Ahlfors $(n+2)$-regular measure, $E = \supp(\mu)$ is a $(n+2)$-ADR set in $\mathbb R^{d+1}$.  Thus, we can apply all the theorems previously proved in this paper to $E$.  We endow the measure $\mu$ (or equivalently, the set $E$) with a system of Christ-David cubes $\cD$.  {Choose $A,M>1$ as in the proof of Theorem \ref{t:BWGL-bilat-Corona}, with $\kappa=A/2\geq2$ and $M\geq13\kappa$. Choose $\theta>0$ sufficiently small for Lemma \ref{l:modification}, and then choose $\eta,\ve'>0$ sufficiently small depending on $\theta$, with $\ve=\ve'/M$.}  Because $\mu$ satisfies BAUP, Theorem \ref{t:BWGL-bilat-Corona} guarantees the existence of a coronization $(\cB,\cG,\cF)$ of $\cD$ such that, for each $S \in \cF$ with associated top cube $Q(S)$, there is a vertical plane $P = P_S\in\verticalplanes$ and a Lip-$(1,1/2,\eta)$ function $\psi = \psi_S:P \rightarrow P^\perp$ such that \eqref{e:corona-d} and \eqref{e:corona-D} both hold.  We fix such an $S$, $P$, and $\psi$ for the remainder of this section.  We use Lemma \ref{l:modification} to produce a function $H:P \rightarrow P^\perp$ with $\|D^{1/2} H\|_* \lesssim 1$ such that, if
$$\psi' := \psi + \theta^{15/16}H,$$
then \eqref{e:dist-estimates-1} and \eqref{e:dist-estimates-2} hold.  We define the graphs $\Gamma = \{(X,t) + \psi(X,t) : (X,t) \in P\}$, and $\Gamma' = \{(X,t) + \psi'(X,t) : (X,t) \in P\}.$

We recall some aspects of the construction of $\psi$ from Section \ref{s:graphs}.  The set $E$ has a ``contact region"
$$F \coloneqq \{(X,t) \in E: d(X,t)=0\}.$$
For $(Y,s)\in F$, we have $D(\pi(Y,s))\leq d(Y,s)=0$, so that $H(\pi(Y,s))=0$. This, combined with the fact that $\psi=\hat\psi$ on $\pi(F)$ (recall \eqref{d:psidef} and Lemma \ref{l:invertible}), it follows that \[F\subseteq\Gamma'.\] 
We also have a partition of $P \setminus \pi_P(F)$ is partitioned into ``Whitney cubes" $\{I_i\}$ satisfying \eqref{e:I-D} and \eqref{e:I_i-I_j}, and we introduced the index set
$$\Lambda := \{i \in \mathbb N: I_i \cap B_{2\kappa \ell(Q(S))}(X_{Q(S)},t_{Q(S)}) \neq \emptyset\}.$$
By \eqref{d:psidef},
\begin{align}
    \psi(x,t)=0 \quad\mbox{for all }(x,t)\in P\setminus\Big(\pi(F)\cup\bigcup_{i\in\Lambda}3I_i\Big).
\end{align}
For the arguments of this section, it will also be important to consider the cubes with indices in
$$\Lambda_0 := \{i \in \mathbb N: I_i \cap B_{4\kappa \ell(Q(S))}(X_{Q(S)},t_{Q(S)}) \neq \emptyset\}.$$
For any such index $i$, we define 
\[\Gamma'_{i} = \pi^{-1}(3I_i) \cap \Gamma'.\]
We further define
$$\Gamma'_s := F\cup\bigcup_{i\in\Lambda}\Gamma'_i,$$
$$\Gamma'_{f} := F\cup\bigcup_{i\in\Lambda_0}\Gamma'_i.$$
We let $\sigma$ denote the parabolic Hausdorff measure of homogeneous dimension $(n+2)$ restricted to $\Gamma'$.  We will also use $\mgs$ and $\mgf$ to denote the restriction of $\mg$ to $\Gamma'_s$ and $\Gamma'_f$, respectively.

On $F$, these measures agree with $\pHmeasure^{n+2}|_F$. The ADR lower bound for $\mu$ and the covering definition of $\pHmeasure^{n+2}$ give \begin{align}\label{e:measures-on-F}
\sigma(B)\leq C\mu(B) \mbox{ for every Borel set } B\subseteq F,
\end{align}
where $C$ depends only on the dimensions and the ADR constants of $\mu$.

At points of $F$, we use the ordinary truncated transform $\mathcal R_{\nu,\epsilon}f$, given by the same integral for $\epsilon>0$.

For each $i \in \Lambda$, there is a corresponding cube $Q(i) \in S$ such that \eqref{e:size-Q(i)} and \eqref{e:dist-Q(i)} both hold.  If $i \in \Lambda_0\setminus \Lambda$, we set $Q(i) = Q(S)$. We actually can also set $Q(i)=Q(S)$ if $3I_i \cap \big(B_{2\kappa \diam(Q(S))}(X_{Q(S)},t_{Q(S)})\big)^c \neq \emptyset$, and still retain \eqref{e:size-Q(i)} and \eqref{e:dist-Q(i)}, as the next lemma shows.
\begin{lemma}
    Suppose $i \in \Lambda_0$ is such that $3I_i \cap \big(B_{2\kappa \diam(Q(S))}(X_{Q(S)},t_{Q(S)})\big)^c \neq \emptyset$.  Then    $$\kappa^{-1} \diam_p(I_i) \approx \ell(Q(S))$$
    and 
    $$\dist(I_i,\pi(Q(S))) \leq 60\diam_p (I_i).$$
\end{lemma}
\begin{proof}
    We first notice that
    $$D(x_i,t_i) \leq \ell(Q(S)) + \dist(x_i,t_i,\pi(Q(S)) \lesssim \kappa \diam_p (Q(S))$$
    Hence, \eqref{e:10-60} tells us that $\ell( Q(S))\gtrsim  \kappa^{-1}\diam_p(I_i)$.  On the other hand, there is some point $(x,t) \in 3I_i \cap (B_{2\kappa \diam Q(S)}(x_{Q(S)},t_{Q(S)}))^c$, and this point clearly satisfies
    $D(x,t) \geq \kappa\diam_p (Q(S))$.  So, another application of \eqref{e:10-60} yields $\diam_p (I_i) \geq (\kappa/60) \diam_p (Q(S))$.

    Let $(x,t) \in I_i$ minimize $\inf_{(y,s) \in I_i}\dist(y,s,\pi(Q(S)))$.  Then, using \eqref{e:10-60} we find that $60 \diam_p(I_i) \geq D(x,t)$.  Then, since $D(x,t) \geq \dist(I_i,Q(S))$, this estimate implies $\dist(I_i,Q(S)) \leq 60 \diam_p(I_i).$
\end{proof}

We adopt the notation 
\[\tilde \Lambda = \{i\in \Lambda_0 : 3I_i \cap B_{2\kappa \diam Q(S)}(X_{Q(S)},t_{Q(S)})^c \neq \emptyset\}.\]  
We assume $Q(S) = Q(i)$ for each $i \in \tilde \Lambda$ for the remainder of this section.  Let $\Gamma_H$ be the graph of $\theta^{15/16}H$.  Recalling the definition of $\psi$ from Section \ref{s:graphs}, in particular that $A_i \equiv 0$ for all $i \not\in \Lambda$, we have $\supp(\psi) \subseteq B_{2\kappa \diam Q(S)}(X_{Q(S)},t_{Q(S)})$. As a consequence, it follows from the above definitions that
\begin{align}
\Gamma' \setminus \pi^{-1}(B_{2\kappa \diam Q(S)}(x_{Q(S)},t_{Q(S)})^c) = \Gamma_H \setminus \pi^{-1}(B_{2\kappa \diam Q(S)}(x_{Q(S)},t_{Q(S)})^c).
\end{align}

\begin{lemma}\label{l:separation1}
    There is a constant $c_*$ such that
    \begin{align}
        \min\left\{\dist_p(\Gamma_i',E) , \dist(Q(i),E \setminus 2B_{Q(i)}) \right\} \geq 4c_*\diam(I_i).
    \end{align}
    for all $i \in \Lambda_0$. 
\end{lemma}

\begin{proof}
    Fix some $i \in \Lambda_0$. We start by estimating $\dist_p(\Gamma_i',E)$ from below. Fix any point in $(X,t) \in \Gamma_i'$. By definition, $\pi_P(X,t) \in 3I_i$. Thus, by \eqref{e:I-D},  $\diam_p(I_i) \sim D_S(\pi_P(X,t))$. It then follows directly from \eqref{e:dist-estimates-2} that $\dist_p(X,t,E) \gtrsim \diam_p(I_i)$. Since $(X,t)$ was arbitrary, this implies 
    \begin{align}
        \dist_p(\Gamma_i',E) \gtrsim \diam_p(I_i).
    \end{align}
    To estimate $\dist(Q(i),E \setminus 2B_{Q(i)})$ from below, fix some $(Y,s) \in Q(i)$. Since $Q(i) \subseteq B_{Q(i)}$, it follows from the triangle inequality and \eqref{e:dist-Q(i)} that  $\dist(Y,s,E \setminus 2B_{Q(i)}) \geq \ell(Q(i)) \sim \diam_p(I_i)$. Since $(Y,s)$ was arbitrary, this implies 
    \begin{align}
        \dist(Q(i),E \setminus 2B_{Q(i)}) \gtrsim \diam_p(I_i). 
    \end{align}
    The previous two display equations imply the desired estimate. 
\end{proof}

The following is immediate from the definition of $c_*$ above and the definition of the truncated kernel. 

\begin{corollary}\label{c:separation-kernels}
    Let $i \in \Lambda_0$ and $0 < \epsilon < c_* \diam_p(I_i)$. Then, for every $(X,t)\in\Gamma_i'$, $(X',t')\in Q(i)$, and $(Y,s)\in E\setminus 2B_{Q(i)}$, we have 
\[
K_\epsilon\bigl((X,t)-(Y,s)\bigr)
=
K\bigl((X,t)-(Y,s)\bigr)
\]
and
\[
K_\epsilon\bigl((X',t')-(Y,s)\bigr)
=
K\bigl((X',t')-(Y,s)\bigr).
\] 
\end{corollary}

When reading the remainder of this section, the reader should recall that the collection $\{3I_i\}_i$ has the bounded overlap property, and thus $\{\Gamma'_i\}_i$ does as well.

\begin{lemma}\label{sepbound}
$\|\sup_{\epsilon>0}|{\mathcal R}_{\mu,\epsilon} f|\|_{L^2(\mgs)} \lesssim \|f\|_{L^2(\mu)}$ and $\|\sup_{\epsilon>0}|{\mathcal R}_{\mu,\epsilon} f|\|_{L^2(\mgf)} \lesssim \|f\|_{L^2(\mu)}$.
\end{lemma}

\begin{proof}
We will only show $\|\sup_{\epsilon>0}|{\mathcal R}_{\mu,\epsilon} f|\|_{L^2(\mgs)} \lesssim \|f\|_{L^2(\mu)}$, and note that the proof of the other bound follows in the exact same manner. By a standard Cotlar inequality (see \cite[Lemmas 6.8 and 6.10]{fassler2021singular}), it is enough to estimate the following quantity uniformly over $\epsilon>0$
\begin{align*}
\sum_{i \in \Lambda} \int_{\Gamma'_{i}}\left|\int K_\epsilon((X,t)-(Y,s))f(Y,s) d\mu(Y,s) \right|^2 d\mg(X,t)\\
+ \int_{F} \left| \int K_\epsilon((X,t)-(Y,s))f(Y,s) d\mu(Y,s) \right|^2 d\mg(X,t).
\end{align*}
On $F$, the inner integral is $\mathcal R_{\mu,\epsilon}f$. Recalling \eqref{e:measures-on-F}, we have $\int_F\sup_{\epsilon>0}|\mathcal R_{\mu,\epsilon}f|^2\,d\sigma\leq C\int_E\sup_{\epsilon>0}|\mathcal R_{\mu,\epsilon}f|^2\,d\mu$. Applying \eqref{boundedtransforms} bounds this by a constant times $\|f\|_{L^2(\mu)}^2$, with the constant depending only on the dimensions, the ADR constants of $\mu$, and the bound in \eqref{boundedtransforms}. For the other terms, we note the following.  Fixing $i$ momentarily, consider, with $\epsilon>0$ fixed, the expression
$$\int_{\Gamma_{i}'}\left|\int K_\epsilon((X,t)-(Y,s))f(Y,s) d\mu(Y,s) \right|^2 d\mg(X,t).$$
This is clearly bounded above by 
\begin{align}
&\fint_{Q(i)} \int_{\Gamma'_{i}} \left| \int [K_\epsilon((X,t)-(Y,s))-K_\epsilon((X',t')-(Y,s))]f(Y,s)d\mu(Y,s) \right|^2 d\mg(X,t)d\mu(X',t') \\
&+\frac{2\mg(\Gamma_i')}{\mu(Q(i))}\int_{Q(i)} \bigg| \int_{E} K_{\epsilon}((X',t')-(Y,s))f(Y,s) d\mu(Y,s)\bigg|^2 d\mu(X',t')\\
&\lesssim\fint_{Q(i)} \int_{\Gamma'_{i}} \left|\int_{E\cap2B_{Q(i)}} [K_\epsilon((X,t)-(Y,s))-K_\epsilon((X',t')-(Y,s))]f(Y,s)d\mu(Y,s) \right|^2 d\mg(X,t)d\mu(X',t') \\
+&\fint_{Q(i)} \int_{\Gamma'_{i}} \left| \int_{E\setminus2B_{Q(i)}} [K_\epsilon((X,t)-(Y,s))-K_\epsilon((X',t')-(Y,s))]f(Y,s)d\mu(Y,s) \right|^2 d\mg(X,t)d\mu(X',t') \\
&+\frac{2\mg(\Gamma_i')}{\mu(Q(i))}\int_{Q(i)} \bigg| \int_{E} K_{\epsilon}((X',t')-(Y,s))f(Y,s) d\mu(Y,s)\bigg|^2 d\mu(X',t')\\
=& I + II + III.
\end{align}

By the Lipschitz bound for $\psi'$ and ADR from $\mu$, we have $\mg(\Gamma_i')\leq C\pdiam(I_i)^{n+2}$ and $\mu(Q(i))\geq C^{-1}\ell(Q(i))^{n+2}$. In particular, recalling that $\ell(Q(i)) \sim \diam_p(I_i)$, we have 
\begin{align}\label{e:sigma-vs-mu}
    \mg(\Gamma_i')/\mu(Q(i))\lesssim 1. 
\end{align}
By \eqref{boundedtransforms}, we have $\int_E|\mathcal R_{\mu,\epsilon}f|^2\,d\mu\leq C\|f\|_{L^2(\mu)}^2$, uniformly in $\epsilon>0$ and $f\in L^2(\mu)$. In particular, 
\begin{align}
    III \lesssim \|f\|_{L^2(\mu)}. 
\end{align}

We will now obtain estimates on terms $I$ and $II$ which are independent of $\epsilon$.  We first use the triangle inequality and \eqref{e:sigma-vs-mu} to estimate
\begin{align}
 I &\lesssim \int_{\Gamma'_{i}} \left|\int_{E\cap2B_{Q(i)}} K_\epsilon((X,t)-(Y,s))f(Y,s) d\mu(Y,s)\right|^2 d\mg(X,t) \\
&+ \int_{Q(i)} \left|\int_{E\cap2B_{Q(i)}}K_\epsilon((X',t')-(Y,s))f(Y,s) d\mu(Y,s) \right|^2 d\mu(X',t').
\end{align}
The second term is clearly bounded on the order of $\|f\|_{L^2(\mu)}^2$ by our assumptions about $\mu$.  To deal with the first term above, we estimate it further as follows:
\begin{align}
\label{term1}&\int_{\Gamma'_{i}} \left|\int_{E\cap2B_{Q(i)}} \mathds{1}_{\|(X,t)-(Y,s)\| \leq \diam (I_i)} K_\epsilon((X,t)-(Y,s)) f(Y,s)d\mu (Y,s) \right|^2 d\mg (X,t)\\
\label{term2}+& \int_{\Gamma'_{i}} \left|\int_{E\cap2B_{Q(i)}} \mathds{1}_{\|(X,t)-(Y,s)\| > \diam (I_i)} K_\epsilon((X,t)-(Y,s)) f(Y,s)d\mu(Y,s) \right|^2 d\mg(X,t).
\end{align}
For the second term above, we use that $\|(X,t)-(Y,s)\| > \diam (I_i)$ and Calder\'on-Zygmund size estimates to bound the inner integral by a maximal function applied to $f$, and hence obtain a bound on the order of $\|f\|_{L^2(\mu)}^2$.
We can bound the first integral in the same manner after an application of Lemma \ref{l:separation1}.

Now we come to term $II$.  We first consider the case where $\epsilon < c_*\ell(I_i)$, with $c_*$ as in Lemma \ref{l:separation1}. It follows from Corollary \ref{c:separation-kernels} that $K_{\epsilon}((X,t)-(Y,s))=K((X,t)-(Y,s))$ and $K_{\epsilon}((X',t')-(Y,s))=K((X',t')-(Y,s))$ for all $(X,t),(X',t'),(Y,s)$ appearing in the expression which defines $II$. So, in this case, we use a Calder\'on-Zygmund smoothness estimate to say
\begin{align*}
&\bigg| \int_{E\setminus2B_{Q(i)}} [K_\epsilon((X,t)-(Y,s))-K_\epsilon((X',t')-(Y,s))]f(Y,s)d\mu(Y,s) \bigg|\\
\lesssim& \int_{E\setminus2B_{Q(i)}} \frac{\ell(I_i)}{\|(X',t')-(Y,s)\|^{n+3}}|f(Y,s)| d\mu(Y,s)\\ 
\lesssim& \mathcal M_{\mu}(f)(X',t'):=\sup_{r>0}\frac{1}{\mu(B_r(X',t'))}\int_{B_r(X',t')}|f(Y,s)|\,d\mu(Y,s).
\end{align*}
For the last inequality, split the integral into the annuli $2^j\ell(Q(i))\leq\|(X',t')-(Y,s)\|<2^{j+1}\ell(Q(i))$, $j\geq0$. The ADR upper bound gives a contribution at most 
\[C2^{-j}\ell(I_i)\mathcal M_{\mu}(f)(X',t')/\ell(Q(i))\]
from the $j$th annulus, where $C$ depends only on the dimensions and the ADR upper bound. Summing over $j$ and using \eqref{e:size-Q(i)} gives the last inequality. Hence, term $II$ is bounded by $C\|f\|_{L^2(\mu)}^2$ in this case.  If $\epsilon>c_*\ell(I_i)$, we set $c'=c_*/2$ and obtain the estimate
\begin{align*}
    &\left| \int_{E\setminus2B_{Q(i)}} [K_\epsilon((X,t)-(Y,s))-K_\epsilon((X',t')-(Y,s))]f(Y,s)d\mu(Y,s)\right|\\
    \leq& \left| \int_{E\setminus2B_{Q(i)}} [K_{c'\ell(I_i)}((X,t)-(Y,s))-K_{c'\ell(I_i)}((X',t')-(Y,s))]f(Y,s)d\mu(Y,s)\right|\\
    +&
    \bigg| \int_{E\setminus2B_{Q(i)}} [K_\epsilon((X,t)-(Y,s))-K_{c'\ell(I_i)}((X,t)-(Y,s))\\
    &\hspace{2em} + K_{c'\ell(I_i)}((X',t')-(Y,s))-K_\epsilon((X',t') -(Y,s))]f(Y,s)d\mu(Y,s)\bigg| \\
    =& II_{inner,1} + II_{inner,2} . 
\end{align*}
The term corresponding to $II_{inner,1}$ is bounded by a constant multiple of $\mathcal{M}_\mu(f)(X',t')$ using the reasoning from the case where $\epsilon<c_*\ell(I_i)$, the $II_{inner,2}$ is bounded pointwise by $C\mathcal M_{\mu}(f)(X',t')$, uniformly in $\epsilon$, by the kernel size and smoothness estimates and Lemma \ref{l:separation1}. Averaging over $Q(i)$ and integrating over $\Gamma_i'$, \eqref{e:sigma-vs-mu} and the $L^2(\mu)$ boundedness of $\mathcal M_\mu$ give a bound of $C\|f\|_{L^2(\mu)}^2$ for this contribution to $II$. 
\end{proof}

\begin{lemma}\label{l:duality}
$\|\sup_{\epsilon>0}| {\mathcal R}_{\mgs,\epsilon}f|\|_{L^2(\mu)} \lesssim \|f\|_{L^2(\mgs)}$ and $\|\sup_{\epsilon>0}| {\mathcal R}_{\mgf,\epsilon}f|\|_{L^2(\mu)} \lesssim \|f\|_{L^2(\mgf)}$.
\end{lemma}
\begin{proof}
    This follows from Lemma \ref{sepbound} and duality.
\end{proof}

\begin{lemma}\label{l:interiorbound}
$\mathcal{R}_{\mgs}:L^2(\mgs) \rightarrow L^2(\mgs)$ and
$\mathcal R_{\mgf}:L^2(\mgf) \rightarrow L^2(\mgf)$.
\end{lemma}
\begin{proof}
Again, we prove only the bound for $\mgs$, as the bound for $\mgf$ is shown in the exact same manner.  We will fix a truncation parameter $\epsilon$ and achieve bounds which are independent of $\epsilon$.
Fix $f \in L^2(\mgs)$  
and define the following truncation of $\mathcal R$:  
\begin{align*}
\mathcal R_{\mgs,\epsilon, *} f(X,t) = \int_{\Gamma'_{s}} K_\epsilon((X,t)-(Y,s))
\left(\sum_{i \in \Lambda}\mathds{1}_{I_i}(\pi(X,t))\mathds{1}_{\|\pi((X,t)-(Y,s))\| \geq c\ell(I_i)} \right) f(Y,s) d\mgs(Y,s),
\end{align*}
where $0<c\ll 1$ is to be chosen sufficiently small.
At times it will also be convenient to write this truncation with respect to the coordinates on $(x,t)$ on $\R^{n+1}$ (where $X = (x,x^\perp) \in \mathbb R^n \times \mathbb R^{d-n}$, and $K(X,t) = K(x,x^\perp,t)$). For $(y,s) \in \R^{n+1}$, define
\begin{align}
    g(y,s) \coloneqq f(y,\psi'(y,s),s). 
\end{align}
Then, by a change of variable, 
\begin{align*}
\mathcal R_{\mgs,\epsilon, *}  f(X,t) = \int_{\rn \times\R} K_\epsilon(x-y,\psi'(x,t)-\psi'(y,s),t-s)
\left(\sum_{i \in \Lambda}\mathds{1}_{I_i}(x,t)\mathds{1}_{\|(x,t)-(y,s)\| \geq c\ell(I_i)} \right) g(y,s)\\ \cdot\sqrt{1 + |\nabla \psi'(y,s)|^2} dyds.
\end{align*}

We want to estimate
$$\int_{\Gamma_{s}'} |\mathcal R_{\mgs,\epsilon} f(X,t)|^2 d\mgs(X,t).$$
First, we deal with portion of the integral over the contact set $F$. By Lemma \ref{l:duality}, we have 
$\mathcal R_{\mgs}:L^2(\mgs) \rightarrow L^2(\mu)$.   Hence,
$$\int_{F}|\mathcal R_{\mgs,\epsilon}f(X,t)|^2 d\mgs(X,t) \lesssim \int |\mathcal R_{\mgs,\epsilon}f(X,t)|^2 d\mu(X,t) \lesssim \|f\|^2_{L^2(\mgs)}.$$
Now, for each $i \in \Lambda$, we will bound
$$\int_{\Gamma'_{i}}|\mathcal R_{\mgs,\epsilon}f(X,t)|^2 d\mgs(X,t).$$
To do so, we split the integrand
$$|\mathcal R_{\mgs,\epsilon}f(X,t)| \leq |\mathcal R_{\mgs,\epsilon}f(X,t) - \mathcal R_{\mgs,\epsilon,*}f(X,t)| + |\mathcal R_{\mgs,\epsilon,*}f(X,t)| = I + II.$$
We estimate term $II$ first.  For fixed $(X',t') \in Q(i)$ and $(X,t) \in \Gamma'_{i}$, 
\begin{align*}
&|\mathcal R_{\mgs,\epsilon,*}f(X,t) - \mathcal R_{\mgs,\epsilon} f(X',t')|\\
=& \left|\int K((X,t)-(Y,s))\mathds{1}_{\|(x,t)-(y,s)\| \geq c\ell(I_i)} - K_\epsilon((X',t')-(Y,s)))f(Y,s) d\mgs(Y,s) \right| \\
\lesssim &\int_{\substack{\|(x,t)-(y,s)\| \geq c\ell(I_i)\\ \|(X',t')-(Y,s)\|>\epsilon }}
|(K((X',t')-(Y,s)) - K((X,t)-(Y,s)))f(Y,s)| d\sigma(Y,s)\\
+&\frac{1}{\ell(I_i)^{n+1}}\int_{\substack{ \|(X',t')-(Y,s)\|\leq \epsilon }}
|f(Y,s)| d\mgs(Y,s) \\
+&\fint_{\substack{\|(x,t)-(y,s)\|\leq c\ell(I_i)}}
|f(Y,s)| d\mgs(Y,s)
\end{align*} The second and third terms of the right-hand-side of the above inequality are trivially bounded by a maximal function applied to $f$.  The first term above is also bounded on the order of a maximal function applied to $f$ by Calder\'on-Zygmund smoothness estimates.  Thus, we have
\begin{multline*}
    \int_{\Gamma_i'} |\mathcal R_{\sigma_s,\epsilon,*}f(X,t)|^2 d\mgs(X,t) \lesssim \int_{\Gamma_i'}\mathcal M_{\mgs}f(X,t) d\mgs(X,t) + \int_{Q(i)} \mathcal R_{\mgs,\epsilon}f(X',t') d\mu(X',t').
\end{multline*}
The term involving the maximal function is trivially bounded on the order of $\|f\|_{L^2(\mgs)}$, and the other term is bounded by the same quantity by Lemma \ref{l:duality}.

Now we deal with term $I$. {Since $\Gamma'$ is a parabolic Lipschitz graph with constant $L$, replacing the truncation $\|(X,t)-(Y,s)\|>\epsilon$ by $\|(x,t)-(y,s)\|>\epsilon$ changes the integral only where
$$\epsilon/(1+L)<\|(x,t)-(y,s)\|\leq\epsilon.$$
By the size estimate for $K$ and the upper growth bound for $\mgs$, this contribution is bounded by $C\mathcal M_{\mgs}f(X,t)$, where $C$ depends only on the dimensions, $L$, and the kernel size constant. Hence, for almost every $(X,t)\in\Gamma'_{s}$ with $\pi(X,t)\in I_i$,
$$I\leq\Big|\int_{\substack{(Y,s)\in\Gamma'_{s}\\ \epsilon<\|(x,t)-(y,s)\|<c\ell(I_i)}}K((X,t)-(Y,s))f(Y,s)\,d\mgs(Y,s)\Big|+C\mathcal M_{\mgs}f(X,t).$$
The integral on the right is zero when $\epsilon\geq c\ell(I_i)$.} So, consider the case when $\epsilon < c\ell(I_i)$. We will show that 
\begin{align}\label{regatsmall}
\int_{I_i} I \lesssim \|g\|^2_{L^2(3I_i)}. 
\end{align}
Since $\|f\|_{L^2(\Gamma_i')}^2 \sim \|g\|_{L^2(3I_i)}^2$ and the collection $\{3I_i\}_{i \in \Lambda}$ has bounded overlap, summing the above inequality over $i$ gives a bound on the order of $\|f\|^2_{L^2(\mgs)}$.

Let's show that \eqref{regatsmall} holds. Let $\eta_i \colon \R^{n+1} \to [0,1]$ be a smooth cut-off with support in $4I_i$ such that $\eta_i \equiv 1$ on $3I_i$ and such that
\begin{align}
    \|\partial_t \eta_i\|_\infty \lesssim \ell(I_i)^{-2} \quad \mbox{ and } \quad \|\nabla^k \eta_i\|_\infty \lesssim \ell(I_i)^{-k}
\end{align}
for all $i$ and $k \in \{1,2\}$. Here, $\nabla^2$ is understood to be the Hessian of spatial derivatives. We localize $\psi'$ near the cube $I_i$ by setting
$$\psi'_i (x,t) := \eta_i(x,t)\left(\psi'(x,t) -  \psi'(x_{I_i},t_{I_i})\right),$$
where $(x_{I_i},t_{I_i})$ denotes the centre of $I_i$. Using the properties of $\eta_i$ from above and the fact that $\psi'$ is Lipschitz, it is clear that 
\begin{align}\label{e:properties-psi'}
    \supp(\psi_i') \subseteq 4I_i, \quad 0 \leq \psi_i' \lesssim \ell(I_i) \quad \mbox{ and } \quad   \psi_i' \equiv \psi' - \psi'(x_I,t_I) \mbox{ on $3I_i$}. 
\end{align}
Moreover, using the derivative estimates for $\eta_i$ from above, the derivative estimates for $\psi'$ from Lemma  \ref{l:modification}, and the fact that $\|\psi_i'\|_\infty \lesssim \ell(I_i)$, it is easy to check that 
\begin{align}\label{e:derivatives-psi'}
    \| \partial_t\psi_i' \|_\infty + \| \nabla_x^2 \psi_i' \|_\infty \lesssim \ell(I_i)^{-1}. 
\end{align}

We define the following $\beta$-numbers adapted to $\psi'_i$: 
$$\beta_i((x,t),r) \coloneqq \inf_{a} \bigg(\fint_{B'_{r}(x,t)} \bigg(\frac{|\psi'_i(y,s)-a(y)|}{r}\bigg)^2d y d s\bigg )^\frac{1}{2},$$
where $B'_r(x,t) = \{(y,s) \in \mathbb R^{n} \times \mathbb R: \|x-y\| < r,\, |t-s| <r^2  \}$ and the infimum is taken over all affine functions $a\colon\R^n\to\R^{\adim-\idim}$ which are constant in time. Consider the function $a(y) = \psi'_i(x,t) + \nabla \psi'_i(x,t) \cdot (y-x)$, using the derivative estimate \eqref{e:derivatives-psi'} and applying Taylor's Theorem, it follows that
\begin{align}\label{e:betabig}
\beta_{i}((x,t),r)^2 &\lesssim \frac{r^2}{\ell(I_i)^2} \indent \forall r \lesssim \ell(I_i).
\end{align}
Considering the function $a(y) = \psi_i'(x,t)$, it follows from the inequality $\|\psi_i'\| \lesssim \ell(I_i)$ that 
\begin{align}\label{e:betasmall}
\beta_{i}((x,t),r)^2 &\lesssim \frac{\ell(I_i)^2}{r^2} \indent \forall r \gtrsim \ell(I_i).
\end{align}

We now use the estimates \eqref{e:betabig} and \eqref{e:betasmall} to show that each $\psi'_i$ is a regular Lip-$(1,1/2)$ function.  Choose $R>0$ and $(z,\tau) \in \mathbb R^n\times\mathbb R$.  We will show that
$$\int_0^R \int_{B'_R(z,\tau)} \beta_i((x,t),r)^2 \frac{d x d t d r}{r} \lesssim R^{n+2},$$
where the implicit constant is independent of both $R$ and $(z,\tau)$. The fact that boundedness of such square functions implies that the underlying function is regular Lip-$(1,1/2)$ is well-known (see \cite[Section 2]{hofmann2004caloric} and \cite[Lemma 5.40]{bortz2023corona}).  We split the square function as 
\begin{multline}\label{e:psiibeta}
\int_0^R \int_{B'_R(z,\tau)} \beta_i((x,t),r)^2 \frac{d x d t d r}{r} \\= \int_0^{\tilde c\ell(I_i)} \int_{B'_R(z,\tau)} \beta_i((x,t),r)^2 \frac{d x d t d r}{r}
+ \int_{\tilde c\ell(I_i)}^R \int_{B'_R(z,\tau)} \beta_i((x,t),r)^2 \frac{d x d t d r}{r}.
\end{multline}
For the first term we use \eqref{e:betabig}, and for the second term we use \eqref{e:betasmall}, which yield bounds on the order of $R^{n+2}$.  This shows that each $\psi'_i$ is a regular Lip-$(1,1/2)$ function, with uniform bounds on the expression in \eqref{e:psiibeta}. \\

We now continue with the case where $\epsilon < c\ell(I_i)$.   Fix $(x,t) \in I_i$ and make $c$ small enough such that $(y,s) \in 3I_i$ for all points satisfying $\|(x,t)-(y,s)\|\leq c\ell(I_i)$. Recalling \eqref{e:properties-psi'}, for such a point we have  
\begin{align}\label{e:psi-in-3I}
\psi'_i(y,s) = \psi'(y,s) - \psi'(x_{I_i},t_{I_i}).
\end{align}
We also adopt the notation 
$$\tilde{g}(y,s) = g(y,s) \sqrt{1+|\nabla \psi'(y,s)|^2}.$$
By the above choice of $c$ and \eqref{e:psi-in-3I} (since all $(y,s)$ appearing in the following integrals are contained in $3I_i$) we have 
\begin{align*}
&\int_{I_i} \left|\int_{\substack{\epsilon \leq \|(x,t)-(y,s) \| \leq c\ell(I_i)}} K(x-y,\psi'(x,t)-\psi'(y,s),t-s) \tilde g(y,s)dyds
\right|^2 dxdt\\
=&\int_{I_i} \left| \int_{ \epsilon \leq \|(x,t)-(y,s) \| \leq c\ell(I_i)}  K(x-y,\psi'_i(x,t)-\psi'_i(y,s),t-s) \tilde g(y,s) dy ds \right|^2dxdt \\
\leq &\int_{I_i} \left| \int_{ \epsilon \leq \|(x,t)-(y,s) \| \leq c\ell(I_i)}  K(x-y,\psi'_i(x,t)-\psi'_i(y,s),t-s)\big(\mathds{1}_{3I_i}(y,s)\tilde g(y,s)\big)dy \right|^2dxdt.
\end{align*}
By Theorem 5 of \cite{bortz2021sios}, the final expression above is bounded by $C\|g\|^2_{L^2(3I_i)}$. This completes the proof of \eqref{regatsmall} and, hence, the proof of the lemma. 
\end{proof}

Having proved Lemma \ref{l:interiorbound}, we can finally show that $\mathcal R_\sigma:L^2(\sigma) \rightarrow L^2(\sigma)$.  Given $f \in L^2(\sigma)$, we want to estimate 
\begin{align*}
&\int\left|\int K((X,t)-(Y,s))f(Y,s) d\sigma(Y,s) \right|^2 d\sigma(X,t) \\\lesssim 
&\int \left|\int K((X,t)-(Y,s))f(Y,s) d(\sigma-\sigma_f)(Y,s) \right|^2 d(\sigma-\sigma_f)(X,t)\\
+&\int \left|\int K((X,t)-(Y,s))f(Y,s) d\sigma_f(Y,s) \right|^2 d(\sigma-\sigma_f)(X,t)\\
+&\int \left|\int K((X,t)-(Y,s))f(Y,s) d(\sigma-\sigma_f)(Y,s) \right|^2 d\sigma_f(X,t)\\
+&\int \left|\int K((X,t)-(Y,s))f(Y,s) d\sigma_f(Y,s) \right|^2 d\sigma_f(X,t)\\ 
=& I +II +III+IV.
\end{align*}
By Lemma \ref{l:interiorbound}, $IV \lesssim \|f\|_{L^2(\sigma)}^2$. To deal with $I$, we use the fact that $\Gamma'\setminus\Gamma'_f\subseteq\Gamma_H$, where $\Gamma_H$ is the graph of $\theta^{15/16}H$. The function $H$ is regular Lip(1,1/2), so $\Gamma_H$ is a regular graph.  Hence, if $\sigma_H$ denotes the parabolic Hausdorff measure restricted to $\Gamma_H$,  Theorem 5 of \cite{bortz2021sios} implies $\mathcal R_{\sigma_H}: L^2(\sigma_H) \rightarrow L^2(\sigma_H)$. Since $\sigma-\sigma_f=\sigma_H|_{\Gamma'\setminus\Gamma'_f}$, we extend $\mathds{1}_{\Gamma'\setminus\Gamma'_f}f$ by zero to the rest of $\Gamma_H$. Hence,
\begin{align*}
I &\leq \int \left| \int K((X,t)-(Y,s)) \mathds{1}_{\Gamma'\setminus\Gamma'_f}(Y,s)f(Y,s))d\sigma_H(Y,s)\right|^2 d\sigma_H(X,t)\\
 &\leq C(H) \|\mathds{1}_{\Gamma'\setminus\Gamma'_f}f\|_{L^2(\sigma_H)}^2 \leq C(H)\|f\|_{L^2(\sigma)}^2
\end{align*}
In the inequality above, the constant $C(H)$ depends on dimension, $\|H\|_{lip}$, and $\|D^t_{1/2}H\|_*$.  Now, we deal $II$.  we split the integral as follows:
\begin{align*}
II \lesssim& \int \left|\int K((X,t)-(Y,s))f(Y,s) d(\sigma_f-\sigma_s)(Y,s) \right|^2 d(\sigma-\sigma_f)(X,t)\\ 
&+  \int \left|\int K((X,t)-(Y,s))f(Y,s) d\sigma_s(Y,s) \right|^2 d(\sigma-\sigma_f)(X,t)\\
&= II' + II''
\end{align*}
For $II'$, both $\Gamma'\setminus\Gamma'_f$ and $\Gamma'_f\setminus\Gamma'_s$ are subsets of $\Gamma_H$, and $\sigma_f-\sigma_s=\sigma_H|_{\Gamma'_f\setminus\Gamma'_s}$. We extend $\mathds{1}_{\Gamma'_f\setminus\Gamma'_s}f$ by zero to the rest of $\Gamma_H$ and obtain
\begin{align}\label{separationbound}
II' \leq  \int \left|\int K((X,t)-(Y,s))(1_{\Gamma'_f\setminus \Gamma'_s}(Y,s)f(Y,s)) d\sigma_H(Y,s) \right|^2 d\sigma_H(X,t) \lesssim \|f\|_{L^2(\sigma)}^2.
\end{align}
To bound $II''$, we use the separation between $(X,t)$ and $(Y,s)$ in the integral to bound the inner integral by a maximal function.  Then, we clearly can bound the term by $\|f\|^2_{L^2(\sigma)}$.\\

Having obtained our bounds for $II$, we can bound $III$ via a duality argument.  Notice that we can rephrase the bound on term $II$ as $\mathcal R_{\sigma_f}:L^2(\sigma_f) \rightarrow L^2(\sigma - \sigma_{f})$.  Let $\mathcal R^{-}$ be the operator associated to $K(-\cdot)$.  Then all of the arguments in this document apply to $\mathcal R^{-}$ and in particular we have the bound:
\begin{align*}
\mathcal R^{-}_{\sigma_f}:L^2(\sigma_f) \rightarrow L^2(\sigma - \sigma_{f}).
\end{align*}
Choose $f \in L^2(\sigma-\sigma_f)$, $g \in L^2(\Gamma_{\sigma_f})$ with $\|g\|_{ L^2(\sigma_f)}=1$.  We estimate
\begin{align*}
&\left|\int\int K_\epsilon((X,t)-(Z,\tau))f(Z,\tau) d(\sigma-\sigma_f)(Z,\tau)g(X,t)
d\sigma_f(X,t) \right|\\
\leq&  \int\left|\int K_\epsilon((X,t)-(Z,\tau))g(X,t)d\sigma_f(X,t)
\right| |f(Z,\tau)| d(\sigma-\sigma_f)(Z,\tau) \\
\leq&\|\mathcal R^{-}_{\sigma_f}g\|_{L^2(\sigma-\sigma_f)} \|f\|_{L^2(\sigma - \sigma_f)}
\lesssim \|f\|_{L^2(\sigma-\sigma_f)}.
\end{align*}
Which shows that
$$
\mathcal R_{\sigma-\sigma_f}:L^2(\sigma-\sigma_f) \rightarrow L^2(\sigma_f).$$
From this we easily see that term $III$ is bounded by $\|f\|^2_{ L^2(\sigma_f)}$.  

\frenchspacing
\bibliography{Ref}
\bibliographystyle{alpha}

\end{document}